\documentclass[12pt, reqno]{amsart}
\usepackage{cooltooltips}
\usepackage[all,arc,cmtip]{xy}
\usepackage{enumerate}
\usepackage{mathrsfs}
\usepackage{euscript}
\usepackage{hyperref}
\usepackage{amsthm}
\usepackage{amsmath}
\usepackage{MnSymbol}
\usepackage{comment} 
\usepackage{tikz-cd}
\usepackage{xspace}

\tikzset{
labl1/.style={anchor=north, rotate=90, inner sep=1.2mm}
}

\usepackage[belowskip=-15pt,aboveskip=0pt]{caption}

   \usepackage{bm}

\usepackage{esvect}

\usepackage[latin1]{inputenc}

\usepackage[T4]{fontenc}
\usepackage{textcomp}

\usepackage{usnomencl}

\usepackage{graphicx}
\usepackage{color}
\usepackage{eso-pic}

\usepackage{mathtools}
\usepackage{multicol}
\usepackage{paralist}
\usepackage{geometry}
\usepackage{algorithmic}
\usepackage[ruled,vlined]{algorithm2e}
\usepackage{comment}
\usepackage{tikz-cd}
\usepackage{url}
\usepackage{graphicx,scalerel}
\usetikzlibrary{decorations.pathmorphing}
\usepackage{float}
\usepackage{tabulary}
\usepackage{booktabs}
\newtheorem{theorem}{Theorem}[section]

\newtheorem{lemma}[theorem]{Lemma}
\newtheorem{proposition}[theorem]{Proposition}
\newtheorem{definition}[theorem]{Definition}
\newtheorem{remark}[theorem]{Remark}
\newtheorem{example}[theorem]{Example}

\newcommand{\iuv}[2]{\mathfrak{i}_{\scalebox{0.8}{$\scriptstyle #1$}\mkern-0.7mu\scalebox{0.8}{$\scriptstyle , $}\mkern-1mu\scalebox{0.8}{$\scriptstyle #2$}}}

\newcommand{\topo}{\mathbf{Top}}
\newcommand{\Uij}[2]{U_{\substack{\text{\scalebox{0.7}{$\mkern-1mu #1 \! \wedge\! #2$}}}}}

\newcommand{\Uijk}[3]{U_{\substack{\text{\scalebox{0.7}{$\mkern-1mu #1 \! \wedge\! #2 \!\wedge\! #3$}}}}}

\newcommand{\subsm}[2]{{#1}_{\substack{\text{\scalebox{0.8}{$\mkern-0.8mu #2$}}}}}

\newcommand{\op}{{\substack{\text{\scalebox{0.7}{$\operatorname{op}$}}}}}

\newcommand{\subop}{{\subseteq}_{\operatorname{op}}}
\newcommand{\Rel}[1]{\mathcal{R}_{\substack{\text{\scalebox{0.8}{$#1$}}}}}
\newcommand{\Co}[1]{{#1}_{\substack{\text{\scalebox{0.7}{$\mkern-1.5mu 0$}}}}}
\newcommand{\Cn}[1]{{#1}_{\substack{\text{\scalebox{0.7}{$ \mkern-1.5mu 1$}}}}}

\newcommand{\Cns}[1]{{#1}_{{}_{\substack{\text{\scalebox{0.7}{$ \mkern-1.5mu 1$}}}}}}
\newcommand{\mbf}[1]{\mathbf{#1}}
\newcommand{\etaij}[2]{\eta_{\substack{\text{\scalebox{0.7}{${\mkern-1mu #1 ,\![\mkern-1mu #1\mkern-1mu ,\mkern-1mu #2\mkern-1mu]}$}}}}}
\newcommand{\etaji}[2]{\eta_{\substack{\text{\scalebox{0.7}{${ \mkern-1mu  #2 ,\![\mkern-1mu #2\mkern-1mu ,\mkern-1mu #1\mkern-1mu ]}$}}}}}
\newcommand{\etaaji}[2]{\eta_{\substack{\text{\scalebox{0.7}{${ \mkern-1mu  #2 ,\![\mkern-1mu #1\mkern-1mu ,\mkern-1mu #2\mkern-1mu ]}$}}}}}
\newcommand{\tauij}[2]{\tau_{\substack{\text{\scalebox{0.7}{$\mkern-1.2mu [\mkern-1mu #1\mkern-1mu,\mkern-1mu #2\mkern-1mu]$}}}}}
\newcommand{\tauji}[2]{\tau_{\substack{\text{\scalebox{0.7}{$\mkern-1.2mu [\mkern-1mu #2 \mkern-1mu,\mkern-1mu #1\mkern-1mu]$}}}}}
\newcommand{\etaijk}[4]{\eta_{\substack{\text{\scalebox{0.7}{${[\mkern-1mu #2 \mkern-1mu,\mkern-1mu #1\mkern-1mu ],\![\mkern-1mu #2 \mkern-1mu ,\mkern-1mu #3 \mkern-1mu,\mkern-1mu #4\mkern-1mu]}$}}}}}
\newcommand{\etaiijk}[3]{\eta_{\substack{\text{\scalebox{0.7}{${\mkern-1mu #1 \mkern-1mu,\![\mkern-1mu #1 \mkern-1mu ,\mkern-1mu #2 \mkern-1mu,\mkern-1mu #3\mkern-1mu]}$}}}}}
\newcommand{\tauijk}[4]{\tau_{{\substack{\text{\scalebox{0.7}{$\! #1$}}}}_{\substack{\text{\scalebox{0.8}{$[\mkern-1mu #2 \mkern-1mu,\mkern-1mu #3\mkern-1mu]$}}}}}}
\newcommand{\limi}[1]{\operatorname{lim}\!#1}
\newcommand{\Goi}[2]{\Co{#1} \mkern-1.7mu (\mkern-1mu #2 \mkern-1mu)}
\newcommand{\Goij}[3]{\Co{#1} \mkern-1.7mu (\mkern-1mu #2 \mkern-1mu ,\mkern-1mu #3 \mkern-1mu )}
\newcommand{\Goijk}[4]{\Co{#1} \mkern-1.7mu (\mkern-1mu #2 \mkern-1mu ,\mkern-1mu #3 \mkern-1mu ,\mkern-1mu #4 \mkern-1mu )}

\newcommand{\Coi}[2]{#1 \mkern-1.7mu (\mkern-1mu #2 \mkern-1mu)}
\newcommand{\Coij}[3]{#1\mkern-1.7mu (\mkern-1mu #2 \mkern-1mu ,\mkern-1mu #3 \mkern-1mu )}
\newcommand{\Coijk}[4]{#1 \mkern-1.7mu (\mkern-1mu #2 \mkern-1mu ,\mkern-1mu #3 \mkern-1mu ,\mkern-1mu #4 \mkern-1mu )}
\newcommand{\glI}[1]{\mathbf{Gl}(\mathrm{#1})}
\newcommand{\Gnij}[2]{\Cn{#1} \mkern-1.7mu (\mkern-1.2mu #2  \mkern-1.2mu )}
\newcommand{\Cnij}[2]{#1 \! ( \mkern-1.2mu #2  \mkern-1.2mu )}
\newcommand{\subsc}[2]{{}_{\substack{\text{\scalebox{0.6}{$\mkern-1.8mu #1$}}}} \mkern-1.8mu #2}
\newcommand{\randi}[2]{#1_{\substack{\text{\scalebox{0.7}{$\mkern-1mu \mkern-1.2mu #2\mkern-1.2mu$}}}}}
\newcommand{\randij}[3]{#1_{\substack{\text{\scalebox{0.7}{$\mkern-1mu [\mkern-1.2mu #2\mkern-1.2mu, \mkern-1.2mu #3 \mkern-1.2mu]$}}}}}
\newcommand{\randijk}[4]{#1_{\substack{\text{\scalebox{0.7}{$\mkern-1mu [\mkern-1.2mu #2\mkern-1.2mu, \mkern-1.2mu #3 \mkern-1.2mu,\mkern-1.2mu #4 \mkern-1.2mu]$}}}}}
\newcommand{\randji}[3]{#1_{\substack{\text{\scalebox{0.7}{$\mkern-1mu [\mkern-1.2mu #2\mkern-1.2mu, \mkern-1.2mu #3 \mkern-1.2mu]$}}}}}
\newcommand{\idij}[2]{{\operatorname{id}}_{\substack{\text{\scalebox{0.7}{$\mkern-1mu [\mkern-1.2mu #1\mkern-1.2mu, \mkern-1.2mu #2 \mkern-1.2mu]$}}}}}
\newcommand{\idijk}[3]{{\operatorname{id}}_{\substack{\text{\scalebox{0.7}{$\mkern-1mu [\mkern-1.2mu #1\mkern-1.2mu, \mkern-1.2mu #2 \mkern-1.2mu, \mkern-1.2mu #3 \mkern-1.2mu]$}}}}}
\newcommand{\dindi}[3]{{#1}_{{\substack{\text{\scalebox{0.8}{$\mkern-1mu #2$}}}}_{{\mkern-1mu}_{\substack{\text{\scalebox{0.7}{$\mkern-1mu #3$}}}}}}}
\newcommand{\dindij}[4]{{#1}_{{\substack{\text{\scalebox{0.8}{$\mkern-1mu #2$}}}}_{{\mkern-1mu}_{\substack{\text{\scalebox{0.7}{$\mkern-1mu [\mkern-1.2mu #3\mkern-1.2mu, \mkern-1.2mu #4 \mkern-1.2mu]$}}}}}}}

\newcommand{\dindijk}[5]{{#1}_{{\substack{\text{\scalebox{0.8}{$\mkern-1mu #2$}}}}_{{\mkern-1mu}_{\substack{\text{\scalebox{0.7}{$\mkern-1mu [\mkern-1.2mu #3\mkern-1.2mu, \mkern-1.2mu #4 \mkern-1.2mu, \mkern-1.2mu #5 \mkern-1.2mu]$}}}}}}}

\newcommand{\dindiv}[4]{{#1}_{{\substack{\text{\scalebox{0.8}{$\mkern-1mu #2$}}}}_{{\mkern-1mu}_{\substack{\text{\scalebox{0.7}{$\mkern-1mu #3$}}}_{{}_{\substack{\text{\scalebox{0.6}{$\mkern-1mu #4$}}}}}}}}}

\newcommand{\dindgij}[5]{{#1}_{{\substack{\text{\scalebox{0.8}{$\mkern-1mu #2$}}}_{{\mkern-1mu}_{{\substack{\text{\scalebox{0.6}{$\mkern-1mu #3$}}}_{{}_{\substack{\text{\scalebox{0.6}{$\mkern-1mu [\mkern-1.2mu #4\mkern-1.2mu, \mkern-1.2mu #5 \mkern-1.2mu]$}}}}}}}}}}}

\newcommand{\dindgijk}[6]{{#1}_{{{\substack{\text{\scalebox{0.8}{$\mkern-1mu #2$}}}_{{\mkern-1mu}_{{\substack{\text{\scalebox{0.6}{$\mkern-1mu #3$}}}_{{}_{\substack{\text{\scalebox{0.6}{$\mkern-1mu [\mkern-1.2mu #4\mkern-1.2mu, \mkern-1.2mu #5 \mkern-1.2mu, \mkern-1.2mu #6 \mkern-1.2mu]$}}}}}}}}}}}}

\newcommand{\dindpi}[3]{{#1}_{{\substack{\text{\scalebox{0.8}{$\mkern-1mu #2$}}}}_{{{\mkern-1mu}_{\substack{\text{\scalebox{0.7}{$\mkern-1mu #3$}}}}}}'}}
\newcommand{\dindpij}[4]{{#1}_{{\substack{\text{\scalebox{0.8}{$\mkern-1mu #2$}}}}_{{\mkern-1mu}_{\substack{\text{\scalebox{0.7}{$\mkern-1mu [\mkern-1.2mu #3\mkern-1.2mu, \mkern-1.2mu #4 \mkern-1.2mu]$}}}}}'}}
\newcommand{\dindpijk}[5]{{\mkern-1mu #1}_{{\substack{\text{\scalebox{0.8}{$\mkern-1mu #2$}}}}_{{\mkern-1mu}_{\substack{\text{\scalebox{0.7}{$\mkern-1mu [\mkern-1.2mu #3\mkern-1.2mu, \mkern-1.2mu #4 \mkern-1.2mu, \mkern-1.2mu #5 \mkern-1.2mu]$}}}}}'}}

\newcommand{\dindsi}[3]{{#1}_{{\substack{\text{\scalebox{0.8}{$\mkern-1mu #2$}}}}_{{}_{\substack{\text{\scalebox{0.7}{$#3$}}}}}}}

\newcommand{\gcov}[1]{\bm{\mathcal C}_{\substack{\text{\scalebox{0.8}{$#1$}}}}}

\newcommand{\circc}{\! \circ  \!}

\newcommand{\suchthat}{\;\ifnum\currentgrouptype=16 \middle\fi|\;}

\newcommand\restr[2]{{
\left.\kern-\nulldelimiterspace 
#1 
\vphantom{\big|} 
\right|_{#2} 
}}

\makeatletter
\newcommand*{\doublerightarrow}[2]{\mathrel{
		\settowidth{\@tempdima}{$\scriptstyle#1$}
		\settowidth{\@tempdimb}{$\scriptstyle#2$}
		\ifdim\@tempdimb>\@tempdima \@tempdima=\@tempdimb\fi
		\mathop{\vcenter{
				\offinterlineskip\ialign{\hbox to\dimexpr\@tempdima+1em{##}\cr
					\rightarrowfill\cr\noalign{\kern.5ex}
					\rightarrowfill\cr}}}\limits^{\!#1}_{\!#2}}}

\title{\bfseries A Categorical Perspective on Gluing}
\author{Sophie Marques and Damas Mgani}

\begin{document}

\begin{abstract}

This paper introduces the concept of gluing in a general category, enabling us to define categories that admit glued-up objects. To achieve this, we introduce the notion of a gluing index category. Subsequently, we provide an entirely abstract definition of a gluing data functor requiring only the given category to admit pushouts. We explore various characterizations of cones and limits over these functors. We introduce the concept of refined gluing, which in turn enables us to combine different gluing data effectively. Furthermore, we demonstrate that several categories of topological spaces admit glued-up objects. This, in turn, allows us to establish a concept of gluing covering and to prove that the collection of those coverings forms a Grothendieck topology.

\noindent \textbf{Keywords:} Gluing, topological spaces, refinement, category theory, functor, limit, Grothendieck topology, covering.

\noindent \textbf{2020 Math. Subject Class:} 14A15, 18F20, 18F60.

\noindent \begin{center}
\rm e-mail: \href{mailto:d smarques@sun.ac.za}{ smarques@sun.ac.za}

\it
Department of Mathematical Sciences, 
University of Stellenbosch, \\
Stellenbosch, 7600, 
South Africa\\
\&
NITheCS (National Institute for Theoretical and Computational Sciences), \\
South Africa \\ \bigskip

\rm e-mail: \href{mailto:d.mgani99@gmail.com}{d.mgani99@gmail.com}

\it
Department of Mathematical Sciences, 
University of Stellenbosch, \\
Stellenbosch, 7600, 
South Africa

\end{center} 

\end{abstract}

\setcounter{tocdepth}{3}
\maketitle
  \tableofcontents

\newpage

\begin{center} 
  \begin{tabular}{l p{11cm} }
  \textbf{CATEGORIES} & \\[0.5cm]
$ \mathbf{C}$& Category with the following components \\& $( \Co{\mbf{C}} ,\Cn{\mbf{C}}, \mathbf{d}_{\substack{\text{\scalebox{0.7}{$\mathbf{C}$}}}},\mathbf{c}_{\substack{\text{\scalebox{0.7}{$\mathbf{C}$}}}},\mathbf{e}_{\substack{\text{\scalebox{0.7}{$\mathbf{C}$}}}},\mathbf{m}_{\substack{\text{\scalebox{0.7}{$\mathbf{C}$}}}})$ where $ \Co{\mbf{C}}$ is a set whose elements are called objects in $\mathbf{C}$, $ \Cn{\mbf{C}}$ is a set whose elements are called morphisms in $\mathbf{C}$, $\mathbf{d}_{\substack{\text{\scalebox{0.7}{$\mathbf{C}$}}}}$ and $\mathbf{c}_{\substack{\text{\scalebox{0.7}{$\mathbf{C}$}}}}$ are maps from $ \Cn{\mbf{C}}$ to $ \Co{\mbf{C}}$, called domain and codomain respectively,  $\mathbf{e}_{\substack{\text{\scalebox{0.7}{$\mathbf{C}$}}}}$ is a map from $ \Co{\mbf{C}}$ to $ \Cn{\mbf{C}}$, called the identity, $\mathbf{m}_{\substack{\text{\scalebox{0.7}{$\mathbf{C}$}}}}$ is a map from the set\\& 
	 $\{(f,g)\in  \Cn{\mbf{C}}\times  \Cn{\mbf{C}}\;|\;\mathbf{d}_{\substack{\text{\scalebox{0.7}{$\mathbf{C}$}}}}(f)=\mathbf{c}_{\substack{\text{\scalebox{0.7}{$\mathbf{C}$}}}}(g)\}$ to $ \Cn{\mbf{C}}$, called composition. Categories are defined in this way in \cite{Joji}\\

  \textbf{SETS} & \\[0.5cm]

$U\times_V W$ & An arbitrarily chosen pullback of two given map $\phi: U \rightarrow V$ and $\phi: W \rightarrow V$ in some category $\mathbf{C}$, when it is not specified otherwise.\\
	
$\limi{\mathbf{G}}$ & An arbitrarily chosen terminal cone via the axiom of choice in the category of cones over some diagram $\mathbf{G}$ \\
	
$(\subsm{Q}{\mathbf{G}}, \subsm{\iota}{\subsm{Q}{\mathbf{G}}})$ & Standard representative of the limit of $\mathbf{G}$ where $\mathbf{G}$ is an $\topo^{\op}$-gluing data functor and $\subsm{Q}{\mathbf{G}}$ is a glued-up $\topo^{\op}$-object along $\mathbf{G}$ through $\subsm{\iota}{\subsm{Q}{\mathbf{G}}}^{\op}$ (see Definition \ref{deflema}) \\
	
\end{tabular}
 \end{center}

\begin{center} 
  \begin{tabular}{l p{11cm} }
 \textbf{MORPHISMS} & \\[0.5cm]	
$V\subsm{\subseteq}{\operatorname{op}}U$ & $V$ open subset of $U$\\

$\iuv{V}{U}$ &  Inclusion map from open set $V$ to $U$ of a topological space $X$ \\

$\iuv{Z\subsm{\times}{U} W}{Z}$ &  Canonical projection map from $Z \subsm{\times}{U} W$ to $Z$, once a pullback has been chosen in its respective category of two given maps $\phi: W \rightarrow U$ and $\phi: Z \rightarrow U$ in some category $\mathbf{C}$  \\

$\mathcal{F}$ & Functor from $\mbf{C}$ to $\mbf{D}$ where $\mbf{C}$ and $\mbf{D}$ are categories. That is, a pair $(\Co{\mathcal{F}},\Cn{\mathcal{F}})$ where $\Co{\mathcal{F}}$ is the map from $\Co{\mbf{C}}$ to $\Co{\mathbf{D}}$ and $\Cn{\mathcal{F}}$ is the map from $\Cn{\mbf{C}}$ to $\Cn{\mbf{D}}$ such that $\Cn{\mathcal{F}}$$ ( \subsm{\operatorname{id}}{A}) =\subsm{\operatorname{id}}{\Co{\mathcal{F}}(A)}$ and $\Cn{\mathcal{F}} ( f \circ g )=\Cn{\mathcal{F}} (f)\circ \Cn{\mathcal{F}} ( g)$ where $A\in \Co{\mbf{C}}$ and $f,g $ in $ \Cn{\mbf{C}}$ \\

$f^{\op}$ & The morphism corresponding to a morphism $f$ of $\mathbf{C}$ in the opposite category $\mathbf{C}^{\op}$ \\

	\end{tabular}
 \end{center}

\section*{Introduction}

The concept of gluing plays a pivotal role in a multitude of mathematical disciplines, spanning from the realm of topology to the intricate theories of sheaf theory and schemes. Its fundamental nature lies in its capacity to serve as a unifying tool, allowing mathematicians to seamlessly integrate disparate components and gain deeper insights into the intricate tapestry of mathematical structures.

Henri Poincar\'e, a luminary in the world of mathematics, made significant contributions to the development of the gluing concept during the mid-20th century. His pioneering work, as meticulously chronicled in \cite{poincare2010}, laid the essential groundwork for understanding and harnessing the power of gluing in mathematical contexts. Poincar\'es insights and innovations have since paved the way for numerous advancements and applications of the gluing principle across diverse mathematical landscapes  (\cite[\S 2]{curry} and \cite{grothendieck1971}).

One of our primary contributions is the categorical description of the gluing process. Our work achieves the following key objectives:

\begin{enumerate} 
\item We describe glued-up objects as the limit of a functor that we refer to as the "gluing data functor" (see Definition \ref{limext} (1)) (confirming the uniqueness up to isomorphism).

\item We establish a comprehensive characterization of gluing data in a completely abstract manner, introducing the innovative concept of a "gluing index category" (see Definition \ref{gic}). This abstraction allows for the definition of gluing data as functors from the gluing index category that preserve pushouts, providing a flexible and versatile framework for a wide range of applications.

\item This allows us to inquire whether certain given categories admit glued-up objects and gain a deeper understanding of the nature of these categories (see Remark \ref{admia}).

\item We introduce the concept of a refinement of a gluing data functor (see Definition \ref{refino}), enabling the seamless merging (or composing) of gluing functors by considering gluing data functors within the category of gluing data functors  (see Definition \ref{gdff} and Proposition \ref{refinery}).

\item In different categories of topological spaces, we provide rigorous proofs that any gluing data functor admits a limit. Moreover, we demonstrate that the coverings induced by gluing data functors form a Grothendieck topology within this context, reinforcing the foundational importance of gluing processes (see Theorem \ref{sitee}).

\end{enumerate}

Through these contributions, we provide a rigorous and categorical foundation for understanding and utilizing the gluing process in diverse mathematical categories.

In the first section of our paper, we introduce the concept of Gluing Index Categories (see Definition \ref{gic}) and gluing data functors (see Definition \ref{limext} (1)), highlighting some of their essential properties (see Remark \ref{reglue} and \ref{rem1}) and providing intuitive insights into their construction. These definitions lay the groundwork for our investigation into categories that admit glued-up objects, allowing us to better understand their nature. We also develop a simplified framework for characterizing cones (see Lemma \ref{cocone}) and limits over gluing data functors (see Theorem \ref{preglued}), offering practical advantages.

Additionally, we introduce the notion of refining gluing functors (see Definition \ref{refino}), enabling us to combine gluing data functors over the category of gluing data functors whose morphisms are refinement maps (see  Definition \ref{gdff} and Example \ref{torusglu}).

In the second section, we provide a complete and detailed description of the glued-up object associated with a gluing data functor in the context of topological spaces (see Theorem \ref{gluingtop}). We establish that a glued-up object always exists for gluing data functors over topological spaces, extending our result to the category of topological spaces whose morphisms are open maps (see Lemma \ref{otop}). It is worth noting that the gluing data we work with in this paper is a generalization of the gluing data typically considered in the literature. We also offer a detailed and straightforward example that provides an intuitive application of our definition and illustrates how the composition of gluing data functors works in practice (see Example \ref{torusglu}). We present a concise topological interpretation of the glued-up object, closely resembling what we refer to as a "covering" (see Lemma \ref{coverings}). This definition of a gluing covering allows us to prove that the set of such coverings forms a Grothendieck topology over the category of topological spaces (see Theorem \ref{sitee}), including its restricted subcategory where maps consist of open continuous maps. Thus, we establish a relationship between coverings and gluing data functors. 
Our work is complemented by visual figures that provide readers with a concrete understanding of the concepts presented.

\section{Gluing objects categorically} 

In this section, we introduce the Gluing Index Categories and the gluing data functors, novel constructions that can serve as tools for understanding and formalizing gluing operations across various mathematical contexts. Encapsulating gluing data within the gluing data functor framework reveals that achieving a limit over this gluing data functor, which represents the gluing data, directly corresponds to the concept of a glued-up object. 

This newfound clarity in understanding gluing data emerges as we establish subsequent lemmas while developing each definition. These lemmas gradually refine our comprehension, enabling us to encapsulate the core essence of gluing objects in a category in a more abstract and unified manner. 

The insights obtained from these constructions will find practical application in the subsequent section, particularly within the context of topological spaces. We refer to  \cite{Joji}, \cite{adamek2009abstract} and \cite{maclane} for categorical background.

\subsection{Gluing Index Categories and Gluing Data Functors}

In the following definition, we present the Gluing Index Category, denoted as $\glI{I}$. This category encapsulates fundamental components required to establish a flexible gluing framework applicable to diverse categories. This conceptualization is inspired by the notion of topological gluing data. Our aim is to comprehensively describe gluing data in a categorical manner. Our specific objective is to define a functor from the Gluing Index Category with certain categorical properties. This functor will be constructed in such a way that limits consistently exist over these functors in the categories where gluing is usually used. Furthermore, this definition of a functor will be precisely equivalent to providing a gluing data within the conventional categories where gluing is defined.

In order to gain some intuition, one can think 
\begin{itemize} 
\item each individual index corresponds to an open set, 
\item pairs of indices represent opens contained within the single index open that serves as the patching for the gluing process, and 
\item triples of indices denote the intersection of these double index opens. 
\end{itemize} 
The morphisms within $\glI{I}$ align with the respective inclusion maps, and their uniqueness guarantees the satisfaction of essential gluing data relations. Moreover, the equivalence relation applied to the indices within $\glI{I}$ reflects analogous relations encountered in the realm of topological gluing data. Further insight into the more intricate and technical properties of these gluing index categories can be gleaned by carefully examining the retrieval of the gluing data through the framework of diagrams.

\begin{definition}\label{gic}
	Let ${\mathrm{I}}$ be a set. We define the \textbf{\textit{gluing index category of type}} $\mathrm{I}$ denoted $\glI{I}$ as follows:
	
	\begin{enumerate}
		\item \textbf{Objects}: The equivalence classes of elements in $\mathrm{I} \cup \mathrm{I}^2\cup \mathrm{I}^3$ with respect to an equivalence relation $\Rel{\mathrm{I}}$ where $\Rel{\mathrm{I}}$ is the equivalence relation, generated by the relations $i \Rel{\mathrm{I}} (i,i)$, $(i,j, k)\Rel{\mathrm{I}} (i, k, j)$ and $(i,j, j) \Rel{\mathrm{I}} (i, j)$. We denote the equivalence class of the element $i$ as $[i]$, the equivalence class of the element $(i,j,k)$ as $[i,j,k] $ and the equivalence class of the element $(i,j)$ as $[i,j] $, for all $i,j,k \in \mathrm{I}$.
		\item \textbf{Morphisms}: Morphisms in $\glI{I}$ are structured as follows:
		
		\begin{enumerate}
			\item For each $a,b\in \Co{\glI{I}}$, a morphism from $a$ to $b$ is unique when it exists.
			
			\item Each object $a \in \Co{\glI{I}}$ has an associated identity map.
			
			\item Two morphisms in $\glI{I}$, $f: a \rightarrow b$ and $g : c \rightarrow d$, are composable when $b=c$.
			
			\item For each $(i,j) \in \mathrm{I}^2$, we have:
			
			\begin{enumerate}
				\item A unique morphism from $i$ to $[i,j]$, denoted $\etaij{i}{j}$.
				
				\item A unique morphism from $[j,i]$ to $[i,j]$, denoted $\tauij{i}{j}$.
			\end{enumerate}
			
			\item For each $(i,j,k)\in \mathrm{I}^3$ and $n \in \{j,k\}$, we have:
			
			\begin{enumerate}
				\item A unique morphism from $[i,n]$ to $[i,j,k]$, denoted $\etaijk{n}{i}{j}{k}$.
				
				\item A unique morphism from $[j,i,k]$ to $[i,j,k]$, denoted $\tauijk{k}{i}{j}{k}$.
			\end{enumerate}
		\end{enumerate}
	\end{enumerate}
\end{definition}

In light of the preceding definition, we encounter fundamental properties of the Gluing Index Category, $\glI{I}$, which plays a pivotal role in understanding its structure. These properties can be succinctly summarized by observing that for any two objects $a$ and $b$ within $\glI{I}$, the presence of a unique morphism from $a$ to $b$ is a hallmark characteristic. This uniqueness bestows several essential properties upon $\glI{I}$, illuminating its underlying structure. Within this context, we uncover a series of intricate relationships that link various morphisms and objects in $\glI{I}$, revealing the rich interplay between its elements. These relationships are encapsulated in the ensuing Remark \ref{reglue}, which offers a deeper insight into the consequences of uniqueness within $\Cn{\glI{I}}$. 

\begin{remark}\label{reglue}
	The construction of $\glI{I}$ ensures that for any two objects $a,b \in \Co{\glI{I}}$, a morphism from $a$ to $b$ is unique. This uniqueness leads to several essential properties:
	
	\begin{enumerate}
		\item For all $i,j,k \in \mathrm{I}$:
		\begin{enumerate}
		\item $\etaij{i}{i}=\tauij{i}{i}= \subsm{\operatorname{id}}{i}$;
			\item $\tauij{i}{j}\circc \tauji{i}{j} = \idij{i}{j}$ and $\tauji{i}{j}\circc \tauij{i}{j} =  \idij{j}{i}$;
			\item $\tauijk{k}{i}{j}{k} \circc \tauijk{i}{j}{k}{i} = \tauijk{j}{i}{k}{j}$ and $\tauijk{k}{i}{j}{k} \circc \tauijk{k}{j}{i}{k} =\idijk{i}{j}{k}$;
			\item $\etaijk{j}{i}{j}{k} \circc\etaij{i}{j}= \etaijk{k}{i}{j}{k} \circc \etaij{i}{k}$;
			\item $\tauijk{k}{i}{j}{k} \circc \etaijk{i}{j}{i}{k}=  \etaijk{j}{i}{j}{k} \circc \tauij{i}{j}$.
		\end{enumerate}
		
		\item For all $i,j,k \in \mathrm{I}$, the pushout $[i,j] \subsm{\sqcup}{i} [i,k]$ with respect to $\etaij{i}{j}$ and $\etaij{i}{k}$ exists, and we have $[i,j,k] = [i,j] \subsm{\sqcup}{i} [i,k]$. This follows from the equality $\etaijk{k}{i}{j}{k} \circc\etaij{i}{j}= \etaijk{k}{i}{j}{k} \circc \etaij{i}{k}$ and the uniqueness of morphisms.
	\end{enumerate}
\end{remark}

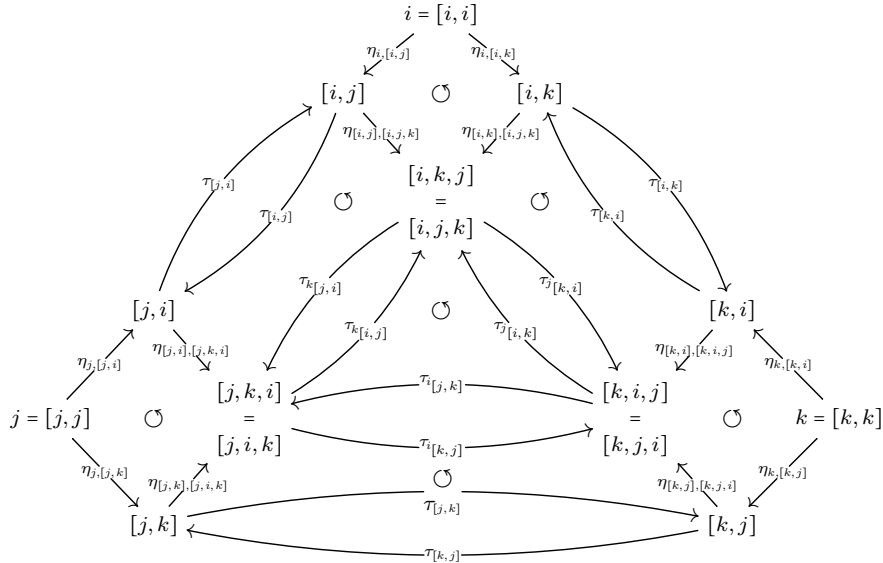
\begin{figure}[H]
$$ {\tiny\xymatrixrowsep{0.2in}
\xymatrixcolsep{0.1in} \xymatrix{ 
&&&& i=[i,i] \ar[rd]|-{\etaij{i}{k}} \ar[ld]|-{\etaij{i}{j}}&&&& \\
&&&[i,j]  \ar@/^1pc/[ddll]|-{\tauji{j}{i}} \ar[rd]|-{\tiny{\etaijk{j}{i}{j}{k}}}  & \circlearrowleft  & [i,k]  \ar@/^1pc/[ddrr]|-{\tauji{k}{i}} \ar[ld]|-{\tiny{\etaijk{k}{i}{j}{k}}}&&& \\
&&&\circlearrowleft& {\begin{array}{@{}c@{}} \left[ i,k,j \right] \\ = \\ \left[ i,j, k \right]  \end{array}} \ar@/^1pc/[rrdd]|-{\tauijk{j}{k}{i}{j}} \ar@/_1pc/[lldd]|-{\tauijk{k}{j}{i}{j}}  & \circlearrowleft & &&\\  
& [j,i]\ar@/^1pc/[uurr]|-{\tauji{i}{j}} \ar[rd]|-{\tiny{\etaijk{i}{j}{k}{i}}}&&& \circlearrowleft  &&&[k,i] \ar@/^1pc/[uull]|-{\tauji{i}{k}} \ar[ld]|-{\tiny{\etaijk{i}{k}{i}{j}}}& \\
  j=[j,j] \ar[ru]|-{\etaij{j}{i}}\ar[rd]|-{\etaij{j}{k}}&\circlearrowleft  &\ar@/_1pc/[rruu]|-{\tauijk{k}{i}{j}{k}} {\begin{array}{@{}c@{}} \left[ j,k,i \right] \\ = \\ 
  \left[ j,i ,k \right]  \end{array}}   \ar@/_1pc/[rrrr]|-{\tauijk{i}{k}{j}{i}}   &&   &&  {\begin{array}{@{}c@{}} \left[ k,i,j \right] \\ = \\ \left[ k,j ,i \right]  \end{array}} \ar@/^1pc/[lluu]|-{\tauijk{j}{i}{k}{j}}   \ar@/_1pc/[llll]|-{\tauijk{i}{j}{k}{i}}
&\circlearrowleft  & k=[k,k]  \ar[ld]|-{\etaij{k}{j}} \ar[lu]|-{\etaij{k}{i}} \\
&[j,k]\ar[ru]|-{\tiny{\etaijk{k}{j}{i}{k}}} \ar@/^1pc/[rrrrrr]|-{{\begin{array}{@{}c@{}} \circlearrowleft \\ \tauji{k}{j}  \end{array}}} &&&&&&\ar@/^1pc/[llllll]|-{\tauji{j}{k}} [k,j]\ar[lu]|-{\tiny{\etaijk{j}{k}{j}{i}}} &\\
}}$$
\caption{Diagram representation of a gluing index category with three elements $i, j, k$. The identity maps on every element are not presented in this diagram for the sake of clarity. The arrows in both directions each compose into the identity map.\\}
\end{figure}

We introduce the concept of a $\mathbf{C}$-gluing data functor, a fundamental notion that plays a pivotal role in understanding gluing operations categorically. The existence of a limit determines whether a $\mathbf{C}$-gluing data functor is sufficient to construct a glued-up object within $\mathbf{C}$. Consequently, we introduce this concept. This definition lays the foundation for understanding how gluing operations are realized within a given category.

\begin{definition}\label{limext}
	Let $\mathrm{I}$ be a set and $\mathbf{C}$ be a category.
	\begin{enumerate} 
	\item We define a $\mathbf{C}$-\textbf{\textit{gluing data functor}} $\mathbf{G}$ of type $\mathrm{I}$ to be a functor from $\mathbf{Gl}(\mathrm{I})$ to $\mathbf{C}$ sending the pushout $[i,j,k]$ to the corresponding pushout, and the image of $\etaijk{n}{i}{j}{k}$ is the canonical morphism defined from the pushout, for all $i,j, k \in \mathrm{I}$ and $n \in \{j,k\}$. For any $i,j, k \in \mathrm{I}$, we denote $\Goi{\mathbf{G}}{i} := \Co{\mathbf{G}}([i])$, $\Goij{\mathbf{G}}{i}{j} := \Co{\mathbf{G}}([i,j])$ and $\Goijk{\mathbf{G}}{i}{j}{k} := \Co{\mathbf{G}}([i,j,k])$.
	
	\item If $\limi{\mathbf{G}}$ exists, then we say that $\mathbf{G}$ is a \textbf{\textit{gluable data functor}}. In this case, we say that $L$ is a \textbf{\textit{glued-up $\mathbf{C}$-object along $\mathbf{G}$ through}} $\subsm{\pi}{L}$ if $(L, \subsm{\pi}{L})$ is a cone over $\mathbf{G}$ that is isomorphic to $\limi{\mathbf{G}}$.
	\end{enumerate}
\end{definition}

\begin{figure}[H]
\includegraphics[scale=0.3]{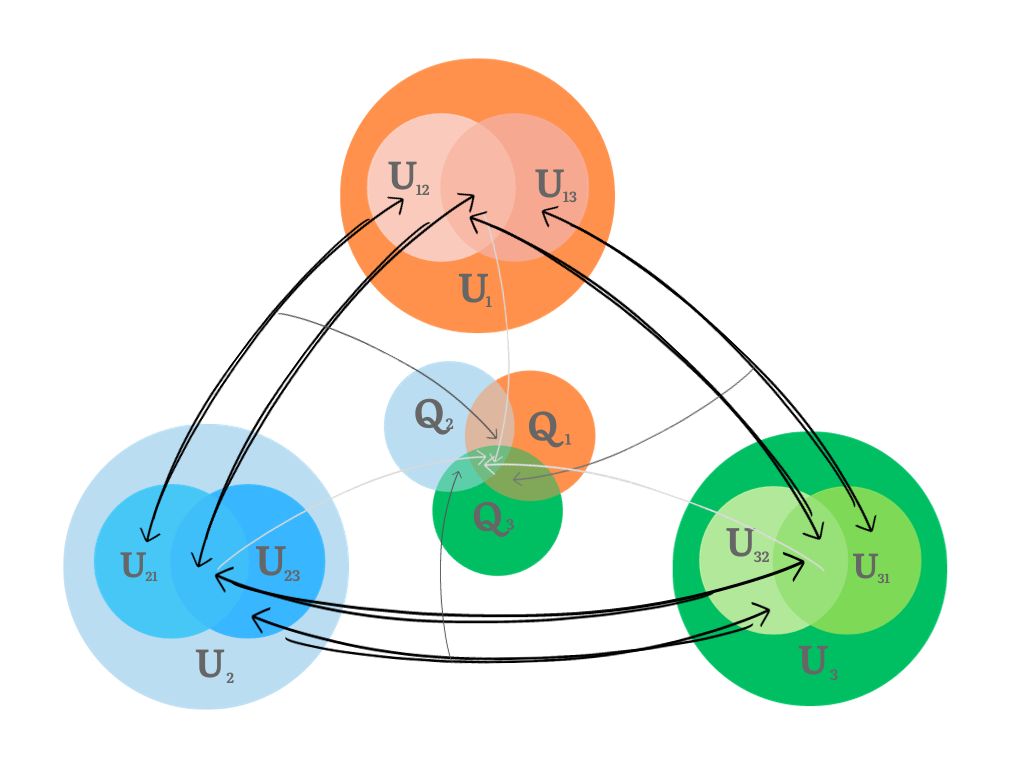}
 \caption{Representation for gluing and the glued-up object in $\mathbf{Top}^\op$ is as follows: The glued-up object $Q$ is situated in the center and is obtained from the three topological spaces, namely $U_1$, $U_2$, and $U_3$, which are mapped via the limit maps to $Q_1$, $Q_2$, and $Q_3$ respectively, forming a covering of $Q$. Moreover, $U_{12}$ is glued with $U_{21}$ and mapped to $Q_1\cap Q_2$, $U_{13}$ is glued with $U_{31}$ and mapped to $Q_1\cap Q_3$, and $U_{32}$ is glued with $U_{23}$ and mapped to $Q_2\cap Q_3$. Furthermore, all the double intersections, namely $U_{12}\cap U_{13}$, $U_{21}\cap U_{23}$, and $U_{31}\cap U_{32}$, are all mapped to the triple intersection $Q_1\cap Q_2 \cap Q_3$.
\\}
\end{figure}

It's important to emphasize the significance of the uniqueness of morphisms within the gluing index category $\glI{I}$. This uniqueness is pivotal, as it enables us to establish crucial equations that underpin the foundations of our framework. As we examine the implications of these definitions and remarks, we gain valuable insights into how gluing data functors serve as a vital bridge between the abstract realm of category theory and the practical application of gluing operations within diverse mathematical structures.

\begin{remark}\label{rem1}
	\begin{enumerate}
		\item Since each morphism between any two objects in $\Co{\glI{I}}$ is unique, any gluing data functor $\mathbf{G}$ is faithful.
		
		\item From Remark \ref{reglue} $(1)$ $(a)$, applying a $\mathbf{C}$-gluing data functor $\mathbf{G}$ to both equations yields $\Gnij{\mbf{G}}{\tauij{i}{j}}\circ \Gnij{\mbf{G}}{\tauij{j}{i}}=\subsm{\operatorname{id}}{\Goij{\mbf{G}}{i}{j}}$ and $\Gnij{\mbf{G}}{\tauij{j}{i}}\circ \Gnij{\mbf{G}}{\tauij{i}{j}}=\subsm{\operatorname{id}}{\Goij{\mbf{G}}{j}{i}}$. Therefore, $\Gnij{\mbf{G}}{\tauij{i}{j}}$ is indeed an isomorphism.
		\item When $\mathrm{I}=\{i\}$ we have $\glI{I}=\{i\}$, a $\mathbf{C}$-gluing data functor is a functor of the form $$\begin{array}{cccl}\mathbf{G}&: \glI{I} & \rightarrow & \mathbf{C}\\& i& \mapsto & \Goi{\mathbf{G}}{i}. \end{array}$$ Therefore a glued-up object is simply an object of $\mathbf{C}$ isomorphic to $\Goi{\mathbf{G}}{i}$.
	\end{enumerate}
\end{remark}
In the following definition, we introduce a category type: categories that admit glued-up objects, allowing us to explore their properties and characteristics.
\begin{definition}
We say that a \textbf{category admits a glued-up object} if every gluing data functor admits a glued-up object.
\end{definition}
We will prove that $\topo^\op$ is a category admitting glued-up objects.

\subsection{Characterizing cones and limits over a gluing data functor}

In the upcoming lemma, we present a synthetic characterization of cones over the gluing data functor $\mathbf{G}$. Our goal is to establish the equivalence among three conditions that not only simplify verification but also aid in identifying limits more conveniently. These conditions play a pivotal role in understanding the properties and behavior of cones within the context of a gluing data functor. 

\begin{lemma}\label{cocone} 
	Let $\mathbf{G}$ be a $\mathbf{C}$-gluing data functor. Let $N\in \Co{\mathbf{C}}$ and $\subsc{N}{\psi} : N \rightarrow \mathbf{G}$ is a family $\subsm{\left(\subsm{\subsc{N}{\psi}}{a}:N \rightarrow  \Goi{\mbf{G}}{a}\right)}{a\!\in\!  \Co{\glI{I}}}$ of morphisms in $\Cn{\mathbf{C}}$. The following statements are equivalent: 
	
	\begin{enumerate}
		\item $(N, \subsc{N}{\psi})$ is a cone over the underlying diagram of $\mathbf{G}$;
		\item $(N, \subsc{N}{\psi})$ makes the following diagrams commute, for all $i,j ,k \in \mathrm{I} $ and $n \in \{ j,k\}$:
		
		\begin{figure}[H]
			\begin{center}
				\begin{tabular}{cccl}
					{\tiny	\begin{tikzcd}[column sep=normal]&  N \arrow[]{dr}{ \randji{\subsc{N}{\psi}}{j}{i}}\arrow[swap]{dl}{\randij{\subsc{N}{\psi}}{i}{j}}  \\ \Goij{\mbf{G}}{i}{j}   && \Goij{\mbf{G}}{j}{i} \arrow[]{ll}{ \Gnij{\mbf{G}}{\tauij{i}{j}}}	\end{tikzcd}}			  	
					&	{\tiny\begin{tikzcd}[column sep=normal]&  N\arrow[swap]{dl}{\randij{\subsc{N}{\psi}}{i}{j}}  \arrow[]{dr}{ \randi{\subsc{N}{\psi}}{i}}\\ \Goij{\mbf{G}}{i}{j}   && \Goi{\mbf{G}}{i}\arrow[]{ll}{ \Gnij{\mbf{G}}{\etaij{i}{j}}}	\end{tikzcd}}	
					&   {\tiny \begin{tikzcd}[column sep=normal]&  N\arrow[]{dr}{ \randij{\subsc{N}{\psi}}{i}{n}}	\arrow[swap]{dl}{\randijk{\subsc{N}{\psi}}{i}{j}{k}}   \\\Goijk{\mbf{G}}{i}{j}{k}  && \Goij{\mbf{G}}{i}{n}\arrow[]{ll}{\Gnij{\mbf{G}}{\etaijk{n}{i}{j}{k}}} \end{tikzcd}}			  	
					\\ \tiny(a) & \tiny(b) &\tiny (c)				  	
				\end{tabular}
				\caption{Diagrams for Condition (2)}
				\label{topol10}
			\end{center}
		\end{figure}
		\item $(N, \randi{\subsc{N}{\psi}}{})$ makes the following diagrams commute, for all $i,j ,k \in \mathrm{I} $ and $n \in \{ j,k\}$:
		
		\begin{figure}[H]\begin{center}
			\begin{tabular}{cccr}
				{\tiny	\begin{tikzcd}[column sep=normal]&  N \arrow[]{dr}{ {\randi{\subsc{N}{\psi}}{j}}}	\arrow[swap]{dl}{\randij{\subsc{N}{\psi}}{i}{j}}  \\ \Goij{\mbf{G}}{i}{j}      && \Goi{\mbf{G}}{j} \arrow[]{ll}{\Gnij{\mbf{G}}{\tauij{i}{j}\circ \etaji{i}{j}}}\end{tikzcd}}			  	
				&	{\tiny\begin{tikzcd}[column sep=normal]&  N\arrow[]{dr}{ {\randi{\subsc{N}{\psi}}{i}}} \arrow[swap]{dl}{\randij{\subsc{N}{\psi}}{i}{j}} \\\Goij{\mbf{G}}{i}{j}     &&  \Goi{\mbf{G}}{i}\arrow[]{ll}{\Gnij{\mbf{G}}{\etaij{i}{j}}}	\end{tikzcd}}	
				& 	{\tiny\begin{tikzcd}[column sep=normal]&  N\arrow[]{dr}{ \randij{\subsc{N}{\psi}}{i}{n}}  \arrow[swap]{dl}{\randijk{\subsc{N}{\psi}}{i}{j}{k}}\\\Goijk{\mbf{G}}{i}{j}{k}    && \Goij{\mbf{G}}{i}{n} 	\arrow[]{ll}{\Gnij{\mbf{G}}{\etaijk{n}{i}{j}{k}}}\end{tikzcd}}
				\\ \tiny (a) &\tiny (b) & \tiny(c)
			\end{tabular}\end{center}
			\caption{Diagrams for Condition (3)}
			\label{topol11}
		\end{figure}
	\end{enumerate}
	
	\begin{proof}
		\begin{enumerate}
			\item[$(1)\Rightarrow (2) $]   This is clear. 
			
			\item[$(2)\Rightarrow (3)$] Assume that $(2)$ is satisfied. We only need to prove that $\Gnij{\mbf{G}}{\tauij{i}{j}\circ \etaji{i}{j}}\circ {\randi{\subsc{N}{\psi}}{j}}={\randij{\subsc{N}{\psi}}{i}{j}}$ for all $i,j\in \mathrm{I}$. This is a consequence of the commutativity of the diagrams in Figure \ref{topol10} $(2) (a),(b)$.  
			
			\item[$(3)\Rightarrow (1)$] Suppose that $(N, \randi{\subsc{N}{\psi}}{})$ makes the diagrams in $(3)$ commute. We want to prove that $(N, \randi{\subsc{N}{\psi}}{})$ is a cone over the underlying diagram of $\mathbf{G}$. 
			
			To show that $(N, \randi{\subsc{N}{\psi}}{})$ is a cone over the underlying diagram of $\mathbf{G}$, we need to prove that for any $a, b \in \Co{\glI{I}}$ and morphism $f: a \rightarrow b$ in $\Cn{\glI{I}}$, we have $\Gnij{\mbf{G}}{f}\circc \randi{\subsc{N}{\psi}}{a} = \randi{\subsc{N}{\psi}}{b}$. Since $(3)$ is satisfied, it suffices to prove that this equality holds for $f$ equal to $\tauij{i}{j}$ and $\tauijk{k}{i}{j}{k}$ for all $i,j,k \in \mathrm{I}$. 
			
			Let $i,j,k\in \mathrm{I}$. Since by assumption we have that $\Gnij{\mbf{G}}{\etaij{i}{j}} \circc \randi{\subsc{N}{\psi}}{i}=\randij{\subsc{N}{\psi}}{i}{j}$, then 
			\begin{align*}
				&\Gnij{\mbf{G}}{\tauij{i}{j} \circc \etaji{i}{j}}\circc  \randi{\subsc{N}{\psi}}{j}\\
				&=\Gnij{\mbf{G}}{\tauij{i}{j}}\circc \Gnij{\mbf{G}}{\etaji{i}{j}}\circc  \randi{\subsc{N}{\psi}}{j}\\
				& = \Gnij{\mbf{G}}{\tauij{i}{j}}\circc \randji{\subsc{N}{\psi}}{j}{i}. 
			\end{align*}
			
			Therefore, since $ \Gnij{\mbf{G}}{\tauij{i}{j}\circ \etaji{i}{j}}\circc  \randi{\subsc{N}{\psi}}{j}=\randij{\subsc{N}{\psi}}{i}{j}$, we obtain $\Gnij{\mbf{G}}{\tauij{i}{j}}\circc \randji{\subsc{N}{\psi}}{j}{i}=\randij{\subsc{N}{\psi}}{i}{j}$ as required.
			
			On the other hand, applying the functor $\mathbf{G}$ to the equality $\tauijk{k}{i}{j}{k} \circc \etaijk{i}{j}{i}{k}=   \etaijk{j}{i}{j}{k}  \circc \tauij{i}{j}$, we obtain $\Gnij{\mbf{G}}{\tauijk{k}{i}{j}{k} \circc \etaijk{i}{j}{i}{k}}=\Gnij{\mbf{G}}{\etaijk{j}{i}{j}{k} \circc \tauij{i}{j}}$. Therefore, we have
			\begin{align*}
				& \Gnij{\mbf{G}}{\etaijk{j}{i}{j}{k} \circc \tauij{i}{j}}\circc\randij{\subsc{N}{\psi}}{i}{j} \\
				&=(\Gnij{\mbf{G}}{\tauijk{k}{i}{j}{k}} \circc \Gnij{\mbf{G}}{\etaijk{i}{j}{i}{k}})\circc \randij{\subsc{N}{\psi}}{i}{j}\\
				&=\Gnij{\mbf{G}}{\tauijk{k}{i}{j}{k}} \circc (\Gnij{\mbf{G}}{\etaijk{i}{j}{i}{k}}\circc \randij{\subsc{N}{\psi}}{i}{j})\\
				&=\Gnij{\mbf{G}}{\tauijk{k}{i}{j}{k}} \circc \randijk{\subsc{N}{\psi}}{j}{i}{k}.
			\end{align*}
			
			Moreover, since by assumption $\Gnij{\mbf{G}}{\etaijk{j}{i}{j}{k} \circ \tauij{i}{j}}\circ \randji{\subsc{N}{\psi}}{j}{i}=\randijk{\subsc{N}{\psi}}{i}{j}{k}$, we obtain $\Gnij{\mbf{G}}{\tauijk{k}{i}{j}{k}} \circ \randijk{\subsc{N}{\psi}}{j}{i}{k}=\randijk{\subsc{N}{\psi}}{i}{j}{k}$ as required. This proves that $(N, \randi{\subsc{N}{\psi}}{})$ is a cone over the underlying diagram of $\mathbf{G}$.
		\end{enumerate}
	\end{proof}
\end{lemma}

The upcoming theorem provides a characterization of a glued-up object in the form of a pullback diagram. It serves as a fundamental result that elucidates the criteria for designating an object within a category as a "glued-up object"  along a specific gluing data functor.

\begin{theorem}   \label{preglued}
	Let $\mathbf{G}$ be a gluing data functor, $L\in \Co{\mathbf{C}}$ and $\subsm{\pi}{L}$ is a family $\subsm{\{ \dindi{\pi}{L}{a}\}}{a\! \in\! \Co{\glI{I}}}$ of morphisms $\dindi{\pi}{L}{a}: L\rightarrow \Goi{\mbf{G}}{a}$ for all $a \in \Co{\glI{I}}$. $L$ is a  glued-up $\mathbf{C}$-object  along $\mathbf{G}$ through $ \subsm{\pi}{L}$ if and only if for all $i,j, k \in \mathrm{I}$ and  $n \in \{ j,k\}$, the following properties are satisfied:
	\begin{enumerate}
		\item $\dindij{\pi}{L}{i}{j} = \Gnij{\mbf{G}}{\etaij{i}{j}}\circc \dindi{\pi}{L}{i}$ ;  
		\item $\dindijk{\pi}{L}{i}{j}{k}= \Gnij{\mbf{G}}{\etaijk{n}{i}{j}{k}}\circc \dindij{\pi}{L}{i}{n} $;
		\item $(L, \subsm{(\dindi{\pi}{L}{i})}{i\!\in\! \mathrm{I}})$ is the limit over the pullback diagram defined by the morphisms
		$$\Gnij{\mbf{G}}{\etaij{i}{j}}: \Goi{\mbf{G}}{i} \rightarrow \Goij{\mbf{G}}{i}{j} \text{ and } \Gnij{\mbf{G}}{\tauij{i}{j}\circc \etaji{i}{j}}: \Goi{\mbf{G}}{j} \rightarrow \Goij{\mbf{G}}{i}{j}.$$
	\end{enumerate}  
	\begin{proof}
		Let $L$ be a glued-up $\mathbf{C}$-object constructed by gluing along $\mathbf{G}$ through the morphisms $\dindi{\pi}{L}{}$. According to Definition \ref{limext}, we know that $\limi{\mathbf{G}}$ exists and is isomorphic to $(L, \subsm{\pi}{L})$. Our objective is to demonstrate the satisfaction of properties $(1)$, $(2)$, and $(3)$ as outlined in the theorem.
		
		The diagrams in Lemma \ref{cocone} $(3)$ $(b)$ and $(c)$ are equivalent to properties $(1)$ and $(2)$, respectively. To establish property $(3)$ of the theorem, we start by proving that the following diagram for arbitrary $i$ and $j$ in $\mathrm{I}$:
		\begin{figure}[H] 
			\begin{center}
				{\tiny	\begin{tikzcd}[column sep=normal]
						&\Goij{\mbf{G}}{i}{j}  &&& \Goi{\mbf{G}}{j}\arrow[swap]{lll}{\Gnij{\mbf{G}}{\tauij{i}{j}\circ \etaji{i}{j}}}\\
						&\Goi{\mbf{G}}{i} \arrow[]{u}{\Gnij{\mbf{G}}{\etaij{i}{j}}}&& & L\arrow[]{lll}{ \dindi{\pi}{L}{i}}\arrow[swap]{u} { \dindi{\pi}{L}{j}}
			\end{tikzcd}}\end{center} 	 \caption{}\label{topol13}	
		\end{figure}    
		\noindent commutes. This holds by combining the commutativity of the diagram in Figure (3) (a) and (3) (b) of Lemma \ref{cocone}. 
		
		Let us now consider the scenario where we have a pair $(L',\subsm{\pi}{L'})$, with $L'\in \Co{\mathbf{C}}$ and $\subsm{\pi}{L'}$ representing a family of maps $\dindpi{\pi}{L}{i} : L'\rightarrow \Goi{\mbf{G}}{a}$ where $a\in  \Co{\glI{I}}$, such that the following diagram commutes for all $i,j\in \mathrm{I}$:
		
		\begin{figure}[H] 
			\begin{center}
				{\tiny	
					\begin{tikzcd}[column sep=normal]
						&\Goij{\mbf{G}}{i}{j} & &&  \Goi{\mbf{G}}{j}\arrow[swap]{lll}{{\Gnij{\mbf{G}}{\tauij{i}{j}\circ \etaji{i}{j}}}}\\
						& \Goi{\mbf{G}}{i} \arrow[]{u}{\Gnij{\mbf{G}}{\etaij{i}{j}}}&& &  L'\arrow[]{lll}{\dindpi{\pi}{L}{i}}\arrow[swap]{u} {\dindpi{\pi}{L}{j}}
					\end{tikzcd}
				}
			\end{center} 
			\caption{}
			\label{topol116}
		\end{figure}
		
		Now, our aim is to establish that the pair $(L',\widetilde{\subsm{\pi}{L'}})$ forms a cone over $\mathbf{G}$, where $\widetilde{\subsm{\pi}{L'}}$ is a family of maps $\widetilde{\dindpi{\pi}{L}{a}}: L'\rightarrow \Goi{\mbf{G}}{a}$, and each $\widetilde{\dindpi{\pi}{L}{i}}$ coincides with $\dindpi{\pi}{L}{i}$, $\widetilde{\dindpij{\pi}{L}{i}{j}}$ is given by ${\Gnij{\mbf{G}}{\etaij{i}{j}}}\circc \dindpi{\pi}{L}{i}$, and $ \widetilde{\dindpijk{\pi}{L}{i}{j}{k}}$ is computed as $\Gnij{\mbf{G}}{\etaijk{n}{i}{j}{k}}\circc {\dindpij{\pi}{L}{i}{n}}$, for all $i,j, k \in \mathrm{I}$ and  $n \in \{ j,k\}$.
		
		In order to prove that $(L',\widetilde{\subsm{\pi}{L'}})$ is a cone over $\mathbf{G}$, according to the definition of $\widetilde{\subsm{\pi}{L'}}$ and Lemma \ref{cocone}, it is enough to prove that the following diagram commutes, for all $i,j\in \mathrm{I}$: 
		\begin{figure}[H]
			\begin{center}
				{\tiny	\begin{tikzcd}[column sep=normal]&  L' \arrow[]{dr}{\widetilde{\dindpi{\pi}{L}{j}}}	\arrow[swap]{dl}{\widetilde{\dindpij{\pi}{L}{i}{j}}}  \\ \Goij{\mbf{G}}{i}{j}    && \Goi{\mbf{G}}{j} \arrow[]{ll}{\Gnij{\mbf{G}}{\tauij{i}{j}\circ \etaji{i}{j}}}\end{tikzcd}}
			\end{center}
			\caption{}\label{topol126}
		\end{figure}
		
		\noindent 
		The commutativity of this diagram follows directly from the definition of $\widetilde{\subsm{\pi}{L'}}$ and the commutativity of the diagram in Figure \ref{topol116}. Therefore, the pair $(L',\widetilde{\subsm{\pi}{L'}})$ is a cone over $\mathbf{G}$. By assumption, $\limi{\mathbf{G}}\simeq (L,\subsm{\pi}{L})$, so there exists a unique morphism, say $\mu: L'\rightarrow L$, making each of the following diagrams commute, for all $i,j, k \in \mathrm{I}$ and  $n \in \{ j,k\}$:
		\begin{figure}[H]\begin{center}	\begin{tabular}{cccl}
				{\tiny\begin{tikzcd}[column sep=normal] &L'\arrow[bend left, labels=description]{dddr}{\widetilde{\dindpi{\pi}{L}{j}}} \arrow[dashed]{dd}{\exists ! \mu}\arrow[bend right, labels=description]{dddl}{\widetilde{\dindpij{\pi}{L}{i}{j}}} \\\\ & L\arrow[labels=description]{dr}{{\subsm{\pi}{\subsm{L}{j}}}} \arrow[labels=description]{dl}{\dindij{\pi}{L}{i}{j}}    \\ \Goij{\mbf{G}}{i}{j} & & \Goi{\mbf{G}}{j}\arrow[]{ll}{{\Gnij{\mbf{G}}{\tauij{i}{j}\circ \etaji{i}{j}}}}   
				\end{tikzcd}} &	{\tiny\begin{tikzcd}[row sep=normal] & L'\arrow[bend left, labels=description]{dddr}{\widetilde{\dindpi{\pi}{L}{i}}} \arrow[bend right, labels=description]{dddl}{\widetilde{\dindpij{\pi}{L}{i}{j}}} \arrow[dashed]{dd}{\exists ! \mu}  \\\\ & L \arrow[labels=description]{dr}{\dindi{\pi}{L}{i}}\arrow[labels=description]{dl}{\dindij{\pi}{L}{i}{j}}    \\ \Goij{\mbf{G}}{i}{j}  & & \Goi{\mbf{G}}{i} \arrow[]{ll}{\Gnij{\mbf{G}}{\etaij{i}{j}}}
				\end{tikzcd}} &  {\tiny\begin{tikzcd}[row sep=normal] & L'\arrow[bend left, labels=description]{dddr}{\widetilde{\dindpij{\pi}{L}{i}{n}}} \arrow[bend right, labels=description]{dddl}{\widetilde{\dindpijk{\pi}{L}{i}{j}{k}}} \arrow[dashed]{dd}{\exists ! \mu}     \\\\ &L \arrow[labels=description]{dr}{\subsm{\pi}{\subsm{L}{[i,n]}}} \arrow[labels=description]{dl}{\dindijk{\pi}{L}{i}{j}{k}} \\ \Goijk{\mbf{G}}{i}{j}{k} & & \Goij{\mbf{G}}{i}{n} \arrow[]{ll}{\Gnij{\mbf{G}}{\etaijk{n}{i}{j}{k}}}
				\end{tikzcd}}\\ \tiny (a) &\tiny (b) &\tiny (c) \end{tabular} \end{center}\caption{}\label{topol118}           \end{figure} 
		
		We now pick such a $\mu$ and by construction of $\mu$ the following diagram

		\begin{figure}[H]
			\begin{center}
				{\tiny\begin{tikzcd}[column sep=normal]
						\Goij{\mbf{G}}{i}{j}  && & 
						\Goi{\mbf{G}}{j} \arrow[swap]{lll}{{\Gnij{\mbf{G}}{\tauij{i}{j}\circ \etaji{i}{j}}}}  \\
						\Goi{\mbf{G}}{i}\arrow[]{u}{\Gnij{\mbf{G}}{\etaij{i}{j}}}     && & 
						L\arrow[swap]{u}{\subsm{\pi}{\subsm{L}{j}}}\arrow[]{lll}{\dindi{\pi}{L}{i}} \\
						&&&& L' \arrow[swap,dashed]{ul}{ \mu}\arrow[bend left=10]{ullll}{{ {\dindpi{\pi}{L}{i}}}} \arrow[swap,bend right=15]{uul}{ \dindpi{\pi}{L}{j}}
				\end{tikzcd} }
			\end{center}\caption{}\label{topol117}
		\end{figure}
		\noindent also commutes, for all $i,j\in \mathrm{I}$. This concludes the proof. 
		
		Conversely, suppose that properties $(1)$, $(2)$ and $(3)$ of the theorem are satisfied. We want to prove that $L$ is a glued-up $\mathbf{C}$-object along $\mathbf{G}$ through $\subsm{\pi}{L}$. From the properties $(1)$, $(2)$ and $(3)$ of the theorem, we have that $(L , \subsm{\pi}{L})$ is a cone over $\mathbf{G}$.  
		Suppose that $(L' , \subsm{\pi}{L'})$ is another cone over $\mathbf{G}$. That is, the following diagrams
		
		\begin{figure}[H] \begin{center}\begin{tabular}{cccr}
				{\tiny	\begin{tikzcd}[column sep=normal]&  L' \arrow[]{dr}{\dindpi{\pi}{L}{j}}	\arrow[swap]{dl}{\dindpij{\pi}{L}{i}{j}}  \\ \Goij{\mbf{G}}{i}{j}    && \Goi{\mbf{G}}{j} \arrow[]{ll}{{\Gnij{\mbf{G}}{\tauij{i}{j}\circ \etaji{i}{j}}}}\end{tikzcd}}			  	
				&	{\tiny\begin{tikzcd}[column sep=normal]&  L'\arrow[]{dr}{{\dindpi{\pi}{L}{i}}} \arrow[swap]{dl}{\dindpij{\pi}{L}{i}{j}} \\\Goij{\mbf{G}}{i}{j}    &&\Goi{\mbf{G}}{i}\arrow[]{ll}{\Gnij{\mbf{G}}{\etaij{i}{j}}}	\end{tikzcd}}				  				         & 	{\tiny\begin{tikzcd}[column sep=normal]&  L'\arrow[]{dr} {\dindpij{\pi}{L}{i}{n}}  \arrow[swap]{dl}{{\dindpijk{\pi}{L}{i}{j}{k}}}\\\Goijk{\mbf{G}}{i}{j}{k}    && \Goij{\mbf{G}}{i}{n} 	\arrow[]{ll}{\Gnij{\mbf{G}}{\etaijk{n}{i}{j}{k}}}\end{tikzcd}} \\\tiny (a) &\tiny (b) & \tiny(c)		\end{tabular}	\end{center}	  	
			\caption{}\label{topol120}				  	\end{figure}
		\noindent   commute, for all $i,j, k \in \mathrm{I}$ and  $n \in \{ j,k\}$. Now, we want to show that the pair $(L' , \subsm{\pi}{L'})$ makes the following diagram
		\begin{figure}[H] 
			\begin{center}
				{\tiny	\begin{tikzcd}[column sep=normal]
						&\Goij{\mbf{G}}{i}{j}  & && \Goi{\mbf{G}}{j}\arrow[swap]{lll}{{\Gnij{\mbf{G}}{\tauij{i}{j}\circ \etaji{i}{j}}}}\\
						& \Goi{\mbf{G}}{i} \arrow[]{u}{\Gnij{\mbf{G}}{\etaij{i}{j}}}&& & L'\arrow[]{lll}{{ \dindpi{\pi}{L}{i}}}\arrow[swap]{u} {{ \dindpi{\pi}{L}{j}}}
			\end{tikzcd}}\end{center} 	 \caption{}\label{topol127}	
		\end{figure}
		\noindent commute, for all $i,j \in \mathrm{I}$. By Figure \ref{topol120} $(a)$ and $(b)$, we have 
		$$\Gnij{\mbf{G}}{\tauij{i}{j}\circ \etaji{i}{j}}\circ \dindpi{\pi}{L}{j}={\Gnij{\mbf{G}}{\tauij{i}{j}}\circ\big(\Gnij{\mbf{G}}{\etaji{i}{j}}}\circ \dindpi{\pi}{L}{j}\big)=\Gnij{\mbf{G}}{\tauij{i}{j}}\circ\dindpij{\pi}{L}{j}{i} =\dindpij{\pi}{L}{i}{j}=\Gnij{\mbf{G}}{\etaij{i}{j}}\circ {\dindpi{\pi}{L}{i}}.$$ Hence the diagram in Figure \ref{topol127} commutes.   Then, by Property $(3)$ of the theorem, there exists a unique morphism $\mu: L'\rightarrow L$ making the following diagram 
		\begin{figure}[H]
			\begin{center}
				{\tiny\begin{tikzcd}[column sep=normal]
						\Goij{\mbf{G}}{i}{j}  && & 
						\Goi{\mbf{G}}{j} \arrow[swap]{lll}{{\Gnij{\mbf{G}}{\tauij{i}{j}\circ \etaji{i}{j}}}}  \\
						\Goi{\mbf{G}}{i}\arrow[]{u}{\Gnij{\mbf{G}}{\etaij{i}{j}}}     & && 
						L\arrow[swap]{u}{{\subsm{\pi}{\subsm{L}{j}}}}\arrow[]{lll}{\dindi{\pi}{L}{i}} \\
						&&&& L' \arrow[swap,dashed, near end]{ul}{ \exists ! \mu}\arrow[bend left=15]{ullll}{{\dindpi{\pi}{L}{\!i}}} \arrow[swap,bend right=15]{uul}{\dindpi{\pi}{L}{\!j}}
				\end{tikzcd} }
			\end{center}\caption{}\label{topol128}
		\end{figure}
		\noindent   commute, for all $i,j\in \mathrm{I}$. 
		We choose such a $\mu$. Thanks to the commutativity of the diagram as shown in Figure \ref{topol128}, combined with Properties $(1)$ and $(2)$ of the theorem, which are assumed to be satisfied, we obtain that each of the following diagrams 
		\begin{figure}[H]\begin{center}	\begin{tabular}{cccl}
				{\tiny\begin{tikzcd}[column sep=normal] &L'\arrow[bend left, labels=description]{dddr}{\dindpi{\pi}{L}{j}} \arrow[dashed]{dd}{\exists ! \mu}\arrow[bend right, labels=description]{dddl}{\dindpij{\pi}{L}{i}{j}} \\\\ & L\arrow[labels=description]{dr}{\dindi{\pi}{L}{j}} \arrow[labels=description]{dl}{\dindij{\pi}{L}{i}{j}}    \\ \Goij{\mbf{G}}{i}{j} & & \Goi{\mbf{G}}{j}\arrow[]{ll}{{\Gnij{\mbf{G}}{\tauij{i}{j}\circ \etaji{i}{j}}}}   
				\end{tikzcd}} &	{\tiny\begin{tikzcd}[row sep=normal] & L'\arrow[bend left, labels=description]{dddr}{{\dindpi{\pi}{L}{i}}} \arrow[bend right, labels=description]{dddl}{\dindpij{\pi}{L}{i}{j}} \arrow[dashed]{dd}{\exists ! \mu}  \\\\ & L \arrow[labels=description]{dr}{\dindi{\pi}{L}{i}}\arrow[labels=description]{dl}{\dindij{\pi}{L}{i}{j}}    \\ \Goij{\mbf{G}}{i}{j}  & & \Goi{\mbf{G}}{i} \arrow[]{ll}{\Gnij{\mbf{G}}{\etaij{i}{j}}}
				\end{tikzcd}} &  {\tiny\begin{tikzcd}[row sep=normal] & L'\arrow[bend left, labels=description]{dddr}{\dindpij{\pi}{L}{i}{n}} \arrow[bend right, labels=description]{dddl}{{\dindpijk{\pi}{L}{i}{j}{k}}} \arrow[dashed]{dd}{\exists ! \mu}     \\\\ &L \arrow[labels=description]{dr}{{\subsm{\pi}{\subsm{L}{[i,n]}}}} \arrow[labels=description]{dl}{{\dindijk{\pi}{L}{i}{j}{k}}} \\ \Goijk{\mbf{G}}{i}{j}{k}  & & \Goij{\mbf{G}}{i}{n} \arrow[]{ll}{\Gnij{\mbf{G}}{\etaijk{n}{i}{j}{k}}}
				\end{tikzcd}}\\ \tiny (a) &\tiny (b) &\tiny (c) \end{tabular} \end{center}\caption{}\label{topol121}           \end{figure} 
		\noindent commutes, for all $i,j, k \in \mathrm{I}$ and  $n \in \{ j,k\}$.  
		With this, we have shown that $(L,\subsm{\pi}{L})$ is a cone over $\mathbf{G}$, and any other cone $(L',\subsm{\pi}{L'})$ can be uniquely mapped to $(L,\subsm{\pi}{L})$ via the unique morphism $\mu: L' \to L$ satisfying the commutativity of the diagram in Figure \ref{topol116}. 
		Therefore, $\limi{\mathbf{G}}$ exists and is isomorphic to $(L,\subsm{\pi}{L})$, which implies that $L$ is a glued-up $\mathbf{C}$-object along $\mathbf{G}$ through $\subsm{\pi}{L}$.
	\end{proof}
\end{theorem}

In the following remark, we translate the properties of a glued-up object within the context of a contravariant functor.
\begin{remark}\label{oppc}
\hspace{2em}
\begin{enumerate}
	\item If $\mathbf{G}$ is a gluing data functor of  type $\mathrm{I}$ from $\glI{I}$ to $\mathbf{C}^{\op}$. 
	Then, $Q$ is a glued-up $\mathbf{C}^{\op}$-object along $\mathbf{G}$ through $\subsm{\iota}{Q}^{\op}$ if and only if the following properties hold in $\mathbf{C}$, for all $i,j, k \in \mathrm{I}$ and  $n \in \{ j,k\}$:
	\begin{itemize}
		\item[(a)] $\dindij{\iota}{Q}{i}{j}=\dindi{\iota}{Q}{i} \circc \Gnij{\mbf{G}}{\etaij{i}{j}}^{\op}$;  
		\item[(b)] $\dindijk{\iota}{Q}{i}{j}{k} =\dindij{\iota}{Q}{i}{n}  \circc \Gnij{\mbf{G}}{\etaijk{n}{i}{j}{k}}^{\op} $;
		
		\item[(c)] $(Q, \subsm{(\dindi{\iota}{Q}{i})}{i \! \in \! \mathrm{I}})$ is the limit over the pushout diagram defined by the morphisms \\ $\Gnij{\mbf{G}}{\etaij{i}{j}}^{\op}: \Goij{\mbf{G}}{i}{j} \rightarrow \Goi{\mbf{G}}{i}$ and ${\Gnij{\mbf{G}}{\tauij{i}{j} \! \circ \! \etaji{i}{j}}}^{\op}:  \Goij{\mbf{G}}{i}{j} \rightarrow \Goi{\mbf{G}}{j}$. 
		\end{itemize}
		\item If $\mathrm{I}=\{1,2\}$, then $\Co{\glI{I}}=\{1,2,[1,2],[2,1]\}$, and if $\mathbf{G}$ is a $\mathbf{C}$-gluing data functor from $\glI{I}$, then a glued-up object, if it exists, is the pullback of $\Gnij{\mbf{G}}{\etaij{1}{2}}$ and $\Cn{\mathbf{G}}(\tauij{1}{2}\circ \etaji{1}{2})$. So, pullbacks are a particular case of limits of a certain gluing data functor, and therefore, gluing data functors generalize the concept of pullbacks.

	\end{enumerate}
\end{remark}

\begin{figure}[H]
$${\tiny \xymatrixrowsep{0.2in}
\xymatrixcolsep{0.1in} \xymatrix{ 
&&&& Q    &&&& \\
&&&\circlearrowleft & \subsm{U}{i}=\subsm{U}{i,i} \ar[u]_{\dindsi{\iota}{Q}{i}} &\circlearrowleft &&& \\
&&&\subsm{U}{i,j} \ar[ru]|-{\upsilon_{i,j}}  \ar@/^1pc/[ddll]|-{\subsm{\varphi}{i,j}}  & \circlearrowleft  & \subsm{U}{i,k} \ar[lu]|-{\upsilon_{i,k}}  \ar@/^1pc/[ddrr]|-{\subsm{\varphi}{i,k}} &&& \\
&&&\circlearrowleft& {\begin{array}{@{}c@{}} \subsm{U}{i,k}\subsm{\times}{\subsm{U}{i}} \subsm{U}{i,j}  \\ = \\ \subsm{U}{i,j}\subsm{\times}{\subsm{U}{i}} \subsm{U}{i,k}  \ar[lu]|-{\tiny{\iuv{\subsm{U}{i,j}\subsm{\times}{\subsm{U}{i}} \subsm{U}{i,k}}{\subsm{U}{i,k}}}} \ar[ru]|-{\tiny{\iuv{\subsm{U}{i,k}\subsm{\times}{\subsm{U}{i}} \subsm{U}{i,j}}{\subsm{U}{i,k}}}}\end{array}} \ar@/^1pc/[rrdd]|-{\subsm{\varphi_{j}}{(i,k)}} \ar@/_1pc/[lldd]|-{\subsm{\varphi_{k}}{(i,j)}}  & \circlearrowleft & &&\\  
&\subsm{U}{j,i} \ar[ld]|-{\upsilon_{j,i}}  \ar@/^1pc/[uurr]|-{\subsm{\varphi}{j,i}} &&& \circlearrowleft  &&&\subsm{U}{k,i} \ar[rd]|-{\upsilon_{k,i}}  \ar@/^1pc/[uull]|-{\subsm{\varphi}{k,i}} & \\
  \subsm{U}{j}=\subsm{U}{j,j} \ar@/_-4pc/[uuuuurrrr]^{\dindsi{\iota}{Q}{j}}  &\circlearrowleft  &\ar@/_1pc/[rruu]|-{\subsm{\varphi_{k}}{(j,i)}} {\begin{array}{@{}c@{}} \subsm{U}{j,k}\subsm{\times}{\subsm{U}{j}} \subsm{U}{j,i} \\ = \\ \subsm{U}{j,i}\subsm{\times}{\subsm{U}{j}} \subsm{U}{j,k} \ar[ld]|-{\tiny{\iuv{\subsm{U}{j,i}\subsm{\times}{\subsm{U}{j}} \subsm{U}{j,k}}{\subsm{U}{j,k}}}} \ar[lu]|-{\tiny{\iuv{\subsm{U}{j,k}\subsm{\times}{\subsm{U}{j}} \subsm{U}{j,i}}{\subsm{U}{j,i}}}}
    \end{array}}   \ar@/_1pc/[rrrr]|-{\subsm{\varphi_{i}}{(j,k)} }   &&   &&  {\begin{array}{@{}c@{}} \subsm{U}{k,i}\subsm{\times}{\subsm{U}{k}} \subsm{U}{k,j} \\ = \\ \subsm{U}{k,j}\subsm{\times}{\subsm{U}{k}} \subsm{U}{k,i} \ar[rd]|-{\tiny{\iuv{\subsm{U}{k,j}\subsm{\times}{\subsm{U}{k}} \subsm{U}{k,i}}{\subsm{U}{k,j}}}} \ar[ru]|-{\tiny{\iuv{\subsm{U}{k,i}\subsm{\times}{\subsm{U}{k}} \subsm{U}{k,j}}{\subsm{U}{k,i}}}} \end{array}} \ar@/^1pc/[lluu]|-{\subsm{\varphi_{j}}{(k,i)}}   \ar@/_1pc/[llll]|-{\subsm{\varphi_{i}}{(k,j)}}
&\circlearrowleft  & \subsm{U}{k}=\subsm{U}{k,k} \ar@/^-4pc/[uuuuullll]_{\dindsi{\iota}{Q}{k}} \\
&\subsm{U}{j,k} \ar[lu]|-{\upsilon_{j,k}}  \ar@/^1pc/[rrrrrr]|-{{\begin{array}{@{}c@{}} \circlearrowleft \\ \subsm{\varphi}{j,k}  \end{array}}} &&&&&&\ar@/^1pc/[llllll]|-{\subsm{\varphi}{k,j}} \subsm{U}{k,j} \ar[ru]|-{\upsilon_{k,j}}   &\\
}}$$
\caption{Diagram representation of a gluing index category and the glued object $(Q, \iota)$, where, for all $i,j,k\in \mathrm{I}$ and $n\in \{j,k\}$: $\Goi{\mathbf{G}}{i} := \subsm{U}{i}$; $\Goij{\mathbf{G}}{i}{j} := \subsm{U}{i,j}$; $\Goijk{\mathbf{G}}{i}{j}{k} := \subsm{U}{i,j}\subsm{\times}{\subsm{U}{i}} \subsm{U}{i,k}$; $\Gnij{\mbf{G}}{\tauij{i}{j}}^{\op}:=\subsm{\varphi}{i,j}$; $\Gnij{\mbf{G}}{\tauijk{k}{i}{j}{k}}^{\op}:=\subsm{\varphi_{k}}{(i,j)}$; $\Gnij{\mbf{G}}{\etaij{i}{j}}^{\op}:=\upsilon_{i,j}$; $\Gnij{\mbf{G}}{\etaijk{n}{i}{j}{k}}^{\op}:=\iuv{\subsm{U}{i,j}\subsm{\times}{\subsm{U}{i}} \subsm{U}{i,k}}{\subsm{U}{i,n}}$.\\ }
\end{figure}

\subsection{Refinement}
Inspired by the notion of refinement in coverings (see \cite[\S 1]{Conrad}), we introduce a refinement concept for gluing data functors of different types, utilizing a mapping between their types. This refinement notion allows us to define a refinement map, drawing inspiration from topological spaces.

Moreover, this refinement concept naturally leads to the definition of a functor from the category of sets to the category of choice, employing the concept of gluing data functors.

\begin{definition}\label{refino}
\begin{enumerate}
\item Let $\mathrm{I}, \mathrm{J}$ be two index sets, and $\gamma : \mathrm{I}\rightarrow \mathrm{J}$ be a map. We define $\mathbf{Gl}(\gamma):\mathbf{Gl}(\mathrm{I})\rightarrow \mathbf{Gl}(\mathrm{J})$ to be the functor such that 
\begin{itemize}
\item $\Co{\mathbf{Gl}(\gamma)}(i) = \gamma(i)$;
\item $\Co{\mathbf{Gl}(\gamma)}([i,j]) = [\gamma(i), \gamma(j)]$;
\item $\Co{\mathbf{Gl}(\gamma)}([i,j,k]) = [\gamma(i), \gamma(j), \gamma(k)]$;
\item $\Cn{\mathbf{Gl}(\gamma)}(\etaij{i}{j}) = \etaij{\gamma(i)}{\gamma(j)}$;
\item $\Cn{\mathbf{Gl}(\gamma)}(\etaijk{n}{i}{j}{k}) = \etaijk{\gamma(n)}{\gamma(i)}{\gamma(j)}{\gamma(k)}$;
\item $\Cn{\mathbf{Gl}(\gamma)}(\tauij{i}{j}) = \tauij{\gamma(i)}{\gamma(j)}$;
\item $\Cn{\mathbf{Gl}(\gamma)}(\tauijk{k}{i}{j}{k}) =\tauijk{\gamma(k)}{\gamma(i)}{\gamma(j)}{\gamma(k)}$;
\end{itemize}
\item We define the $\mathbf{Gl}$-functor, denoted $\mathbf{Gl}$, to be the functor from $\mathbf{Sets}$ to $\mathbf{C}$ such that $\Co{\mathbf{Gl}}$ sends a set $\mathrm{I}$ to $\mathbf{Gl}(\mathrm{I})$ and $\Cn{\mathbf{Gl}}$ sends a map $\gamma$ to $\mathbf{Gl}(\gamma)$.
\end{enumerate}
\end{definition}

\begin{definition}
Let $\mathbf{C}$ be a category, and let $\mathrm{I}$ and $\mathrm{J}$ be sets. We say that a $\mathbf{C}$-gluing functor $\mathbf{G}$ of type $\mathrm{J}$ \textbf{\textit{refines}} a $\mathbf{C}$-gluing functor $\mathbf{F}$ of type $\mathrm{I}$ if there exists a map $\gamma: \mathrm{I} \rightarrow \mathrm{J}$ and a natural transformation $\subsm{\rho}{\gamma, \mathbf{G}, \mathbf{F}}:  \mathbf{G}\circ \mathbf{Gl}(\gamma)  \rightarrow  \mathbf{F}$. In this case, we also say that $\mathbf{G}$ is a refinement of $\mathbf{F}$, and $\subsm{\rho}{\gamma, \mathbf{G}, \mathbf{F}}$ is a refinement morphism from $\mathbf{G}$ to $\mathbf{F}$.
\end{definition}

In the next lemma, we observe that the refinement of gluing functions induces a unique canonical morphism between their respective glued-up objects.
\begin{lemma}  \label{refine}
Let $\mathbf{C}$ be a category, $\mathbf{G}$ be a $\mathbf{C}$-gluing functor of type $\mathrm{J}$, and $\mathbf{F}$ be a $\mathbf{C}$-gluing functor of type $\mathrm{I}$. Suppose that $\limi{\mathbf{G}}$ and $\limi{\mathbf{F}}$ exist. If $\mathbf{G}$ is a refinement of $\mathbf{F}$, then there exists a unique cone morphism $\mu : \limi{\mathbf{G}} \rightarrow \limi{\mathbf{F}}$.
\end{lemma}

\begin{proof}
We write 
\begin{itemize}
\item $ \limi{\mathbf{G}}$ as $(\subsm{L}{\mathbf{G}}, \subsm{\pi}{\mathbf{G}})$;
\item $ \limi{\mathbf{F}}$ as $(\subsm{L}{\mathbf{F}}, \subsm{\pi}{\mathbf{F}})$.
\end{itemize}
Since $\limi{\mathbf{F}}$ is a terminal cone over $\mathbf{F}$, we obtain the existence of the unique map $\mu: \subsm{L}{\mathbf{G}}\rightarrow \subsm{L}{\mathbf{F}}$ by considering the following commutative diagram for all $i,j\in \mathrm{I}$:

\begin{figure}[H]
	\begin{center}
		{\tiny
			\begin{tikzcd}[column sep=normal]
				\Co{\mathbf{F}}(i,j)                                                                                                                                                                          &  & & \Co{\mathbf{F}}(j) \arrow[swap]{lll} {\Cn{\mathbf{F}}(\etaaji{i}{j})} &  &                                                                                                 \\
				&  &                                                                                                                                                                       &  &  &                                                                                                                                                       &  &                                                                                                 \\
				&  & \Co{\mathbf{F}}(i) \arrow[swap]{lluu} {\Cn{\mathbf{F}}(\etaij{i}{j})} &            &  & \subsm{L}{\mathbf{F}}  \arrow[swap]{lluu}{\dindsi{\pi}{\mathbf{F}}{j}}  \arrow[near end]{lll}{\dindsi{\pi}{\mathbf{F}}{i}}\\
				\Co{\mathbf{G}}(\gamma(i), \gamma(j))  \arrow{uuu}{ \subsm{\subsm{\rho}{\gamma, \mathbf{G}, \mathbf{F}}}{[i,j]}}                                                                                                                                                                                                                                                  &  &  & \Co{\mathbf{G}}(\gamma(j)) \arrow[swap]{lll} {\Cn{\mathbf{G}}(\etaaji{\gamma(i)}{\gamma(j)})} \arrow[near end]  {uuu}{\subsm{\subsm{\rho}{\gamma, \mathbf{G}, \mathbf{F}}}{j}}   &  &  \\	&  &                                                                                                                                                                       &  &  &   &  &   \\&  & \Co{\mathbf{G}}(\gamma(i)) \arrow[]{lluu}{\Cn{\mathbf{G}}(\etaij{\gamma(i)}{\gamma(j)})}      \arrow[swap,near start]{uuu} {\subsm{\subsm{\rho}{\gamma, \mathbf{G}, \mathbf{F}}}{i}}  &                                                                                                                                                       &  & \subsm{L}{\mathbf{G}}\arrow[]{lll} {\dindsi{\pi}{\mathbf{G}}{i}}  \arrow[swap]{lluu}{\dindsi{\pi}{\mathbf{G}}{j}}   \arrow[swap,dotted]{uuu}{\exists ! \mu}                                                                           
			\end{tikzcd}
		}
		
	\end{center}  \caption{}\label{topoul128}
\end{figure}

\end{proof}

We can define the category of gluing data functors, where we define morphisms to be refinement morphisms.

\begin{definition}\label{gdff}
Let $\mathbf{C}$ be a category. We define the category of $\mathbf{C}$-\textbf{\textit{gluing data functors}}, denoted as $\mathbf{Gdf}(\mathbf{C})$, to be the category whose objects are $\mathbf{C}$-gluing data functors, and morphisms are refinement morphisms.
\end{definition}

With all the concepts defined in the previous definitions, we can now explore the composition of gluing data functors within the category of gluing data functors. This process allows us to combine gluing data functors, akin to composing them.
\begin{proposition}\label{refinery}
Let $\mathbf{C}$ be a category and $\mathrm{I}$ be a set. Let $\mathcal{G}$ be a $\mathbf{Gdf}(\mathbf{C})$-gluing data functor of type $\mathrm{I}$. Suppose that 
\begin{itemize}
\item[(a)] $\limi{\subsm{\mathcal{G}}{a}}$ exists for all $a\in \Co{\glI{I}}$.
\item[(b)] $\limi{\subsm{\mathcal{G}}{[i,j,k]}}:=\limi{\subsm{\mathcal{G}}{[i,j]}}  \subsm{\plus}{\limi{\subsm{\mathcal{G}}{i}}} \limi{\subsm{\mathcal{G}}{[i,k]}}$, for all $i,j , k\in \mathrm{I}$.
\end{itemize}
We define the functor $\tilde{\mathcal{G}}:\glI{I}\rightarrow \mathbf{Gdf}(\mathbf{C})$ such that $\tilde{\mathcal{G}}(a):=\limi{\subsm{\mathcal{G}}{a}}$ for all $a\in \Co{\glI{I}}$, and $\tilde{\mathcal{G}}(u)$ is the unique morphism from $\limi{\subsm{\mathcal{G}}{\text{dom}(u)}}$ to $\limi{\subsm{\mathcal{G}}{\text{codom}(u)}}$ for  all $u\in \Cn{\glI{I}}$ (see Lemma \ref{refine}). Then 
\begin{enumerate}
\item $\tilde{\mathcal{G}}\in \mathbf{Gdf}(\mathbf{C})$.
\item $\limi {{\mathcal{G}}}$ exists. More precisely,  $\limi {{\mathcal{G}}}=\tilde{\mathcal{G}}$;
\item When $\limi\;(\limi {{\mathcal{G}}})$ exists, then $\limi\;(\limi {{\mathcal{G}}})= \limi \tilde{{\mathcal{G}}}$.
\end{enumerate}
\end{proposition}

\begin{proof}
By definition, $\tilde{\mathcal{G}}$ is a functor from $\glI{I}$ to $\mathbf{Gdf}(\mathbf{C})$ that sends pushouts to pushouts, as shown in part (b) above. Therefore, it is a gluing data functor. Since $\tilde{\mathcal{G}}$ is the limit of $\mathcal{G}$ point-wise, we have $\tilde{\mathcal{G}}=\limi \mathcal{G}$. This concludes the proof.
\end{proof}

\section{Gluing topological spaces categorically}\label{joycat}
We refer to \cite{munkres} for  the topological background. In this section, we explore the process of gluing topological spaces using the formal language developed in the previous section. We also introduce alternative perspectives on glued-up objects that shed light on how certain required properties can be viewed as equivalent to the usual glued-up object.

In the context of our discussion on gluing in topological spaces, we now present a formal definition that generalize the conventional concept of gluing data for a topological space. The following definition introduces the collection of the components that are necessary to allow the gluing of topological spaces.  
\begin{definition}\label{tpsgd}
	A collection $$\left(\mathrm{I},\subsm{\{{\subsm{U}{i}}\}}{i\!\in\! \mathrm{I}},\subsm{\{\subsm{U}{i,j}, \subsm{\upsilon}{i,j} \}}{(i,j)\! \in\! \mathrm{I}^2}, \subsm{\{\subsm{\varphi}{i,j}, \subsm{\subsm{\varphi}{k}}{(i,j)}\}}{(i,j,k)\!\in\! \mathrm{I}^3}\right)$$ where 
	\begin{enumerate}
		\item $\mathrm{I}$ is a set;
		\item $\subsm{\{\subsm{U}{i}\}}{i\!\in\! \mathrm{I}}$ is a family of topological spaces;
		\item $\subsm{\{\subsm{U}{i,j}, \subsm{\upsilon}{i,j} \}}{(i,j)\! \in\! \mathrm{I}^2}$ is a family where $\subsm{U}{i,j}$ are topological spaces and $\subsm{\upsilon}{i,j} : \subsm{U}{i,j} \rightarrow \subsm{U}{i}$ are continuous maps, for all $i, j \in \mathrm{I}$; 
		\item $\subsm{\{\subsm{\varphi}{i,j}, \subsm{\subsm{\varphi}{k}}{(i,j)}\}}{(i,j,k)\!\in\! \mathrm{I}^3}$ is a family of continuous maps where $\subsm{\varphi}{i,j} : \subsm{U}{i,j} \rightarrow \subsm{U}{j,i}$ and $\subsm{\subsm{\varphi}{k}}{(i,j)}:\subsm{U}{i, j }\times_{U_i} \subsm{U}{i, k}\rightarrow \subsm{U}{j, i }\times_{U_j} \subsm{U}{j, i } $, for all $(i,j,k) \in \mathrm{I}^3$
	
	\end{enumerate}
	such that, for each $i,j,k\in \mathrm{I}$, 
	\begin{enumerate}
		\item[a)] $\subsm{U}{i,i}=\subsm{U}{i}$;
		\item[b)] $\subsm{\varphi}{i,i}=\subsm{\operatorname{id}}{\subsm{U}{i,i}};$ 
		
		\item [c)] $\subsm{\subsm{\varphi}{j}}{(i,k)}=\subsm{\subsm{\varphi}{i}}{(j,k)}\circc \subsm{\subsm{\varphi}{k}}{(i,j)}$.
		\item[d)] $\subsm{\mathfrak{i}}{\subsm{U}{j, i }\times_{U_j} \subsm{U}{j, i }, \subsm{U}{j,i}}\circ \subsm{\varphi_{k}}{(i,j)}=\subsm{\varphi}{i,j}\circ \subsm{\mathfrak{i}}{\subsm{U}{i, j }\times_{U_i} \subsm{U}{i, k }, \subsm{U}{i,j}}$.
	\end{enumerate}
	is called a \textbf{\textit{topological space gluing data}}.
\end{definition}

\begin{remark}\label{inve}
	Given $\left(\mathrm{I},\subsm{\{{\subsm{U}{i}}\}}{i\!\in\! \mathrm{I}},\subsm{\{\subsm{U}{i,j}, \upsilon_{i,j} \}}{(i,j)\! \in\! \mathrm{I}^2}, \subsm{\{\subsm{\varphi}{i,j}, \subsm{\varphi_{k}}{(i,j)}\}}{(i,j,k)\!\in\! \mathrm{I}^3}\right)$ a topological space gluing data, $\subsm{\varphi}{i,j}$ is a homeomorphism from $\subsm{U}{i, j}$ to $\subsm{U}{j,i }$ whose inverse is $\subsm{\varphi}{j,i}$. This is a consequence of $c)$ applied to $k=i$ and $b)$.
\end{remark}

In the upcoming lemma, we explore the connection between $\topo^{\op}$-gluing data functors and topological space gluing data. We prove the equivalence between these two constructs, shedding light on how the abstract notion of gluing in category theory aligns with the concrete gluing of topological spaces. This insight bridges the gap between theory and practice, providing a powerful tool for both abstract mathematical exploration and practical applications in topology.
\begin{lemma}\label{equi}
	An $\topo^{\op}$-gluing data functor $\mathbf{G}$ induces the topological space gluing data $$\left(\mathrm{I}, \subsm{\{\Goi{\mbf{G}}{i}\}}{i\!\in\! \mathrm{I}},\subsm{\{\Goij{\mbf{G}}{i}{j}, \Gnij{\mbf{G}}{\etaij{i}{j}}^{\op}\}}{(i,j)\!\in\! \mathrm{I}^2},\subsm{\{\Gnij{\mbf{G}}{\tauij{i}{j}}^{\op},\Gnij{\mbf{G}}{\tauijk{k}{i}{j}{k}}^{\op}\}}{( i,j,k)\! \in\! \mathrm{I}^3}\right).$$ Conversely, a topological space gluing data $$\left(\mathrm{I},\subsm{\{{\subsm{U}{i}}\}}{i\!\in\! \mathrm{I}},\subsm{\{\subsm{U}{i,j}, \upsilon_{i,j} \}}{(i,j)\! \in\! \mathrm{I}^2}, \subsm{\{\subsm{\varphi}{i,j}, \subsm{\varphi_{k}}{(i,j)}\}}{(i,j,k)\!\in\! \mathrm{I}^3}\right)$$ induces the $\topo^{\op}$-gluing data functor $\mathbf{G}$ defined by $\Goi{\mbf{G}}{i}:= \subsm{U}{i}$, $\Goij{\mbf{G}}{i}{j}:=\subsm{U}{i,j}$,  $\Gnij{\mbf{G}}{\etaij{i}{j}}^{\op}:=\upsilon_{i,j}, \Gnij{\mbf{G}}{\tauij{i}{j}}^{\op}:=\subsm{\varphi}{i,j}$ and $\Gnij{\mbf{G}}{\tauijk{k}{i}{j}{k}}^{\op}:=\subsm{\varphi_{k}}{(i,j)}$ for all $i, j,k\in \mathrm{I}$.
	\begin{proof}
		Consider an $\topo^{\op}$-gluing data functor $\mathbf{G}$. We aim to prove that the collection $\left(\mathrm{I}, \subsm{\{\Goi{\mbf{G}}{i}\}}{i\!\in\! \mathrm{I}},\subsm{\{\Goij{\mbf{G}}{i}{j}, \Gnij{\mbf{G}}{\etaij{i}{j}}^{\op}\}}{(i,j)\!\in\! \mathrm{I}^2},\subsm{\{\Gnij{\mbf{G}}{\tauij{i}{j}}^{\op},\Gnij{\mbf{G}}{\tauijk{k}{i}{j}{k}}^{\op}\}}{( i,j,k)\! \in\! \mathrm{I}^3}\right)$ constitutes a topological space gluing data. According to the definition of a functor from $\glI{I}$ to $\topo^{\op}$, we know that $\subsm{\{\Goi{\mbf{G}}{i}\}}{i\!\in\! \mathrm{I}}$ represents a family of topological spaces, $\subsm{\{\Goij{\mbf{G}}{i}{j}\}}{(i,j)\!\in\! \mathrm{I}^2}$ corresponds to a family of topological spaces such that $\Gnij{\mbf{G}}{\etaij{i}{j}}^{\op}$ is a continuous map, and $\subsm{\{\Gnij{\mbf{G}}{\tauij{i}{j}}^{\op},\Gnij{\mbf{G}}{\tauijk{k}{i}{j}{k}}^{\op}\}}{( i,j,k)\! \in\! \mathrm{I}^3}$ denotes a family of continuous maps where  $\Gnij{\mbf{G}}{\tauij{i}{j}}^{\op}$ is a map from $\Goij{\mbf{G}}{i}{j}$ to $\Goij{\mbf{G}}{j}{i}$ and $\Gnij{\mbf{G}}{\tauijk{k}{i}{j}{k}}^{\op}$ is map from $\Goijk{\mbf{G}}{i}{j}{k}$ to $\Goijk{\mbf{G}}{j}{i}{k}$ for all $i,j,k\in \mathrm{ I}$. Now, we proceed to verify that conditions $a)$, $b)$, $c)$ and $d)$ of Definition \ref{tpsgd} are fulfilled.
		
		$a)$ can be deduced when applying the gluing data functor $\mathbf{G}$ to Remark \ref{reglue} $(1)$ $(a)$. Since a functor maps identities to identities, condition $b)$ is automatically satisfied. To establish condition $c)$, we apply the gluing data functor $\mathbf{G}$ to Remark \ref{reglue} $(1)$ $(c)$. Finally, condition $d)$ is obtained by applying $\mathbf{G}$ to Remark \ref{reglue} $(1)$ $(e)$.
		
		Conversely, let $\left(\mathrm{I},\subsm{\{{\subsm{U}{i}}\}}{i\!\in\! \mathrm{I}},\subsm{\{\subsm{U}{i,j}, \upsilon_{i,j} \}}{(i,j)\! \in\! \mathrm{I}^2}, \subsm{\{\subsm{\varphi}{i,j}, \subsm{\varphi_{k}}{(i,j)}\}}{(i,j,k)\!\in\! \mathrm{I}^3}\right)$ be a topological space gluing data. We want to define $\mathbf{G}$ as an $\topo^{\op}$-gluing data functor such that $\Goi{\mbf{G}}{i}:= \subsm{U}{i}$, $\Goij{\mbf{G}}{i}{j}:=\subsm{U}{i,j}$,  $\Gnij{\mbf{G}}{\etaij{i}{j}}^{\op}:=\upsilon_{i,j}, \Gnij{\mbf{G}}{\tauij{i}{j}}^{\op}:=\subsm{\varphi}{i,j}$ and $\Gnij{\mbf{G}}{\tauijk{k}{i}{j}{k}}^{\op}:=\subsm{\varphi_{k}}{(i,j)}$ for all $i, j,k\in \mathrm{I}$. To prove that $\mathbf{G}$ is well-defined, it suffices to show that $\mathbf{G}$ preserves the equalities $\subsm{\tau}{[i,i]}=\subsm{\operatorname{id}}{[i,i]}$,  $\tauijk{k}{i}{j}{k} \circ \tauijk{i}{j}{k}{i} = \tauijk{j}{i}{k}{j}$, and $\tauijk{k}{i}{j}{k} \circ \etaijk{i}{j}{i}{k}= \etaijk{j}{i}{j}{k} \circ \tauij{i}{j}$ for all $i, j,k\in \mathrm{I}$. This follows directly from Definition \ref{tpsgd} conditions $b)$, $c)$ and $d$, respectively. 
		
	\end{proof}
\end{lemma}

\subsection{Characterization of glued-up \texorpdfstring{${\protect\topo}^{\protect\op}$} \texorpdfstring{-}objects}

In the upcoming definition, we provide a tangible representation for the limit over ${\topo}^{\op}$-gluing data functors. This definition provides the concrete description that can be applied in topological settings. By constructing what we call 'the standard representative of the limit of $\mathbf{G}$,' we create a practical and complete framework for understanding and working with these limits. This representative allow us to manipulate and analyze limits effectively. It serves as a powerful tool for translating theoretical concepts into practical applications in the category of topological spaces.

\begin{definition}[Lemma]\label{deflema}
Let $\mathbf{G}$ be a $\topo^{\op}$-gluing data functor.
We define the \textbf{\textit{ standard representative of the limit of $\mathbf{G}$}} as the pair $(\subsm{Q}{\mathbf{G}}, \dindi{\iota}{Q}{\mathbf{G}}^{\op})$ where         
\begin{itemize}
    \item $\subsm{Q}{\mathbf{G}}:={\subsm{\coprod\nolimits}{i\!\in\! \mathrm{I}} \Goi{\mbf{G}}{i}}/\Rel{\mbf{G}}$ such that $\Rel{\mbf{G}}$ is the equivalence relation on the disjoint union $\subsm{\coprod\nolimits}{i\!\in\! \mathrm{I}} \Goi{\mbf{G}}{i}$ defined by $(x,i)\Rel{\mbf{G}} (y, j)$ if there exists $u\in \Goij{\mbf{G}}{i}{j}$ such that 
    $$x= \Gnij{\mbf{G}}{\etaij{i}{j}}^{\op} (u) \text{ and } y= \Gnij{\mbf{G}}{\etaij{j}{i}}^{\op} \circ \Gnij{\mbf{G}}{\tauij{i}{j}}^{\op}(u),$$ for any $(x,i), (y, j) \in \subsm{\coprod\nolimits}{i\!\in\! \mathrm{I}}\Goi{\mbf{G}}{i}$ where $i, j\in \mathrm{I}$. Moreover, $\subsm{Q}{\mathbf{G}}$ is a topological space via the final topology with respect to the family $\subsm{\{ {\dindiv{\iota}{Q}{\mbf{G}}{i}}\!\}}{i\! \in\! \mathrm{ I}}$. 
    \item $\dindi{\iota}{Q}{\mathbf{G}}^{\op}=\subsm{\{\dindiv{\iota}{Q}{\mbf{G}}{a}: \Goi{\mbf{G}}{a}\rightarrow {\subsm{Q}{\mathbf{G}}}\}}{a\!\in\! \Co{\glI{I}}}$, with ${\dindiv{\iota}{Q}{\mbf{G}}{i}}:= \pi \circ \subsm{\boldsymbol{\varepsilon}}{\Goi{\mbf{G}}{i},\subsm{\coprod\nolimits}{i\!\in\! \mathrm{I}} \Goi{\mbf{G}}{i}}$, ${\dindgij{\iota}{Q}{\mathbf{G}}{i}{j}} :={\dindiv{\iota}{Q}{\mbf{G}}{i}} \circ \Gnij{\mbf{G}}{\etaij{i}{j}}^{\op} $, and ${\dindgijk{\iota}{Q}{\mathbf{G}}{i}{j}{k}}:=\dindiv{\iota}{Q}{\mbf{G}}{i} \circ \Gnij{\mbf{G}}{\etaiijk{i}{j}{k}}^{\op} $ where $\subsm{\boldsymbol{\varepsilon}}{\Goi{\mbf{G}}{i},\subsm{\coprod\nolimits}{j\!\in\! \mathrm{I}} \Goi{\mbf{G}}{j}}$ is the canonical map from $ \Goi{\mbf{G}}{i}$ to $\subsm{\coprod\nolimits}{j\!\in\! \mathrm{I}} \Goi{\mbf{G}}{j}$ sending $x$ to $(x,i)$, and $\pi: \subsm{\coprod\nolimits}{i\!\in\! \mathrm{I}} \Goi{\mbf{G}}{i} \rightarrow {\subsm{Q}{\mathbf{G}}}$ is the quotient map, for all $i,j,k\in \mathrm{I}$.  
\end{itemize} 
For all $i \in \mathrm{ I}$, ${\dindiv{\iota}{Q}{\mbf{G}}{i}}\!\!: \Goi{\mbf{G}}{i}\rightarrow {\subsm{Q}{\mathbf{G}}}$ is a one-to-one continuous map. We will prove in Theorem \ref{gluingtop} that  $(\subsm{Q}{\mathbf{G}}, \dindi{\iota}{Q}{\mathbf{G}}^{\op})$ is indeed a limit over $\mathbf{G}$.
\end{definition}
\begin{proof} We first verify that $\Rel{\mbf{G}}$ is an equivalence relation. 
\begin{itemize}
\item We have that $\Rel{\mbf{G}}$ is reflexive since $\Goij{\mbf{G}}{i}{i}=\Goi{\mbf{G}}{i}$ and $\Gnij{\mbf{G}}{\etaij{i}{i}}^{\op} =\Gnij{\mbf{G}}{\tauij{i}{i}}^{\op}=\subsm{\operatorname{id}}{i}$, for $i\in \mathrm{I}$. 

\item We prove that $\Rel{\mbf{G}}$ is symmetric. Let $(x,i), (y, j) \in \subsm{\coprod\nolimits}{i\!\in\! \mathrm{I}}\Goi{\mbf{G}}{i}$ where $i, j\in \mathrm{I}$. Suppose that $(x,i)\Rel{\mbf{G}} (y,j)$. By definition of $\Rel{\mbf{G}}$, there exists $u\in \Goij{\mbf{G}}{i}{j}$ such that 
$$
x= \Gnij{\mbf{G}}{\etaij{i}{j}}^{\op} (u) \quad \text{and} \quad y= \Gnij{\mbf{G}}{\etaij{j}{i}}^{\op} \circ \Gnij{\mbf{G}}{\tauij{i}{j}}^{\op}(u).
$$
Taking $u':= \Gnij{\mbf{G}}{\tauij{i}{j}}^{\op}(u)$, we have $u'\in \Goij{\mbf{G}}{j}{i}$ and 
$$
y= \Gnij{\mbf{G}}{\etaij{j}{i}}^{\op} (u') \quad \text{and} \quad x= \Gnij{\mbf{G}}{\etaij{i}{j}}^{\op} \circ \Gnij{\mbf{G}}{\tauij{j}{i}}^{\op}(u').
$$

\item  We prove that $\Rel{\mbf{G}}$ is transitive. Let $(x,i),(y,j), (z,k) \in \subsm{\coprod\nolimits}{i\!\in\! \mathrm{I}}\Goi{\mbf{G}}{i}$  where $i, j,k \in \mathrm{I}$. Suppose that $(x,i)\Rel{\mbf{G}} (y,j)$  and $(y,j)\Rel{\mbf{G}} (z,k)$. By definition of $\Rel{\mbf{G}}$, there exist $u\in \Goij{\mbf{G}}{i}{j}, v\in  \Goij{\mbf{G}}{j}{k}$ such that 

\begin{itemize} 
\item $x= \Gnij{\mbf{G}}{\etaij{i}{j}}^{\op} (u)$ and $y= \Gnij{\mbf{G}}{\etaij{j}{i}}^{\op} \circ \Gnij{\mbf{G}}{\tauij{i}{j}}^{\op}(u)$, and
\item  $y= \Gnij{\mbf{G}}{\etaij{j}{k}}^{\op} (v)$ and $z= \Gnij{\mbf{G}}{\etaij{k}{j}}^{\op} \circ \Gnij{\mbf{G}}{\tauij{j}{k}}^{\op}(v)$. 
\end{itemize} 

By definition of the pullback $ \Goij{\mbf{G}}{j}{i}\subsm{\times}{\Goi{\mbf{G}}{j}}  \Goij{\mbf{G}}{j}{k}$, we have 
$$(\Gnij{\mbf{G}}{\tauij{i}{j}}^{\op}(u),v)\in  \Goij{\mbf{G}}{j}{i}\subsm{\times}{\Goi{\mbf{G}}{j}}  \Goij{\mbf{G}}{j}{k}.$$

For simplicity, without loss of generality, we can identify:

\begin{itemize} 
\item $\Goijk{\mbf{G}}{j}{i}{k}$ with the pullback $\Goij{\mbf{G}}{j}{i}\subsm{\times}{\Goi{\mbf{G}}{j}}  \Goij{\mbf{G}}{j}{k}$,
\item $\Goijk{\mbf{G}}{i}{j}{k}$ with the pullback $\Goij{\mbf{G}}{i}{j}\subsm{\times}{\Goi{\mbf{G}}{i}}  \Goij{\mbf{G}}{i}{k}$,
\item $\Goijk{\mbf{G}}{k}{j}{i}$ with the pullback $ \Goij{\mbf{G}}{k}{j}\subsm{\times}{\Goi{\mbf{G}}{k}}  \Goij{\mbf{G}}{k}{i}$,
\end{itemize}

which is possible since by definition $\mbf{G}$ sends pushouts to pushouts, and pushouts in the opposite category are pullbacks. 

By our assumptions, we know that $y= \Gnij{\mbf{G}}{\etaij{j}{i}}^{\op} \circ \Gnij{\mbf{G}}{\tauij{i}{j}}^{\op}(u)=\Gnij{\mbf{G}}{\etaij{j}{k}}^{\op} (v)$. We set $\alpha:=\Gnij{\mbf{G}}{\tauijk{k}{j}{i}{k}}^{\op}(\Gnij{\mbf{G}}{\tauij{i}{j}}^{\op}(u),v)$, and we know that $\alpha \in \Goijk{\mbf{G}}{i}{j}{k}$. Moreover,

$$
\begin{array}{lll} 
\subsm{\mathfrak{i}}{\Goijk{\mbf{G}}{i}{j}{k}, \Goij{\mbf{G}}{i}{j}}(\alpha) & = & \subsm{\mathfrak{i}}{\Goijk{\mbf{G}}{i}{j}{k}, \Goij{\mbf{G}}{i}{j}} \circ \Gnij{\mbf{G}}{\tauijk{k}{j}{i}{k}}^{\op}(\Gnij{\mbf{G}}{\tauij{i}{j}}^{\op}(u),v)\\
& = & \Gnij{\mbf{G}}{\tauij{j}{i}}^{\op}\circ \subsm{\mathfrak{i}}{\Goijk{\mbf{G}}{j}{i}{k}, \Goij{\mbf{G}}{j}{i}}(\Gnij{\mbf{G}}{\tauij{i}{j}}^{\op}(u),v)=u,
\end{array}
$$
by Remark \ref{reglue} $(1) (e)$.
Thus $\alpha=(u,\subsm{\mathfrak{i}}{\Goijk{\mbf{G}}{i}{j}{k}, \Goij{\mbf{G}}{i}{k}}(\alpha))$. 

Now, we set $w:=\subsm{\mathfrak{i}}{\Goijk{\mbf{G}}{i}{j}{k}, \Goij{\mbf{G}}{i}{k}}(\alpha)$, and we know that $w\in \Goij{\mbf{G}}{i}{k}$. By definition of the pullback $\Goijk{\mbf{G}}{i}{j}{k}$, we have $x=\Gnij{\mbf{G}}{\etaij{i}{j}}^{\op} (u)=\Gnij{\mbf{G}}{\etaij{i}{k}}^{\op} (w)$. 

Now we want to prove that $z=\Gnij{\mbf{G}}{\etaij{k}{i}}^{\op}\circ \Gnij{\mbf{G}}{\tauij{i}{k}}^{\op}(w)$. We have:
\begin{align*}
&\Gnij{\mbf{G}}{\etaij{k}{i}}^{\op}\circ \Gnij{\mbf{G}}{\tauij{i}{k}}^{\op}(w)\\
=&\Gnij{\mbf{G}}{\etaij{k}{i}}^{\op}\circ \Gnij{\mbf{G}}{\tauij{i}{k}}^{\op}\circ \subsm{\mathfrak{i}}{\Goijk{\mbf{G}}{i}{j}{k}, \Goij{\mbf{G}}{i}{k}}\circ \Gnij{\mbf{G}}{\tauijk{k}{j}{i}{k}}^{\op}(\Gnij{\mbf{G}}{\tauij{i}{j}}^{\op}(u),v)\\
=&\Gnij{\mbf{G}}{\etaij{k}{i}}^{\op}\circ \subsm{\mathfrak{i}}{\Goijk{\mbf{G}}{k}{j}{i}, \Goij{\mbf{G}}{k}{i}}\circ \Gnij{\mbf{G}}{\tauijk{j}{i}{k}{j}}^{\op} \circ  \Gnij{\mbf{G}}{\tauijk{k}{j}{i}{k}}^{\op}(\Gnij{\mbf{G}}{\tauij{i}{j}}^{\op}(u),v)\\
& \;\text{by Remark \ref{reglue} $(1) (e)$}\\
=&\Gnij{\mbf{G}}{\etaij{k}{i}}^{\op}\circ \subsm{\mathfrak{i}}{\Goijk{\mbf{G}}{k}{j}{i}, \Goij{\mbf{G}}{k}{i}}\circ \Gnij{\mbf{G}}{\tauijk{i}{j}{k}{i}}^{\op} (\Gnij{\mbf{G}}{\tauij{i}{j}}^{\op}(u),v)\\
&\;\text{by Remark \ref{reglue} $(1) (c)$}\\
=&\Gnij{\mbf{G}}{\etaij{k}{j}}^{\op}\circ \Gnij{\mbf{G}}{\tauij{j}{k}}^{\op}\circ  \subsm{\mathfrak{i}}{\Goijk{\mbf{G}}{j}{i}{k}, \Goij{\mbf{G}}{j}{k}} (\Gnij{\mbf{G}}{\tauij{i}{j}}^{\op}(u),v)\\
&\;\text{by Remark \ref{reglue} $(1) (e)$}\\
=&\Gnij{\mbf{G}}{\etaij{k}{j}}^{\op} \circ \Gnij{\mbf{G}}{\tauij{j}{k}}^{\op}(v)=z.
\end{align*}
\end{itemize}
\noindent This completes proving that $\Rel{\mbf{G}}$ is an equivalence relation.

Finally, we prove that the map ${\dindiv{\iota}{Q}{\mbf{G}}{i}}\!\!: \Goi{\mbf{G}}{i}\rightarrow {\subsm{Q}{\mathbf{G}}}$ is a one-to-one continuous map. Since $\subsm{Q}{\mathbf{G}}$ is a topological space under the final topology with respect to the family $\subsm{\{{\dindiv{\iota}{Q}{\mbf{G}}{i}}\!\}}{i\!\in\! \mathrm{I}}$, it follows that ${\dindiv{\iota}{Q}{\mbf{G}}{i}}$ is a continuous map, for all $i \in \mathrm{I}$. So, we only need to prove that ${\dindiv{\iota}{Q}{\mbf{G}}{i}}$ is one-to-one and open, for all $i \in \mathrm{I}$.

To establish the one-to-one property, consider $x,y\in \Goi{\mbf{G}}{i}$ such that ${\dindiv{\iota}{Q}{\mbf{G}}{i}}\!\!(x)={\dindiv{\iota}{Q}{\mbf{G}}{i}}\!\!(y)$. According to the definition of ${\dindiv{\iota}{Q}{\mbf{G}}{i}}$, we have $\pi(x,i)=\pi(y,i)$. Hence, by the definition of $\Rel{\mbf{G}}$, it follows that $x=y$, since $\Goij{\mbf{G}}{i}{i}=\Goi{\mbf{G}}{i}$ and $\Gnij{\mbf{G}}{\etaij{i}{i}}^{\op} =\Gnij{\mbf{G}}{\tauij{i}{i}}^{\op}=\subsm{\operatorname{id}}{i}$, for $i\in \mathrm{I}$.

\end{proof}

\begin{remark}\label{remtop}
	Let $\mathbf{G}$ be an $\topo^{\op}$-gluing data functor and $i \in \mathrm{ I}$. 
		 For all $q\in \subsm{Q}{\mathbf{G}}$ there exists $i\in \mathrm{ I}$ and $x\in \Goi{\mbf{G}}{i}$ such that $\dindiv{\iota}{Q}{\mbf{G}}{i}\!(x)=q$. Moreover, by definition of $\dindiv{\iota}{Q}{\mbf{G}}{i}$ we obtain $q=\pi(x,i)$.
	
\end{remark}

We introduce the 'Gluing Topological Spaces Theorem'. This theorem establishes the equivalence between a topological space obtained by gluing objects using the specific gluing data functor constructed above and the conventional construction of glued-up spaces. It elucidates conditions that allow us to view a 'glued-up' object as a limit of this functor, catering to a more intuitive understanding of an index gluing category. Furthermore, this theorem explores properties integral to comprehending these composite spaces. As we embark on the journey of proving this theorem, we shall construct topological spaces through the process of gluing, leaving no essential proof component hidden from view.

\begin{theorem} \label{gluingtop}
	Given a $\topo^{\op}$-gluing data functor $\mathbf{G}$. Let $Q$ be a topological space, $\dindi{\iota}{Q}{}^{\op}$ be a family $\subsm{\{\dindi{\iota}{Q}{a}^{\op}\}}{a\! \in\! \Co{\glI{I}}}$ where $\dindi{\iota}{Q}{a}\!:  \Goi{\mbf{G}}{a}\rightarrow Q$ are in $\Cn{\topo}$, for all $a \in \Co{\glI{I}}$. The following assertions are equivalent:
	\begin{enumerate}
		\item  $Q$ is a glued-up $\topo^{\op}$-object along $\mathbf{G}$ through $\dindi{\iota}{Q}{}^{\op}$;
		\item $(Q, \subsm{\iota}{Q}^{\op})$ is a cone over $\mathbf{G}$ isomorphic to $(\subsm{Q}{\mathbf{G}},\dindi{\iota}{Q}{\mathbf{G}}^{\op})$ in the category of cones over $\mathbf{G}$;		
		\item For all $i,j,k\in \mathrm{I}$ and  $n \in \{ j,k\}$, $(Q, \dindi{\iota}{Q}{}^\op)$ satisfies the following properties:
		\begin{itemize}
			\item[(a)] $\dindij{\iota}{Q}{i}{j} =\dindi{\iota}{Q}{i} \circ  \Gnij{\mbf{G}}{\etaij{i}{j}}^{\op}$;  
			\item[(b)] $\dindijk{\iota}{Q}{i}{j}{k} ={\iota_{Q_{[i,n]}}}  \circ \Gnij{\mbf{G}}{\etaijk{n}{i}{j}{k}}^{\op}$;
			\item[(c)]  $\dindi{\iota}{Q}{i} \circ \Gnij{\mbf{G}}{\etaij{i}{j}}^{\op} ={\dindi{\iota}{Q}{j}} \circ \Gnij{\mbf{G}}{\etaji{i}{j}}^{\op}  \circc\Gnij{\mbf{G}}{\tauij{i}{j}}^\op$;
			\item[(d)] $Q=\subsm{\cup}{i\!\in\! \mathrm{I}} {\dindsi{\iota}{Q}{i}}\!\!(\Goi{\mbf{G}}{i})$;
			\item[(e)] $\dindi{\iota}{Q}{j}\!(\Gnij{\mbf{G}}{\etaji{i}{j}}^{\op} (\Goij{\mbf{G}}{j}{i}))=\dindi{\iota}{Q}{i}\!(\Gnij{\mbf{G}}{\etaji{j}{i}}^{\op} (\Goij{\mbf{G}}{i}{j}))\!=\dindi{\iota}{Q}{i}\!(\Goi{\mbf{G}}{i})\cap {\dindi{\iota}{Q}{j}}\!(\Goi{\mbf{G}}{j})$;
			\item[(f)] $\dindi{\iota}{Q}{i}$ is a one-to-one continuous map.
		\end{itemize}	
	\end{enumerate}

	\begin{proof} 
		We start the proof proving that $(\subsm{Q}{\mathbf{G}},\dindi{\iota}{Q}{\mathbf{G}}^{\op})$ satisfies $(3)$ $(d)$ and $(e)$ as we will need it throughout the proof. 
		\begin{itemize} 
			\item We start by proving that $\subsm{Q}{\mathbf{G}}= \subsm{\cup}{i\!\in\! \mathrm{I}} {\dindiv{\iota}{Q}{\mbf{G}}{i}}\!\!(\Goi{\mbf{G}}{i})$. Let $i\in \mathrm{I}$. Using the definitions of ${\dindiv{\iota}{Q}{\mbf{G}}{i}}\!\!(\Goi{\mbf{G}}{i})$ and $\subsm{Q}{\mathbf{G}}$, we obtain $\subsm{Q}{\mathbf{G}}\subseteq \subsm{\cup}{i\!\in\! \mathrm{I}} {\dindiv{\iota}{Q}{\mbf{G}}{i}}\!\!(\Goi{\mbf{G}}{i})$. On the other hand, using the definition of $\subsm{\cup}{i\!\in\! \mathrm{I}} {\dindiv{\iota}{Q}{\mbf{G}}{i}}\!\!(\Goi{\mbf{G}}{i})$ and ${\dindiv{\iota}{Q}{\mbf{G}}{i}}\!\!(\Goi{\mbf{G}}{i})$, we get $\subsm{\cup}{i\!\in\! \mathrm{I}} {\dindiv{\iota}{Q}{\mbf{G}}{i}}\!\!(\Goi{\mbf{G}}{i})\subseteq \subsm{Q}{\mathbf{G}}$.

Hence, property $(3)(d)$ is satisfied.

\item We need to prove that $\dindi{\iota}{Q}{i} \circ \Gnij{\mbf{G}}{\etaij{i}{j}}^{\op} ={\dindi{\iota}{Q}{j}} \circ \Gnij{\mbf{G}}{\etaji{i}{j}}^{\op}  \circc\Gnij{\mbf{G}}{\tauij{i}{j}}$. Let $x \in \Goij{\mbf{G}}{i}{j}$. We have:

\begin{align*}
&\dindi{\iota}{Q}{i} \circ \Gnij{\mbf{G}}{\etaij{i}{j}}^{\op} (x) = \pi ( \Gnij{\mbf{G}}{\etaij{i}{j}}^{\op} (x), i), \text{ and } \\
&{\dindi{\iota}{Q}{j}} \circ \Gnij{\mbf{G}}{\etaji{i}{j}}^{\op}  \circc\Gnij{\mbf{G}}{\tauij{i}{j}} (x) = \pi ( \Gnij{\mbf{G}}{\etaji{i}{j}}^{\op}  \circ\Gnij{\mbf{G}}{\tauij{i}{j}} (x), j)
\end{align*}

Therefore, we obtain the equality wanted by the definition of the equivalence relation $\Rel{\mbf{G}}$. Hence, property $(3)(c)$ is satisfied.

\item Finally, we prove that, for all $i,j \in \mathrm{I}$, we have:
\[
\dindi{\iota}{Q}{j}\!(\Gnij{\mbf{G}}{\etaji{i}{j}}^{\op} (\Goij{\mbf{G}}{j}{i}))=\dindi{\iota}{Q}{i}\!(\Gnij{\mbf{G}}{\etaji{j}{i}}^{\op} (\Goij{\mbf{G}}{i}{j}))\!=\dindi{\iota}{Q}{i}\!(\Goi{\mbf{G}}{i})\cap {\dindi{\iota}{Q}{j}}\!(\Goi{\mbf{G}}{j}).
\]
Let $i,j\in \mathrm{I}$ and $x\in \dindi{\iota}{Q}{j}\!(\Gnij{\mbf{G}}{\etaji{i}{j}}^{\op} (\Goij{\mbf{G}}{j}{i}))$, there exists $u\in \Goij{\mbf{G}}{j}{i}$ such that:

\[
\begin{array}{lll} x&=&\dindi{\iota}{Q}{j}\!(\Gnij{\mbf{G}}{\etaji{i}{j}}^{\op}(u))=\pi(\Gnij{\mbf{G}}{\etaji{i}{j}}^{\op}(u),j)\\
&=&\pi(\Gnij{\mbf{G}}{\etaaji{i}{j}}^{\op}\circ \Gnij{\mbf{G}}{\tauij{j}{i}}(u),j)=\pi( \Gnij{\mbf{G}}{\etaij{i}{j}}^{\op}(v),i)
\end{array}
\]
\noindent
where $v\in \Gnij{\mbf{G}}{\tauij{j}{i}}(u)\in \Goij{\mbf{G}}{i}{j}$. Hence we obtain $x\in \dindi{\iota}{Q}{i}\!(\Gnij{\mbf{G}}{\etaji{j}{i}}^{\op} (\Goij{\mbf{G}}{i}{j}))$. We can prove similarly the reverse inclusion. Let $x\in \dindi{\iota}{Q}{i}\!(\Gnij{\mbf{G}}{\etaji{j}{i}}^{\op} (\Goij{\mbf{G}}{i}{j})$, we know that there exists $u\in \Goij{\mbf{G}}{i}{j})$ such that $x=\dindi{\iota}{Q}{i}\!(\Gnij{\mbf{G}}{\etaji{j}{i}}^{\op}(u))$.

We have $x\in \dindi{\iota}{Q}{i}(\Goi{\mbf{G}}{i})$ since $\Gnij{\mbf{G}}{\etaji{j}{i}}^{\op}(u)\in \Goi{\mbf{G}}{i}$. Let $x\in \dindi{\iota}{Q}{j}\!(\Gnij{\mbf{G}}{\etaji{i}{j}}^{\op} (\Goij{\mbf{G}}{j}{i}))$ then there exists $v\in \Goij{\mbf{G}}{j}{i}$ such that $x=\dindi{\iota}{Q}{j}\!(\Gnij{\mbf{G}}{\etaji{i}{j}}^{\op}(v))$

This implies that $x\in \dindi{\iota}{Q}{j}(\Goi{\mbf{G}}{j})$ since $\Gnij{\mbf{G}}{\etaji{i}{j}}^{\op}(v)\in \Goi{\mbf{G}}{j}$. Therefore, $x\in \dindi{\iota}{Q}{i}\!(\Goi{\mbf{G}}{i})\cap {\dindi{\iota}{Q}{j}}\!(\Goi{\mbf{G}}{j})$. Conversely, suppose that $x\in \dindi{\iota}{Q}{i}\!(\Goi{\mbf{G}}{i})\cap {\dindi{\iota}{Q}{j}}\!(\Goi{\mbf{G}}{j})$. There exists $u\in \Goi{\mbf{G}}{i}$ and $v\in \Goi{\mbf{G}}{j}$ such that $x=\pi(u,i)=\pi(v,j)$. That is, $(u,i)\Rel{\mbf{G}} (v,j)$. By the definition of the relation $\Rel{\mbf{G}}$, we know that there exists $w\in \Goij{\mbf{G}}{i}{j}$ such that $u= \Gnij{\mbf{G}}{\etaij{i}{j}}^{\op} (w)$ and $v= \Gnij{\mbf{G}}{\etaij{j}{i}}^{\op} \circ \Gnij{\mbf{G}}{\tauij{i}{j}}^{\op}(w)$. We conclude that $x\in \dindi{\iota}{Q}{i}\!(\Gnij{\mbf{G}}{\etaji{j}{i}}^{\op} (\Goij{\mbf{G}}{i}{j}))$. This completes the proof of property $(3)(e)$.

		\end{itemize}
		\begin{enumerate}
			\item[$(1)\Rightarrow (2)$] 
			Let $i\in\mathrm{I}$. 
			We prove that $(\subsm{Q}{\mathbf{G}},\dindi{\iota}{Q}{\mathbf{G}}^{\op})$ is a terminal cone over $\mathbf{G}$. 
			\begin{itemize} 			
				\item Our aim is to prove that the pair $(\subsm{Q}{\mathbf{G}},\dindi{\iota}{Q}{\mathbf{G}}^{\op})$ satisfies property $(1) (a)$, $(b)$, and $(c)$ of Remark \ref{oppc}. Property $(1) (a)$ holds by definition of $\dindgij{\iota}{Q}{\mbf{G}}{i}{j}$. Property $(1) (b)$ follows directly from the definition of ${\dindgij{\iota}{Q}{\mbf{G}}{i}{j}}$ and $\dindgijk{\iota}{Q}{\mathbf{G}}{i}{j}{k}$. Property $(1) (c)$ has been prove above.
			
				\item Finally, suppose that $(Q',{\subsm{\iota}{{Q'}}})$ is another pair  making the following diagram
				\begin{figure}[H] 
					\begin{center}
						{\tiny	\begin{tikzcd}[column sep=normal]
								&\Goij{\mbf{G}}{i}{j}\arrow[]{rrrr}{\Gnij{\mbf{G}}{\etaji{i}{j}}^{\op} \circ \Gnij{\mbf{G}}{\tauij{i}{j}}^{\op}} \arrow[swap]{d}{\Gnij{\mbf{G}}{\etaij{i}{j}}^{\op}}& && &\Goi{\mbf{G}}{j}\arrow[]{d} {{\dindpi{\iota}{{Q}}{\!j}}}\\ 
								& \Goi{\mbf{G}}{i} \arrow[swap]{rrrr}{\dindpi{\iota}{{Q}}{i}}&& & &Q'
					\end{tikzcd}}\end{center} 	 \caption{}\label{topol130}	
				\end{figure}
				\noindent commute, for all $i,j\in \mathrm{I}$. We want to prove that there exists a unique map $\mu:\subsm{Q}{\mathbf{G}}\rightarrow Q'$ in $\Cn{\mathbf{\topo}}$ making the following diagram
				\begin{figure}[H]
					\begin{center}
						{\tiny  \begin{tikzcd}[column sep=normal]
								\Goij{\mbf{G}}{i}{j}\arrow[swap]{d}{\Gnij{\mbf{G}}{\etaij{i}{j}}^{\op}}  \arrow[]{rrrr}{{\Gnij{\mbf{G}}{\etaji{i}{j}}^{\op}\circ \Gnij{\mbf{G}}{\tauij{i}{j}}^{\op}}}  && & &
								\Goi{\mbf{G}}{j}\arrow[bend left=20]{ddr}{{\dindpi{\iota}{{Q}}{j}}} \arrow[]{d}{{\dindiv{\iota}{Q}{\mbf{G}}{j}}}\\ 
								\Goi{\mbf{G}}{i}\arrow[swap,bend right=15]{drrrrr}{\dindpi{\iota}{{Q}}{i}}    \arrow[swap]{rrrr}{{\dindiv{\iota}{Q}{\mbf{G}}{i}}}  &&&&
								\subsm{Q}{\mathbf{G}}\arrow[dashed]{dr}{ \mu} \\ 
								&&&& &Q'
						\end{tikzcd} }
					\end{center}\caption{}\label{topol132}
				\end{figure}
				\noindent commutes, for all $i,j\in \mathrm{I}$. 
				\begin{itemize} 
					\item If such a $\mu$ exists we have $\mu\circ {\dindiv{\iota}{Q}{\mbf{G}}{i}}=\dindpi{\iota}{{Q}}{i}$, for all $i\in \mathrm{I}$. Let $q\in {\subsm{Q}{\mathbf{G}}}$. By Remark \ref{remtop}, there is $y\in \subsm{\coprod\nolimits}{i\!\in\! \mathrm{I}} \Goi{\mbf{G}}{i}$, such that $q= \pi (y)$. 
					Thus, we have $q= \dindiv{\iota}{Q}{\mbf{G}}{i}\!\!(x)$. Therefore, if $\mu$ exists, it is uniquely determined by $\mu(q)=\dindpi{\iota}{{Q}}{i}(x)$. 
					\item We will demonstrate the well-definedness of $\mu$. Let $i, j \in \mathrm{I}$, $x \in \Goi{\mbf{G}}{i}$, and $y \in \Goi{\mbf{G}}{j}$ such that $\dindiv{\iota}{Q}{\mbf{G}}{i}\!(x) = \dindiv{\iota}{Q}{\mbf{G}}{j}\!(y)$. We prove that $\dindpi{\iota}{{Q}}{i}(x)= \dindpi{\iota}{{Q}}{i}(y)$.
					
	By definition of $\dindiv{\iota}{Q}{\mbf{G}}{i}$ and $\dindiv{\iota}{Q}{\mbf{G}}{j}$ implies $\pi(x,i)=\pi(y,j)$. This implies that $(x,i)\Rel{\mbf{G}} (y,j)$. By definition of the relation $\Rel{\mbf{G}}$ there exists $u\in \Goij{\mbf{G}}{i}{j}$ such that \begin{equation}\label{rg1}x= \Gnij{\mbf{G}}{\etaij{i}{j}}^{\op} (u)\end{equation} and \begin{equation}\label{rg2}y= \Gnij{\mbf{G}}{\etaij{j}{i}}^{\op} \circ \Gnij{\mbf{G}}{\tauij{i}{j}}^{\op}(u).\end{equation} Applying $\dindpi{\iota}{Q}{i}$ to Equation \ref{rg1} and $\dindpi{\iota}{Q}{j}$ to Equation \ref{rg2} we obtain $$\dindpi{\iota}{Q}{i}(x)=\dindpi{\iota}{Q}{i}(\Gnij{\mbf{G}}{\etaij{i}{j}}^{\op} (u))$$ and $$\dindpi{\iota}{Q}{j}(y)=\dindpi{\iota}{Q}{j}(\Gnij{\mbf{G}}{\etaij{j}{i}}^{\op} \circ \Gnij{\mbf{G}}{\tauij{i}{j}}^{\op}(u)).$$ By commutativity of the diagram in Figure \ref{topol130} we get $\dindpi{\iota}{Q}{i}(x)=\dindpi{\iota}{Q}{j}(y)$.
					
					Hence, we can conclude that the map $\mu: Q \rightarrow Q'$ is indeed well-defined.
					
					\item We now prove that $\mu$ is continuous. Let $U\subsm{\subseteq}{\operatorname{op}} {Q'}$, we prove that \\$\mu^{-1}(U)\subsm{\subseteq}{\operatorname{op}} \subsm{Q}{\mathbf{G}}$. By definition of the final topology with respect to ${\dindiv{\iota}{Q}{\mbf{G}}{i}}$, we need to prove that $\dindiv{\iota}{Q}{\mbf{G}}{i}^{-1}\!(\mu^{-1}(U))\subsm{\subseteq}{\operatorname{op}} \Goi{\mbf{G}}{i}$ for all $i \in \mathrm{I}$. Let $i\in \mathrm{I}$. We know that $\dindiv{\iota}{Q}{\mbf{G}}{i}^{-1}\!(\mu^{-1}(U))=(\mu\circ {\dindiv{\iota}{Q}{\mbf{G}}{i}})^{-1}(U)= \dindpi{\iota}{Q}{i}^{-1} (U)$, by the definition of $\mu$. Since $\dindpi{\iota}{{Q}}{i}$ is a morphism in $\Cn{\topo}$, $\dindpi{\iota}{{Q}}{i}^{-1}(U)\subsm{\subseteq}{\operatorname{op}} \Goi{\mbf{G}}{i}$. Hence this proves that $\mu^{-1}(U)\subsm{\subseteq}{\operatorname{op}} \subsm{Q}{\mathbf{G}}$. Therefore $\mu$ is continuous.
					 
				\end{itemize}
			\end{itemize}
			
			This completes proving that $\subsm{Q}{\mathbf{G}}$ is a glued-up $\topo^{\op}$-object along $\mathbf{G}$ through $\dindi{\iota}{Q}{\mathbf{G}}^{\op}$. Finally, we prove (2) using the uniqueness of limits up to isomorphism. 
			\item[$(2) \Rightarrow (3)$] Suppose $(Q, \subsm{\iota}{Q}^{\op})$ is a cone over $\mathbf{G}$ isomorphic to $(\subsm{Q}{\mathbf{G}}, \dindi{\iota}{Q}{\mathbf{G}}^{\op})$ in the category of cones over $\mathbf{G}$. Since $(Q, \subsm{\iota}{Q}^{\op})\simeq (\subsm{Q}{\mathbf{G}}, \dindi{\iota}{Q}{\mathbf{G}}^{\op})$, we obtain that $(Q, \subsm{\iota}{Q}^{\op})$ satisfies $(3)$ $(a)$-$(f)$ is equivalent to proving that $(\subsm{Q}{\mathbf{G}}, \dindi{\iota}{Q}{\mathbf{G}}^{\op})$ satisfies $(3)$ $(a)$-$(f)$.  Since $ (\subsm{Q}{\mathbf{G}}, \dindi{\iota}{Q}{\mathbf{G}}^{\op})$ is a cone over $\mathbf{G}$, we have that properties $(3)$ $(a)$-$(c)$ are satisfied. Finally, properties $(3)$ $(d)$-$(f)$ have been proven above and this concludes the proof.
			\item[$(3)\Rightarrow (1)$] Suppose that $(Q, \dindi{\iota}{Q}{}^{\op})$ satisfies the properties $(a)$-$(f)$ of the statement $(3)$. We want to prove that $Q$ is glued-up $\topo^{\op}$-space along $\mathbf{G}$ through $\dindi{\iota}{Q}{}^{\op}$. Property $(1) (a)$ and $(1) (b)$ of the Remark \ref{oppc} are precisely the  properties $(a)$ and $(b)$ of statement $(3)$ of this theorem. 
			The following diagram
			\begin{figure}[H]\begin{center}
					{\tiny\begin{tikzcd}[column sep=normal]
							&\Goij{\mbf{G}}{i}{j}\arrow[]{rrrr}{{\Gnij{\mbf{G}}{\etaji{i}{j}}^{\op}\circ \Gnij{\mbf{G}}{\tauij{i}{j}}^{\op}}} \arrow[swap]{d}{\Gnij{\mbf{G}}{\etaij{i}{j}}^{\op}} & &&& \Goi{\mbf{G}}{j}\arrow[]{d} {{\dindi{\iota}{Q}{j}}}\\
							& \Goi{\mbf{G}}{i} \arrow[swap]{rrrr}{\dindi{\iota}{Q}{i}}&& & & Q
				\end{tikzcd}}\end{center}  \caption{}\label{topol122}   \end{figure}
			\noindent 
			commutes for all $i,j\in \mathrm{I}$, by property $(3) (c)$ of our assumptions.
			Next, let $(Q',{\subsm{\iota}{Q'}})$ be another pair as above, making the following diagram 
			\begin{figure}[H]\begin{center}
					{\tiny	\begin{tikzcd}[column sep=normal]
							&\Goij{\mbf{G}}{i}{j}\arrow[]{rrrr}{{\Gnij{\mbf{G}}{\etaji{i}{j}}^{\op}\circ \Gnij{\mbf{G}}{\tauij{i}{j}}^{\op}}} \arrow[swap]{d}{\Gnij{\mbf{G}}{\etaij{i}{j}}^{\op}} & &&& \Goi{\mbf{G}}{j}\arrow[]{d} {{\dindpi{\iota}{Q}{\!j}}}\\
							& \Goi{\mbf{G}}{i} \arrow[swap]{rrrr}{\dindpi{\iota}{Q}{\!i}}&& &&Q'
				\end{tikzcd}}\end{center}  \caption{}\label{topol123}   \end{figure}
			\noindent commute. We want to prove that there exists a unique map $\mu: Q\rightarrow Q'$ in $\Cn{\mathbf{\topo}}$ making the following diagram 
			\begin{figure}[H]
				\begin{center}
					{\tiny  \begin{tikzcd}[column sep=normal]
							\Goij{\mbf{G}}{i}{j} \arrow[]{rrrr}{{\Gnij{\mbf{G}}{\etaji{i}{j}}^{\op}\circ \Gnij{\mbf{G}}{\tauij{i}{j}}^{\op}}}\arrow[swap]{d}{\Gnij{\mbf{G}}{\etaij{i}{j}}^{\op}} && & &
							\Goi{\mbf{G}}{j}  \arrow[]{d}{{\dindi{\iota}{Q}{j}}} \arrow[bend left=20]{ddr}{{\dindpi{\iota}{{Q}}{j}}}\\ 
							\Goi{\mbf{G}}{i}\arrow[swap]{rrrr}{\dindi{\iota}{Q}{i}} \arrow[swap,bend right=15]{drrrrr}{\dindpi{\iota}{{Q}}{i}}    & && &
							Q\arrow[dashed,near start]{dr}{\exists!\mu} \\
							& & && &Q' 
					\end{tikzcd} }
				\end{center}\caption{}\label{topol14}
			\end{figure}
			\noindent commute, for all $i,j\in \mathrm{I}$. If such $\mu$ exists the commutativity of the diagram gives $\mu\circ \dindi{\iota}{Q}{i}\!(x)=\dindpi{\iota}{Q}{i}(x)$ for all $i\in \mathrm{I}$ and $x\in  \Goi{\mbf{G}}{i}$. Thus, $\mu$ is uniquely determined by $\dindpi{\iota}{Q}{i}$. Indeed, by property $(3) (d)$ of the assumption, given $q \in {Q}$, by Remark \ref{remtop}, there exists $i\in \mathrm{I}$ and $x\in \Goi{\mbf{G}}{i}$ such that $q= \dindi{\iota}{Q}{i}\!(x)$ and $\mu(q)=\dindpi{\iota}{Q}{i}(x)$. We prove that such a $\mu$ is a well-defined morphism in $\Cns{\topo}$ in the exact same way that we have proven these properties for the $\mu$ obtained in $(1)\Rightarrow (2)$   This completes the proof.
		\end{enumerate}
	\end{proof}
\end{theorem}
\begin{remark}
\begin{enumerate}	
\item Given a $\topo^{\op}$-gluing data functor $\mathbf{G}$ and $Q$ a topological space, let $\dindi{\iota}{Q}{}^{\op}$ be a family $\subsm{\{\dindi{\iota}{Q}{a}^{\op} \}}{a\!\in\! \Co{\glI{I}}}$ where $\dindi{\iota}{Q}{a}: \Goi{\mbf{G}}{a}\rightarrow Q$ is in $\Cn{\topo}$ for all $a\in \Co{\glI{I}}$. If $Q$ is a glued-up $\topo^{\op}$-object along $\mathbf{G}$ through $\dindi{\iota}{Q}{}^{\op}$, then the topology on $Q$ is the final topology with respect to $\dindi{\iota}{Q}{}^{\op}$. This follows directly from $(Q, \subsm{\iota}{Q}^{\op})\simeq (\subsm{Q}{\mathbf{G}}, \dindi{\iota}{Q}{\mathbf{G}}^{\op})$ proved in Theorem \ref{gluingtop}.

\item With the notation of Theorem \ref{gluingtop}, let $i,j ,k  \in \mathrm{I}$, we have 
\begin{align*}&{\dindiv{\iota}{Q}{\mbf{G}}{i}}\!(\Gnij{\mbf{G}}{\etaij{i}{j}}^{\op}\circ \Gnij{\mbf{G}}{\etaijk{j}{i}{j}{k}}^{\op}(\Goij{\mbf{G}}{i}{j} \subsm{\times}{\Goi{\mbf{G}}{i}} \Goij{\mbf{G}}{i}{k}))\\&= {\dindiv{\iota}{Q}{\mbf{G}}{j}}\!(\Gnij{\mbf{G}}{\etaij{j}{i}}^{\op}\circ  \Gnij{\mbf{G}}{\etaijk{i}{j}{i}{k}}^{\op}(\Goij{\mbf{G}}{j}{i} \subsm{\times}{\Goi{\mbf{G}}{j}}\Goij{\mbf{G}}{j}{k})) 
\\&= {\dindiv{\iota}{Q}{\mbf{G}}{i}}\!(\Gnij{\mbf{G}}{\etaij{i}{j}}^{\op}(\Goij{\mbf{G}}{i}{j}))\cap {\dindiv{\iota}{Q}{\mbf{G}}{i}}\!(\Gnij{\mbf{G}}{\etaij{i}{k}}^{\op}(\Goij{\mbf{G}}{i}{k}))\\&= {\dindiv{\iota}{Q}{\mbf{G}}{j}}\!(\Gnij{\mbf{G}}{\etaij{j}{i}}^{\op}(\Goij{\mbf{G}}{j}{i}))\cap {\dindiv{\iota}{Q}{\mbf{G}}{j}}\!(\Gnij{\mbf{G}}{\etaij{j}{k}}^{\op}(\Goij{\mbf{G}}{j}{k}))
\\&=  {\dindiv{\iota}{Q}{\mbf{G}}{i}}\!( \Goi{\mbf{G}}{i})\cap {\dindiv{\iota}{Q}{\mbf{G}}{j}}\!( \Goi{\mbf{G}}{j})\cap {\dindiv{\iota}{Q}{\mbf{G}}{k}}\!( \Goi{\mbf{G}}{k})
;\end{align*}  
Since $  \Gnij{\mbf{G}}{\tauijk{k}{i}{j}{k}}^{\op} (\Goij{\mbf{G}}{i}{j} \subsm{\times}{\Goi{\mbf{G}}{i}}  \Goij{\mbf{G}}{i}{k})= \Goij{\mbf{G}}{j}{i} \subsm{\times}{\Goi{\mbf{G}}{j}}  \Goij{\mbf{G}}{j}{k}$, then \begin{align*} &{\dindiv{\iota}{Q}{\mbf{G}}{i}}\!(\Gnij{\mbf{G}}{\etaij{i}{j}}^{\op}\circ \Gnij{\mbf{G}}{\etaijk{j}{i}{j}{k}}^{\op}(\Goij{\mbf{G}}{i}{j} \subsm{\times}{\Goi{\mbf{G}}{i}} \Goij{\mbf{G}}{i}{k}))\\&={\dindiv{\iota}{Q}{\mbf{G}}{j}}\!\!\circ \Gnij{\mbf{G}}{\etaij{j}{i}}^{\op} \Gnij{\mbf{G}}{\etaijk{i}{j}{i}{k}}^{\op} \circ \Gnij{\mbf{G}}{\tauijk{k}{i}{j}{k}}^{\op} (\Goij{\mbf{G}}{i}{j} \subsm{\times}{\Goi{\mbf{G}}{i}}  \Goij{\mbf{G}}{i}{k})\\&={\dindiv{\iota}{Q}{\mbf{G}}{j}}\!(\Gnij{\mbf{G}}{\etaij{j}{i}}^{\op}\circ  \Gnij{\mbf{G}}{\etaijk{i}{j}{i}{k}}^{\op}(\Goij{\mbf{G}}{j}{i} \subsm{\times}{\Goi{\mbf{G}}{j}}\Goij{\mbf{G}}{j}{k})).\end{align*} Moreover, we can prove as in the proof of the Theorem \ref{gluingtop} that
\begin{align*}&{\dindiv{\iota}{Q}{\mbf{G}}{i}}\!(\Gnij{\mbf{G}}{\etaij{i}{j}}^{\op}\circ \Gnij{\mbf{G}}{\etaijk{j}{i}{j}{k}}^{\op}(\Goij{\mbf{G}}{i}{j} \subsm{\times}{\Goi{\mbf{G}}{i}} \Goij{\mbf{G}}{i}{k}))\\&={\dindiv{\iota}{Q}{\mbf{G}}{i}}\!(\Gnij{\mbf{G}}{\etaij{i}{j}}^{\op}(\Goij{\mbf{G}}{i}{j}))\cap {\dindiv{\iota}{Q}{\mbf{G}}{i}}\!(\Gnij{\mbf{G}}{\etaij{i}{k}}^{\op}(\Goij{\mbf{G}}{i}{k})).\end{align*}  and that 
\begin{align*}&{\dindiv{\iota}{Q}{\mbf{G}}{j}}\!(\Gnij{\mbf{G}}{\etaij{j}{i}}^{\op}\circ  \Gnij{\mbf{G}}{\etaijk{i}{j}{i}{k}}^{\op}(\Goij{\mbf{G}}{j}{i} \subsm{\times}{\Goi{\mbf{G}}{j}}\Goij{\mbf{G}}{j}{k}))
\\&={\dindiv{\iota}{Q}{\mbf{G}}{j}}\!(\Gnij{\mbf{G}}{\etaij{j}{i}}^{\op}(\Goij{\mbf{G}}{j}{i})) \cap {\dindiv{\iota}{Q}{\mbf{G}}{j}}(\Gnij{\mbf{G}}{\etaij{j}{k}}^{\op}(\Goij{\mbf{G}}{j}{k})).\end{align*}
\end{enumerate} 
\end{remark}

We aim to illustrate the practical application of the processes we have described thus far by providing a straightforward and visually intuitive example. This example will help you grasp the tangible outcomes of our results.
\begin{example} \label{torusglu}
In topology, the concept of gluing allows us to seamlessly transform spaces, creating complex structures. To illustrate this concept, we will demonstrate how gluing can metamorphose basic shapes, such as a square, into more intricate ones, like a cylinder. Eventually, we will proceed to the process of gluing a cylinder into a torus, as visualized in the accompanying figure.\\
	\begin{center}
		\includegraphics[scale=0.3]{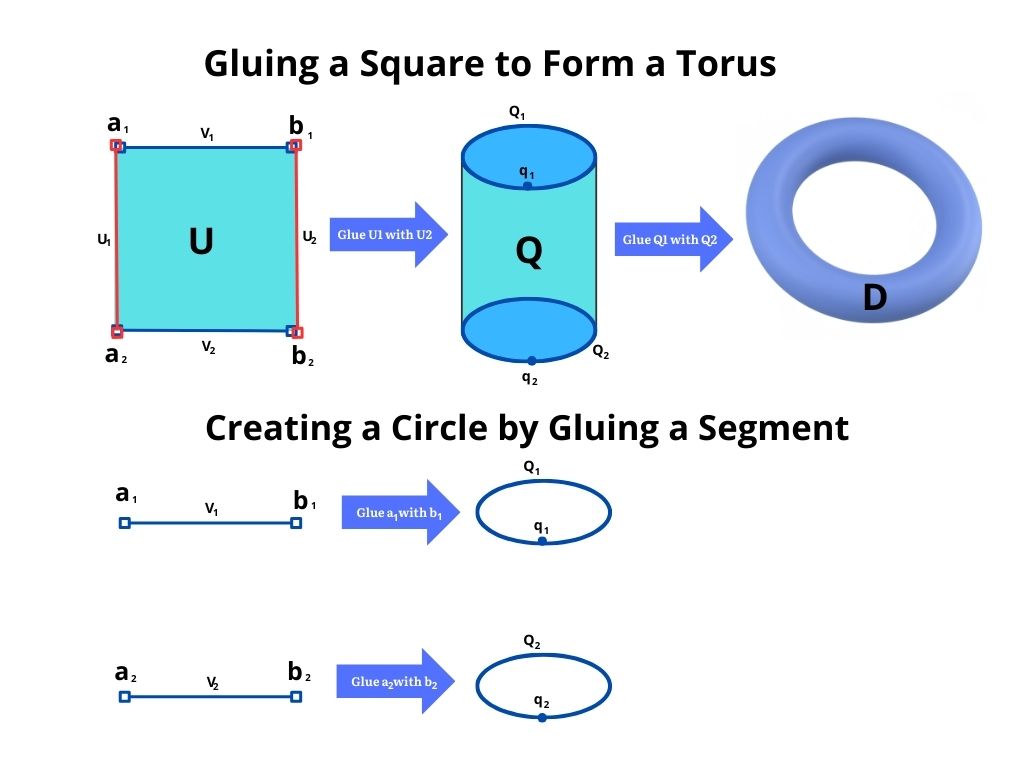}
	\end{center}
\textsf{Gluing Data:} To mathematically represent this gluing operation, we utilize a gluing data functor. Let's dissect the key elements of this process:

\begin{itemize}
  \item We commence with a topological space denoted as $U$, representing our square. This space is divided into four subspaces: $U_1$ and $U_2$, which correspond to opposite sides of the square, and $V_1$ and $V_2$, corresponding to the other pair of opposite sides, as depicted in the illustration. Additionally, we label the vertices of the square within $V_i$ as $a_i$ and $b_i$, where $i\in \{1,2\}$.
  
  \item To facilitate the gluing of the square into a cylinder, we introduce an isomorphism $\phi$ from $U_1$ to $U_2$.
  
  \item For further gluing, we introduce an isomorphism $\psi$ from $V_1$ to $V_2$.
  
  \item We consider another topological space, $Q$, which represents a cylindrical structure and two subspaces: $Q_1$ and $Q_2$, which correspond to the opposite circle on the top and the bottom as depicted in the illustration.
  
  \item To enable the gluing from a cylinder to a torus, we introduce an isomorphism $\theta$ from $Q_1$ to $Q_2$.
  
  \item Finally, we introduce a topological torus, denoted as $D$.
\end{itemize}

This setup forms the foundation for our gluing operations.

	{\sf Glued-Up Object as a Coequalizer:} The glued-up cylinder obtain by gluing $U$ along $U_1$ and $U_2$ can be expressed as a coequalizer of two morphisms:
	\begin{align*}
		&\iota_{U_1, U}: \text{Inclusion of } U_1 \text{ into the entire space } U. \\
		&\iota_{U_2, U} \circ \phi: \text{Inclusion of } U_2 \text{ into } U \text{ through the isomorphism } \phi.
	\end{align*}
	{\sf View as a Pushout} This coequalizer can also be interpreted as a pushout:
	\begin{itemize}
		\item We take the coproduct (disjoint union) of $U$ and $U_1$, denoted as $U \coprod U_1$.
		\item Introduce the canonical projection $\pi_{U \coprod U_1, U_1}$.
		\item The glued-up object becomes a pushout of two morphisms:
		\begin{align*}
			&\iota_{U_1, U} \circ \pi_{U \coprod U_1, U_1}; \\
			&\iota_{U_2, U} \circ \phi \circ \pi_{U \coprod U_2, U_2}.
		\end{align*}
	\end{itemize}
	
	{\sf Gluing data functor:}
	\begin{itemize}
		\item We set $\mathrm{I} = \{1, 2\}$. So that $\Co{\mathbf{Gl}(\mathrm{I})} = \{1, 2, [1, 2], [2, 1]\}$ and $\Cn{\mathbf{Gl}(\mathrm{I})} = \{\etaij{1}{2}, \etaij{2}{1},\tauij{1}{2}, \tauij{2}{1}\}$.\
		\item Now, we specify the gluing data functor to gluing the square into a cylinder $\mathbf{G}: \mathbf{Gl}(\mathrm{I}) \rightarrow \mathbf{Top}^\op$:
		\begin{align*}
			&\Goi{\mbf{G}}{1} = U, \ \Goi{\mbf{G}}{2} = U. \\
			&\Goij{\mbf{G}}{1}{2} = U \coprod U_1, \ \Goij{\mbf{G}}{2}{1} = U \coprod U_2. \\
			&\Gnij{\mbf{G}}{\etaij{1}{2}}^{\op} = \iota_{U_1, U} \circ \pi_{U \coprod U_1, U_1}, \ \Gnij{\mbf{G}}{\etaij{2}{1}}^{\op} = \iota_{U_2, U} \circ \pi_{U \coprod U_2, U_2}. \\
			&\Gnij{\mbf{G}}{\tauij{1}{2}}^{\op} = \operatorname{id} \coprod \phi, \ \Gnij{\mbf{G}}{\tauij{2}{1}}^{\op} = \operatorname{id} \coprod \phi^{-1}.
		\end{align*}
		\item $Q$ is a glued-up object over $\mathbf{G}$.
	\end{itemize}
	{\sf Refinements of $\mathbf{G}$:} 
	\begin{itemize}
		\item We can define two refinements of $\mathbf{G}$ as follows. For $i\in \{1,2\}$, we define the gluing data functor $\subsm{\mathbf{G}}{V_i}: \mathbf{Gl}(\mathrm{I}) \rightarrow \mathbf{Top}^\op$:
		\begin{align*}
			&\Goi{\subsm{\mathbf{G}}{V_i}}{1} = U \coprod V_i, \quad \Goi{\subsm{\mathbf{G}}{V_i}}{2} = U \coprod V_i. \\
			&\Goij{\subsm{\mathbf{G}}{V_i}}{1}{2} = U \coprod U_1 \coprod V_i \coprod \{ a_i\}, \quad \Goij{\subsm{\mathbf{G}}{V_i}}{2}{1} = U \coprod U_2 \coprod V_i \coprod \{ b_i\}. \\
			&\Gnij{\subsm{\mathbf{G}}{V_i}}{\etaij{1}{2}}^{\op} = \iota_{U_1  \coprod \{ a_i\}, U \coprod V_i} \circ \pi_{ U \coprod U_1 \coprod V_i \coprod \{ a_i\}, U_1 \coprod \{ a_i\}}, \\
			&\Gnij{\subsm{\mathbf{G}}{V_i}}{\etaij{2}{1}}^{\op} =  \iota_{U_2  \coprod \{ b_i\}, U \coprod V_i} \circ \pi_{ U \coprod U_2 \coprod V_i \coprod \{ b_i\},U_2  \coprod \{ b_i\}}. \\
			&\Gnij{\subsm{\mathbf{G}}{V_i}}{\tauij{1}{2}}^{\op} = \operatorname{id} \coprod \phi \coprod \operatorname{id} \coprod f_{a_i, b_i}, \\
			&\Gnij{\subsm{\mathbf{G}}{V_i}}{\tauij{2}{1}}^{\op} =\operatorname{id} \coprod \phi^{-1} \coprod \operatorname{id} \coprod f_{a_i, b_i}^{-1}.
		\end{align*}
		Moreover, we have $Q\coprod Q_i$ as a glued-up object over $\subsm{\mathbf{G}}{V_i}$, for all $i\in \{ 1,2\}$.
		
		
		\item We define the refinement map $\subsm{\rho}{\operatorname{id}, \subsm{\mathbf{G}}{V_i}, \mathbf{G}}:\subsm{\mathbf{G}}{V_i} \rightarrow \mathbf{G}$ so that $\Cn{\subsm{\rho}{\operatorname{id}, \subsm{\mathbf{G}}{V_i}, \mathbf{G}}} =\subsm{\subsm{\rho}{\operatorname{id}, \subsm{\mathbf{G}}{V_i}, \mathbf{G}}}{2}=  \pi_{U \coprod V_i, U}$, $\subsm{\subsm{\rho}{\operatorname{id}, \subsm{\mathbf{G}}{V_i}, \mathbf{G}}}{[1,2]}=  \pi_{U \coprod U_1 \coprod V_1 \coprod \{ a_i\}, U \coprod U_1}$, and $\subsm{\subsm{\rho}{\operatorname{id}, \subsm{\mathbf{G}}{V_i}, \mathbf{G}}}{[2,1]}=  \pi_{U \coprod U_2 \coprod V_1 \coprod \{ a_i\}, U \coprod U_2}$. One can prove without difficulty that $\subsm{\rho}{\operatorname{id}, \subsm{\mathbf{G}}{V_i}, \mathbf{G}}$ defines a natural transformation.  
		
		\item We define the refinement map $\subsm{\rho}{\operatorname{id}, \subsm{\mathbf{G}}{V_1}, \subsm{\mathbf{G}}{V_2}}:\subsm{\mathbf{G}}{V_1} \rightarrow \subsm{\mathbf{G}}{V_2}$ so that $\Cn{\subsm{\rho}{\operatorname{id}, \subsm{\mathbf{G}}{V_1}, \subsm{\mathbf{G}}{V_2}}} =\subsm{\subsm{\rho}{\operatorname{id}, \subsm{\mathbf{G}}{V_1}, \subsm{\mathbf{G}}{V_2}}}{2}=  \operatorname{id} \coprod \psi$, $\subsm{\subsm{\rho}{\operatorname{id}, \subsm{\mathbf{G}}{V_1}, \subsm{\mathbf{G}}{V_2}}}{[1,2]}= \operatorname{id} \coprod \operatorname{id} \coprod \psi \coprod f_{a_1, a_2}$, and $\subsm{\subsm{\rho}{\operatorname{id}, \subsm{\mathbf{G}}{V_1}, \subsm{\mathbf{G}}{V_2}}}{[2,1]}=  \operatorname{id} \coprod \operatorname{id} \coprod \psi \coprod f_{b_1, b_2}$. One can prove without difficulty that $\subsm{\rho}{\operatorname{id}, \subsm{\mathbf{G}}{V_1}, \subsm{\mathbf{G}}{V_2}}$ defines a natural correspondence.  
	\end{itemize}
	
	{\sf Combining gluing together into one gluing data functor:}  
	
	\begin{itemize}
		\item Finally, we can define the $\mathbf{Gdf}(\mathbf{Top}^\op)$-gluing data functor as follows. Now, let's specify the gluing index functor $\mathcal{G}: \mathbf{Gl}(\mathrm{I}) \rightarrow \mathbf{Gdf}(\mathbf{Top}^\op)$:
		\begin{align*}
			&\Goi{\mathcal{G}}{1} = \mathbf{G}, \quad \Goi{\mathcal{G}}{2} = \mathbf{G}. \\
			&\Goij{\mathcal{G}}{1}{2} = \subsm{\mathbf{G}}{V_1}, \quad \Goij{\mathcal{G}}{2}{1} = \subsm{\mathbf{G}}{V_2} \\
			&\Gnij{\mathcal{G}}{\etaij{1}{2}}^{\op} = \subsm{\rho}{\operatorname{id}, \subsm{\mathbf{G}}{V_1}, \mathbf{G}}, \quad \Gnij{\mathcal{G}}{\etaij{2}{1}}^{\op} = \subsm{\rho}{\operatorname{id}, \subsm{\mathbf{G}}{V_2}, \mathbf{G}}. \\
			&\Gnij{\mathcal{G}}{\tauij{1}{2}}^{\op} = \subsm{\rho}{\operatorname{id}, \subsm{\mathbf{G}}{V_1}, \subsm{\mathbf{G}}{V_2}}, \\
			&\Gnij{\mathcal{G}}{\tauij{2}{1}}^{\op} = \subsm{\rho}{\operatorname{id}, \subsm{\mathbf{G}}{V_1}, \subsm{\mathbf{G}}{V_2}}^{-1}.
		\end{align*}
		\item We can choose $ \lim \mathcal{G}: \mathbf{Gl}(\mathrm{I}) \rightarrow \mathbf{Top}^\op$ to be defined as follows:
		\begin{align*}
			&\Goi{\lim \mathcal{G}}{1} = Q, \quad \Goi{\lim \mathcal{G}}{2} = Q. \\
			&\Goij{\lim \mathcal{G}}{1}{2} = Q \coprod Q_1, \quad \Goij{\lim \mathcal{G}}{2}{1} = Q \coprod Q_2. \\
			&\Gnij{\lim \mathcal{G}}{\etaij{1}{2}}^{\op} = \iota_{Q_1, Q} \circ \pi_{Q \coprod Q_1, Q_1}, \\
			&\Gnij{\lim \mathcal{G}}{\etaij{2}{1}}^{\op} = \iota_{Q_2, Q} \circ \pi_{Q \coprod Q_2, Q_2}. \\
			&\Gnij{\lim \mathcal{G}}{\tauij{1}{2}}^{\op} = \operatorname{id} \coprod \theta, \quad \Gnij{\lim \mathcal{G}}{\tauij{2}{1}}^{\op} = \operatorname{id} \coprod \theta^{-1}.
		\end{align*} 
		\item $D$ is a glued-up object over $\lim \mathcal{G}$.
	\end{itemize}
	
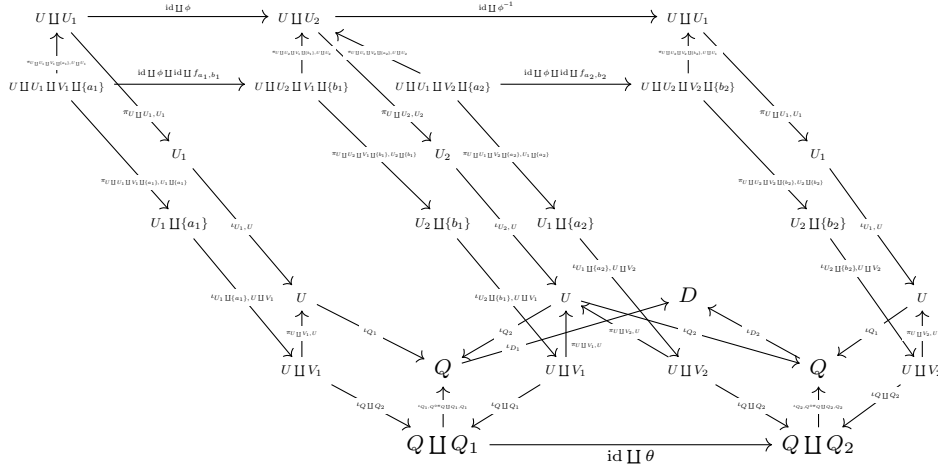
\begin{figure}[H]

{\tiny\begin{tikzcd}[column sep=small]
\scalebox{0.6}{$U\coprod U_1$} \arrow{rr}{\scalebox{0.5}{$\operatorname{id}\coprod \phi$}} \arrow[labels=description, near end]{rdd}{ \scalebox{0.5}{${\subsm{\pi}{U\coprod U_1,U_1}}$}}                                                                                                                 &                                                                                 & \scalebox{0.6}{$U\coprod U_2$ }\arrow[near end, labels=description]{rdd}{\scalebox{0.5}{${\subsm{\pi}{U\coprod U_2,U_2}}$}} \arrow{rrr}{\scalebox{0.5}{$\operatorname{id}\coprod \phi^{-1}$}}                     &                                                                                                                                                                                                                  &                                                                                & \scalebox{0.6}{$U\coprod U_1$} \arrow[near end,labels=description]{rdd}{\scalebox{0.5}{${\subsm{\pi}{U\coprod U_1,U_1}}$}}                                                                       &                                                                                 &                                   \\
\scalebox{0.5}{$U\coprod U_1\coprod V_1\coprod \{a_1\}$} \arrow[labels=description, near start]{u}{\scalebox{0.3}{$\subsm{\pi}{U \coprod U_1 \coprod V_1 \coprod \{ a_1\}, U \coprod U_1}$}} \arrow{rr}{\scalebox{0.5}{${\operatorname{id}\coprod \phi\coprod \operatorname{id}\coprod f_{a_1,b_1}}$}}\arrow[swap,near end, labels=description]{rdd}{ \scalebox{0.4}{${\subsm{\pi}{U\coprod U_1\coprod V_1\coprod \{a_1\},U_1\coprod \{a_1\}}}$}} &                                                                                 & \scalebox{0.5}{$U\coprod U_2\coprod V_1\coprod \{b_1\} \arrow[labels=description]{u}{\scalebox{0.3}{$\subsm{\pi}{U \coprod U_2 \coprod V_1 \coprod \{ b_1\}, U \coprod U_2}$}}$} \arrow[swap,labels=description]{rdd}{\scalebox{0.4}{${\subsm{\pi}{U\coprod U_2\coprod V_1\coprod \{b_1\},U_2\coprod \{b_1\}}}$}} & \scalebox{0.5}{$U\coprod U_1\coprod V_2\coprod \{a_2\}$} \arrow{rr}{\scalebox{0.5}{${\operatorname{id}\coprod \phi\coprod \operatorname{id}\coprod f_{a_2,b_2}}$}}\arrow[swap, labels=description]{lu}{\scalebox{0.3}{$\subsm{\pi}{U \coprod U_1 \coprod V_2 \coprod \{ a_2\}, U \coprod U_2}$}} \arrow[labels=description]{rdd}{\scalebox{0.4}{${\subsm{\pi}{U\coprod U_1\coprod V_2\coprod \{a_2\},U_1\coprod \{a_2\}}}$}} &                                                                                & \scalebox{0.5}{$U\coprod U_2\coprod V_2\coprod \{b_2\}$} \arrow[labels=description]{u}{\scalebox{0.3}{$\subsm{\pi}{U \coprod U_2 \coprod V_2 \coprod \{ b_2\}, U \coprod U_1}$}} \arrow[near end,swap, labels=description]{rdd}{\scalebox{0.4}{${\subsm{\pi}{U\coprod U_2\coprod V_2\coprod \{b_2\},U_2\coprod \{b_2\}}}$}} &                                                                                 &                                   \\
                                                                                                                                                                                                                  & \scalebox{0.6}{$U_1$} \arrow[labels=description]{rdd}{\scalebox{0.5}{${\subsm{\iota}{U_1,U}}$}}                                       &                                                                                                                            &  \scalebox{0.6}{$U_2$} \arrow[labels=description]{rdd}{ \scalebox{0.5}{${\subsm{\iota}{U_2,U}}$}}                                                                                                                                                                       &                                                                                &                     &  \scalebox{0.6}{$U_1$} \arrow[labels=description]{rdd}{ \scalebox{0.5}{${\subsm{\iota}{U_1,U}}$}}                                       &                                   \\
                                                                                                                                                                                                                  &  \scalebox{0.6}{$U_1\coprod \{a_1\}$} \arrow[swap ,labels=description]{rdd}{ \scalebox{0.5}{${\subsm{\iota}{U_1\coprod \{a_1\},U\coprod V_1}}$}} &                                                                                                                            & \scalebox{0.6}{$U_2\coprod \{b_1\}$} \arrow[swap,labels=description]{rdd}{\scalebox{0.5}{${\subsm{\iota}{U_2\coprod \{b_1\},U\coprod V_1}}$}}                                                                                                                                 & \scalebox{0.6}{$U_1\coprod \{a_2\}$} \arrow[near start,labels=description]{rdd}{\scalebox{0.5}{${\subsm{\iota}{U_1\coprod \{a_2\},U\coprod V_2}}$}}&                                                                                                                            & \scalebox{0.6}{$U_2\coprod \{b_2\}$} \arrow[swap, near start, labels=description]{rdd}{\scalebox{0.5}{${\subsm{\iota}{U_2\coprod \{b_2\} \! ,\!U\coprod V_2} }$}} &                                   \\
                                                                                                                                                                                                                  &                                                                                 & \scalebox{0.6}{$U$} \arrow[labels=description]{rd}{\scalebox{0.5}{$\subsm{\iota}{Q_1}$}}                                                                                                              &                                                                                                                                                                                                                  & \scalebox{0.6}{$U$} \arrow[swap,labels=description]{ld}{\scalebox{0.5}{$\subsm{\iota}{Q_2}$}}\arrow[labels=description]{rrd}{\scalebox{0.5}{$\subsm{\iota}{Q_2}$}}                                                       & D                                                                                                                          &                                                                                 & \scalebox{0.6}{$U$} \arrow[swap,labels=description]{ld}{\scalebox{0.5}{$\subsm{\iota}{Q_1}$}}                     \\
                                                                                                                                                                                                                  &                                                                                 & \scalebox{0.6}{$U\coprod V_1$} \arrow[swap,labels=description]{u}{\scalebox{0.4}{$\subsm{\pi}{U \coprod V_1, U}$}} \arrow[swap,labels=description]{rd}{\scalebox{0.5}{$\subsm{\iota}{Q\coprod Q_2}$}}                                                                                          & Q \arrow[labels=description,near start]{rru}{\scalebox{0.5}{$\subsm{\iota}{ D_1}$}}                                                                                                                                                                                                  & \scalebox{0.6}{$U\coprod V_1$} \arrow[labels=description]{ld}{\scalebox{0.5}{$\subsm{\iota}{Q\coprod Q_1}$}} \arrow[swap,near start]{u}{\scalebox{0.4}{$\subsm{\pi}{U \coprod V_1, U}$}}                                              & \scalebox{0.6}{$U\coprod V_2$} \arrow[swap, labels=description]{rd}{\scalebox{0.5}{$\subsm{\iota}{Q\coprod Q_2}$}} \arrow[labels=description]{lu}{\scalebox{0.4}{$\subsm{\pi}{U \coprod V_2, U}$}}                                                                                       & Q \arrow[labels=description]{lu}{\scalebox{0.5}{$\subsm{\iota}{ D_2}$}}                                                                    & \scalebox{0.6}{$U\coprod V_2$} \arrow[swap,labels=description]{u}{\scalebox{0.4}{$\subsm{\pi}{U \coprod V_2, U}$}} \arrow[labels=description,near start]{ld}{\scalebox{0.5}{$\subsm{\iota}{Q\coprod Q_2}$}} \\
                                                                                                                                                                                                                  &                                                                                 &                                                                                                                            & Q\coprod Q_1 \arrow[labels=description]{u}{\scalebox{0.3}{$\iota_{Q_1, Q} \circ \pi_{Q \coprod Q_1, Q_1}$}} \arrow[swap]{rrr}{\operatorname{id}\coprod \theta}                                                                                                                                                                           &                                                                                &                                                                                                                            & Q\coprod Q_2 \arrow[labels=description]{u}{\scalebox{0.3}{$\iota_{Q_2, Q} \circ \pi_{Q \coprod Q_2, Q_2}$}}                                                          &                                  
\end{tikzcd}}

\caption{Representation of $\mathcal{G}$, where the top diagrams each represent $\mathbf{G}$, the bottom-left diagram represents $\subsm{\mathbf{G}}{V_1}$, the bottom-right diagram represents $\subsm{\mathbf{G}}{V_2}$, and the upward-pointing arrows denote the refinement maps.\\ }
\end{figure}
	
\end{example}

The usual gluing data for topological spaces, as discussed in \cite[Proposition $12.27$]{wedhorn2016}, aligns closer with the concept of a gluing data functor expressed in the following manner. 
\begin{lemma}
\label{otop}
Let $\mathbf{G}$ be an $\mathbf{oTop}^{\op}$-gluing data functor, where $\mathbf{oTop}$ is the subcategory of $\topo$ whose objects are topological spaces, and morphisms are open continuous maps. Then, for all $i \in \mathrm{I}$:
\begin{enumerate} 
\item ${\dindiv{\iota}{Q}{\mbf{G}}{i}}$ is a topological embedding;
\item $\dindi{\iota}{Q}{i}\!(\Goi{\mbf{G}}{i})$ is an open topological subspace of $\subsm{Q}{\mathbf{G}}$;
\item $\subsm{Q}{\mathbf{G}} = \subsm{\bigcup}{i\! \in\! \mathrm{I}} \dindi{\iota}{Q}{i}\!(\Goi{\mbf{G}}{i})$;
\item $\subsm{Q}{\mathbf{G}}$ is a glued-up $\mathbf{oTop}^{\op}$-object along $\mathbf{G}$ through $\dindi{\iota}{Q}{\mathbf{G}}^{\op}$.
\end{enumerate} 
\end{lemma}

\begin{proof}
Since $\mathbf{G}$ be an $\mathbf{oTop}^{\op}$-gluing data functor,  $\Gnij{\mbf{G}}{\etaij{i}{j}}$ is an open continuous map.
\begin{enumerate} 
\item We have already proven in Definition (Lemma) \ref{deflema} that $\dindi{\iota}{Q}{i}$ is a one-to-one continuous map. Now, let us prove that ${\dindiv{\iota}{Q}{\mbf{G}}{i}}$ is an open map. Consider $V\subsm{\subseteq}{\operatorname{op}} \Goi{\mbf{G}}{i}$. Our goal is to show that ${\dindiv{\iota}{Q}{\mbf{G}}{i}}\!\!(V)\subsm{\subseteq}{\operatorname{op}} \subsm{Q}{\mathbf{G}}$. In other words, we need to prove that ${\dindiv{\iota}{Q}{\mbf{G}}{j}^{-1}}\!({\dindiv{\iota}{Q}{\mbf{G}}{i}}\!\!(V))\subsm{\subseteq}{\operatorname{op}}\Goi{\mbf{G}}{j}$ for all $j\in \mathrm{I}$, considering the definition of the final topology on $\subsm{Q}{\mathbf{G}}$ with respect to $\dindi{\iota}{Q}{\mathbf{G}}$. Let $i,j \in \mathrm{I}$. To establish this, we prove that ${\dindiv{\iota}{Q}{\mbf{G}}{j}^{-1}}\!({\dindiv{\iota}{Q}{\mbf{G}}{i}}\!\!(V))= {\Gnij{\mbf{G}}{\etaij{j}{i}}^{\op}}(\Gnij{\mbf{G}}{\tauij{i}{j}}^{\op}( {\Gnij{\mbf{G}}{\etaij{i}{j}}^{\op}}^{-1} (V)))$. We have
	\begin{align*}
		&y\in {\Gnij{\mbf{G}}{\etaij{j}{i}}^{\op}}(\Gnij{\mbf{G}}{\tauij{i}{j}}^{\op}( {\Gnij{\mbf{G}}{\etaij{i}{j}}^{\op}}^{-1} (V)))  \\&\Leftrightarrow y={\Gnij{\mbf{G}}{\etaij{j}{i}}^{\op}}(\Gnij{\mbf{G}}{\tauij{i}{j}}^{\op}( z))\; \text{for some $z \in {\Gnij{\mbf{G}}{\etaij{i}{j}}^{\op}}^{-1} (V)$}  
		\\& \Leftrightarrow  \pi(y,j)=\pi(x,i)\; \text{for some $x\in V$ (take $x={{\Gnij{\mbf{G}}{\etaij{i}{j}}^{\op}}} (z)$)}  \\& \Leftrightarrow y\in {\dindiv{\iota}{Q}{\mbf{G}}{j}^{-1}}\!({\dindiv{\iota}{Q}{\mbf{G}}{i}}\!\!(V)), \; \text{since}\; \pi(x,i)={\dindiv{\iota}{Q}{\mbf{G}}{i}}\!\!(x)\; \\& \quad \quad \text{and}\; \pi(y,j)={\dindiv{\iota}{Q}{\mbf{G}}{j}}\!\!(y)\; \text{for some $x\in V $}.
	\end{align*}
	Since ${\Gnij{\mbf{G}}{\etaij{i}{j}}^{\op}}$ is continuous, ${\Gnij{\mbf{G}}{\etaij{i}{j}}^{\op}}^{-1} (V)$ is open. Given that $\Gnij{\mbf{G}}{\tauij{i}{j}}^{\op}$ is a homeomorphism, it follows that $\Gnij{\mbf{G}}{\tauij{i}{j}}^{\op}({\Gnij{\mbf{G}}{\etaij{i}{j}}^{\op}}^{-1}(V))$ is also open. Since ${\Gnij{\mbf{G}}{\etaij{j}{i}}^{\op}}$ is a morphism in $\Cn{\mathbf{oTop}}$, $ {\Gnij{\mbf{G}}{\etaij{j}{i}}^{\op}}(\Gnij{\mbf{G}}{\tauij{i}{j}}^{\op}( {\Gnij{\mbf{G}}{\etaij{i}{j}}^{\op}}^{-1} (V)))$ is open. Thus, successfully demonstrating that $\dindi{\iota}{Q}{i}$ is a topological embedding.
	\item[(2),(3)] We prove that $\subsm{\{ {\dindiv{\iota}{Q}{\mbf{G}}{i}}\!\!(\Goi{\mbf{G}}{i})\}}{i\! \in\! \mathrm{I}}$ is an open covering of $\subsm{Q}{\mathbf{G}}$. We already know by Theorem \ref{gluingtop}, that $\subsm{\{ {\dindiv{\iota}{Q}{\mbf{G}}{i}}\!\!(\Goi{\mbf{G}}{i})\}}{i\! \in\! \mathrm{I}}$ is a covering of $\subsm{Q}{\mathbf{G}}$. Let $i\in \mathrm{I}$. Since ${\dindiv{\iota}{Q}{\mbf{G}}{i}}$ is a topological embedding by $(1)$, ${\dindiv{\iota}{Q}{\mbf{G}}{i}}\!\!(\Goi{\mbf{G}}{i})$ is an open of $\subsm{Q}{\mathbf{G}}$, this proves the result.
	\item[(4)] We can proceed exactly as in the proof of Theorem \ref{gluingtop} $(1) \Rightarrow (2)$. Since we have already proven $(1)$, we only need to show that the map $\mu$ constructed in the proof of Theorem \ref{gluingtop} is also an open map.

Let $V \subseteq \subsm{Q}{\mathbf{G}}^\mathrm{op}$ be given. We have $V = \subsm{\cup}{i \in \mathrm{I}}(V \cap {\dindiv{\iota}{Q}{\mbf{G}}{i}}\!\!(\Goi{\mbf{G}}{i}))$, since $\subsm{Q}{\mbf{G}} = \subsm{\cup}{i \in \mathrm{I}}{\dindiv{\iota}{Q}{\mbf{G}}{i}}\!\!(\Goi{\mbf{G}}{i})$, and ${\dindiv{\iota}{Q}{\mbf{G}}{i}}\!\!(\Goi{\mbf{G}}{i})$ is open in $\subsm{Q}{\mathbf{G}}$ for all $i \in \mathrm{I}$, as proven above.

We want to prove that $\mu(V) = \subsm{\cup}{i \in \mathrm{I}} \mu(V \cap {\dindiv{\iota}{Q}{\mbf{G}}{i}}\!\!(\Goi{\mbf{G}}{i})) \subseteq Q'$. Let $i \in \mathrm{I}$. We have $\mu(V \cap {\dindiv{\iota}{Q}{\mbf{G}}{i}}\!\!(\Goi{\mbf{G}}{i})) = \dindpi{\iota}{{Q}}{i}({\dindiv{\iota}{Q}{\mbf{G}}{i}^{-1}}\!(V \cap {\dindiv{\iota}{Q}{\mbf{G}}{i}}\!\!(\Goi{\mbf{G}}{i})))$ by commutativity of the diagram in Figure \ref{topol14} and $\dindiv{\iota}{Q}{\mbf{G}}{i}$ is a surjective map onto $V \cap \dindiv{\iota}{Q}{\mbf{G}}{i}$.

Now, $V \cap {\dindiv{\iota}{Q}{\mbf{G}}{i}}\!\!(\Goi{\mbf{G}}{i})$ is open in $\subsm{Q}{\mathbf{G}}$ as an intersection of two opens. Moreover, ${\dindiv{\iota}{Q}{\mbf{G}}{i}^{-1}}\!(V \cap {\dindiv{\iota}{Q}{\mbf{G}}{i}}\!\!(\Goi{\mbf{G}}{i}))$ is open as ${\dindiv{\iota}{Q}{\mbf{G}}{i}}$ is continuous. Finally, $\mu(V \cap {\dindiv{\iota}{Q}{\mbf{G}}{i}}\!\!(\Goi{\mbf{G}}{i}))$ is open, since $\dindpi{\iota}{{Q}}{i}$ is an open map as it is in $\Cn{\mathbf{oTop}}$.

	\end{enumerate}
\end{proof}

\subsection{Grothendieck topology on certain topological space categories}
Based on Theorem \ref{gluingtop} and Lemma \ref{otop}, we can derive a natural concept of covering and open coverings. In the following result, we will elucidate how this notion is derived from gluing data functors and how it establishes a Grothendieck topology on the category of Grothendieck sites.
\begin{definition}
Let $U$ be a topological space, and let $i\in \mathrm{I}$. We define a \textbf{\textit{gluing (open) covering}} of $U$ as a family $\subsm{\{\iota_i: \subsm{U}{i}\rightarrow U\}}{i\!\in\! \mathrm{I}}$ satisfying the following conditions:
\begin{enumerate}
\item $\iota_i$ is a one-to-one (open) continuous map.
\item $\subsm{\cup}{i\!\in\!\mathrm{I}}\subsm{\iota}{i}(\subsm{U}{i})=U$.
\end{enumerate}
\end{definition}

We will demonstrate that we can establish a correspondence between the various types of gluing coverings and the gluing data functors within an associated category.
\begin{definition}
Let $U$ be a topological space, and $\subsm{\{\iota_i: U_i \rightarrow U\}}{i\!\in \!\mathrm{I}}$ be a gluing (open) covering of $U$.
\begin{enumerate}
\item We say that $\mathbf{G}$ is a $\mathbf{(o)Top}$-gluing data functor associated with $\subsm{\{\iota_i: U_i \rightarrow U\}}{i\!\in\! \mathrm{I}}$ if $\Goi{\mathbf{G}}{i} = U_i$ for all $i\in \mathrm{I}$ and $U$ is the glued object over $\mathbf{G}$ through $\iota$, where $\iota$ is the family $\subsm{\{\iota_a\}}{a\!\in \!\Co{\glI{I}}}$ such that for all $i,j,k\in \mathrm{I}$
\begin{itemize}
\item $\iota_{[i,j]} = \iota_i \circ \Gnij{\mathbf{G}}{\etaij{i}{j}}^{\op}$;
\item $\iota_{[i,j,k]} = \iota_{[i,j]} \circ \Gnij{\mathbf{G}}{\etaijk{j}{i}{j}{k}}^{\op}$.
\end{itemize}
We refer to $\iota$ as the $\glI{I}$-\textbf{\textit{family of maps associated with}} $\subsm{\{\iota_i: U_i \rightarrow U\}}{i\!\in\! \mathrm{I}}$.
\item We say that $(\mathrm{I},\subsm{\{U_i\}}{i\!\in\! \mathrm{I}}, \subsm{\{\upsilon_{j,i}, \varphi_{i,j}, \subsm{\subsm{\varphi}{k}}{(i,j)}\}}{(i,j,k)\!\in\! \mathrm{I}^3})$ is a \textbf{\textit{topological space gluing data (where all the maps are open maps) associated with}} $\subsm{\{\iota_i: U_i \rightarrow U\}}{i\!\in\! \mathrm{I}}$ if for all $i,j\in \mathrm{I}$
\begin{itemize}
\item $U = \subsm{\bigcup}{i\!\in\! \mathrm{I}} \iota_i(U_i)$;
\item $\iota_j(\upsilon_{j,i}(U_{j,i})) = \iota_i(\upsilon_{i,j}(U_{i,j})) = \iota_i(U_i) \cap \iota_j(U_j)$.
\end{itemize}
\end{enumerate}
\end{definition}

\begin{lemma} \label{coverings}
Let $U$ be a topological space. The following assertions are equivalent:
\begin{enumerate}
\item $\subsm{\{\iota_i: U_i \rightarrow U\}}{i\!\in\! \mathrm{I}}$ is a gluing (open) covering of $U$,
\item there is a $\mathbf{(o)Top}$-gluing data functor associated to $\subsm{\{\iota_i: U_i \rightarrow U\}}{i\!\in\! \mathrm{I}}$ such that $U$ is the glued-up object along $\mathbf{G}$ through the $\glI{I}$-family of maps associated with $\subsm{\{\iota_i: U_i \rightarrow U\}}{i\!\in\! \mathrm{I}}$,
\item there is a topological space gluing data $\big(\mathrm{I},\subsm{\{U_i\}}{i\!\in\! \mathrm{I}}, \subsm{\{\upsilon_{i,j}, \varphi_{i,j}, \subsm{\subsm{\varphi}{k}}{(i,j)}\}}{(i,j,k)\!\in\! \mathrm{I}^3}\big)$ (where all the maps are open) associated with $\{\iota_i: U_i \rightarrow U\}_{i\in \mathrm{I}}$ such that
\begin{itemize}
\item $U=\subsm{\bigcup}{i\!\in\! \mathrm{I}} \iota_i(U_i)$;
\item $\iota_j(\upsilon_{i,j}(U_{i,j}))=\iota_i(\upsilon_{j,i}(U_{j,i}))=\iota_i(U_i)\cap \iota_j(U_j)$
for all $i,j\in \mathrm{I}$.
\end{itemize}
\end{enumerate}
\end{lemma}

\begin{proof}
\begin{enumerate}
\item[$(2)\Leftrightarrow (3)$] Follows from Lemma \ref{equi}.
\item[$(3)\Rightarrow (1)$] Follows from Theorem \ref{gluingtop}.
\item[$(1)\Rightarrow (3)$] Let $i,j,k \in \mathrm{I}$. We set $\Uij{i}{j} = U_i \times_U U_j$ and $\subsm{U}{i \wedge j \wedge k}:= U_i \times_U U_j\times_U U_k$. We prove that
$$\big(\mathrm{I},\subsm{\{\subsm{U}{i}\}}{i\!\in \! \mathrm{I}}, \subsm{\{\Uij{i}{j}, \iuv{\Uij{i}{j}}{\subsm{U}{i}} \}}{(i,j)\! \in\! \mathrm{I}^2},\subsm{\{\subsm{\phi}{i,j}, \subsm{\phi_{k}}{(i,j)}\}}{(i,j,k)\!\in\! \mathrm{I}^3}\big)$$ is a topological space gluing data associated with the covering $\subsm{\{\iota_i: \subsm{U}{i}\rightarrow U\}}{i\!\in \! \mathrm{I}}$ where $\subsm{\phi}{i,j}:\Uij{i}{j}\rightarrow \Uij{j}{i}$ and $\subsm{\phi_{k}}{(i,j)}:\subsm{U}{i \wedge j \wedge k} \rightarrow \subsm{U}{j \wedge i \wedge k}$ are canonical isomorphisms between these pullbacks. We choose $\Uij{i}{i}$ to be $\subsm{U}{i}$ so that  $\subsm{\phi}{i,i}=\subsm{\operatorname{id}}{i}$, Definition \ref{tpsgd} $c)$ and $d)$ follow trivially from these pullback morphisms' property. We have $U=\subsm{\cup}{i\!\in\! \mathrm{I}}\subsm{\iota}{i}(\subsm{U}{i})$ by definition of a gluing covering. Finally, for all $i,j\in \mathrm{I}$, we prove that 
$$\subsm{\iota}{i}(\iuv{\Uij{i}{j}}{\subsm{U}{i}}(\subsm{U}{i,j}))=\subsm{\iota}{j}(\iuv{\Uij{j}{i}}{\subsm{U}{j}} (\subsm{U}{j,i}))=\subsm{\iota}{i}(\subsm{U}{i})\cap \subsm{\iota}{j}(\subsm{U}{j}).$$
Let $i,j \in \mathrm{I}$, and $y\in \subsm{\iota}{i}(\iuv{\Uij{i}{j}}{\subsm{U}{i}}(\subsm{U}{i,j}))$. Then there exist $(u,v) \in \Uij{i}{j}$, such that $y = \subsm{\iota}{i}(\iuv{\Uij{i}{j}}{\subsm{U}{i}}(u,v))= \subsm{\iota}{i}(u)$. Since $(u,v) \in \Uij{i}{j}$, we have $\iota_i (u) = \iota_j (v)$, $u \in U_i$ and $v\in U_j$. Thus $y = \subsm{\iota}{j}(\iuv{\Uij{j}{i}}{\subsm{U}{j}}(v,u))= \iota_j (v)$. That is $y \in \subsm{\iota}{j}(\iuv{\Uij{j}{i}}{\subsm{U}{j}} (\subsm{U}{j,i}))$. The reverse inclusion is proven similarly by exchanging the roles of $i$ and $j$. From the above, we also obtain that if $y\in \subsm{\iota}{i}(\iuv{\Uij{i}{j}}{\subsm{U}{i}}(\subsm{U}{i,j}))$, then $y \in \subsm{\iota}{i}(\subsm{U}{i})\cap \subsm{\iota}{j}(\subsm{U}{j})$. For the reverse inclusion, let $y \in \subsm{\iota}{i}(\subsm{U}{i})\cap \subsm{\iota}{j}(\subsm{U}{j})$, then there is $u \in U_i$ and $v\in U_j$ such that $y=\iota_i (u) = \iota_j (v)$. Then $(u,v) \in \Uij{j}{i}$. Therefore, $y= \subsm{\iota}{i}(\iuv{\Uij{i}{j}}{\subsm{U}{i}}(u,v))\in \subsm{\iota}{i}(\iuv{\Uij{i}{j}}{\subsm{U}{i}}(\subsm{U}{i,j}))$. Finally, all the pullback maps are open by definition of the topology on a pullback.
\end{enumerate}
\end{proof}

The following Definition (Lemma) directly arises from the proofs of Lemma \ref{coverings}, Lemma \ref{equi}, and Theorem \ref{gluingtop}. It provides a canonical gluing data functor associated with a gluing covering.
\begin{definition}[Lemma] \label{gluingcover}
Let $U$ be a topological space, and $\mathcal{U}:=\{\iota_i: U_i \rightarrow U\}_{i\in \mathrm{I}}$ be a gluing covering of $U$. We define the $\mathbf{(o)Top}^{\op}$-gluing data functor $\gcov{\mathcal{U}}: \glI{I} \rightarrow \mathbf{(o)Top}^{\op}$ as follows:
\begin{enumerate}
\item $\Coi{\gcov{\subsm{\mathcal{U}}{0}}}{i}=\subsm{U}{i}$;
\item $\Coij{\gcov{\subsm{\mathcal{U}}{0}}}{i}{j}=\Uij{i}{j}$;
\item $\Coijk{\gcov{\subsm{\mathcal{U}}{0}}}{i}{j}{k}=\Uijk{i}{j}{k}$;
\item $\Cnij{\gcov{\subsm{\mathcal{U}}{1}}}{\etaij{i}{j}}=\iuv{\Uij{i}{j}}{\subsm{U}{i}}^{\op}$;
\item $\Cnij{\gcov{\subsm{\mathcal{U}}{1}}}{\tauij{i}{j}}$ is the canonical isomorphism from $\Uij{i}{j}$ to $\Uij{j}{i}$;
\item $\Cnij{\gcov{\subsm{\mathcal{U}}{1}}}{\etaijk{n}{i}{j}{k}}=\iuv{\Uijk{i}{j}{k}}{\Uij{i}{n}}^{\op}$;
\end{enumerate}
where $\Uij{i}{j} = U_i \times_U U_j$ and $\subsm{U}{i \wedge j \wedge k}:= U_i \times_U U_j\times_U U_k$, for all $i,j,k \in \mathrm{I}$.
Furthermore, we define $\subsm{\mathfrak{i}}{\mathcal{U}}^{\op}:=\{\iuv{\Coi{\gcov{\subsm{\mathcal{U}}{0}}}{a}}{X}\}_{a\in\Co{\glI{I}}}$. Then, $\gcov{\mathcal{U}}$ is a $\mathbf{(o)Top}^{\op}$-gluing data functor, and $U$ is a glued-up $\mathbf{(o)Top}^{\op}$-object along ${\gcov{\mathcal{U}}}$ through $\mathfrak{i}_{\mathcal{U}}^{\op}$.
\end{definition}

The following definition establishes the collection of all gluing coverings within the category of topological spaces.
\begin{definition}
We define $\mathbf{Cov}_{\mathbf{Gl}}(\mathbf{(o)Top})$ to be the set of all gluing (open) coverings of all topological space.\end{definition}
We will now demonstrate that the collection of gluing coverings constitutes a Grothendieck site within the category of topological spaces.
\begin{theorem}\label{sitee}
 $(\mathbf{(o)Top}, \mathbf{Cov}_{\mathbf{Gl}}(\mathbf{(o)Top}))$ is a Grothendieck site
\end{theorem}
\begin{proof}
To establish that $(\mathbf{(o)Top}, \mathbf{Cov}_{\mathbf{Gl}}(\mathbf{(o)Top}))$ forms a Grothendieck site, we must demonstrate that $\mathbf{Cov}_{\mathbf{Gl}}(\mathbf{(o)Top})$ satisfies the criteria of a Grothendieck topology on $\mathbf{(o)Top}$. Specifically, we need to verify the following conditions:

\begin{enumerate}
    \item If $\varphi:V\rightarrow U$ is an isomorphism, then $\{\varphi:V\rightarrow U\}\in \mathbf{Cov}_{\mathbf{Gl}}(\mathbf{(o)Top})$.
    \item If $\{\subsm{\subsm{\iota}{i}: \subsm{U}{i}\rightarrow U\}}{i\!\in \! \mathrm{I}}\in \mathbf{Cov}_{\mathbf{Gl}}(\mathbf{(o)Top})$, and $\subsm{\{\subsm{\ell}{i,j}:\subsm{V}{i,j}\rightarrow \subsm{U}{i}\}}{i\!\in \! \mathrm{I}}\in \mathbf{Cov}_{\mathbf{Gl}}(\mathbf{(o)Top})$ for all $i\in \mathrm{I}$, then the family $\subsm{\{\subsm{\iota}{i}\circ \subsm{\ell}{i,j}:\subsm{V}{i,j}\rightarrow U\}}{i\!\in \! \mathrm{I}, j\in \subsm{\mathrm{J}}{i}}$ is in $\mathbf{Cov}_{\mathbf{Gl}}(\mathbf{(o)Top})$.
    \item If $\subsm{\{\subsm{\iota}{i}:\subsm{U}{i}\rightarrow U\}}{i\!\in \! \mathrm{I}}\in \mathbf{Cov}_{\mathbf{Gl}}(\mathbf{(o)Top})$, and $\varphi:V\rightarrow U$ is a morphism in $\Cn{\mathbf{(o)Top}}$, and $\subsm{U}{i}\subsm{\times}{U} V$ exists for all $i\in \mathrm{I}$, then $\subsm{\{\iuv{\subsm{U}{i}\subsm{\times}{U} V}{V}:\subsm{U}{i}\subsm{\times}{U} V\rightarrow V\}}{i\!\in \! \mathrm{I}}\in \mathbf{Cov}_{\mathbf{Gl}}(\mathbf{(o)Top})$.
\end{enumerate}
We now prove that $(1)$, $(2)$ and $(3)$ are satisfied.
\begin{enumerate}
    \item This condition holds trivially.
    \item We will demonstrate that the family $\subsm{\{\subsm{\iota}{i}\circ \subsm{\ell}{i,j}:\subsm{V}{i,j}\rightarrow U\}}{i\!\in \! \mathrm{I}, j\in \subsm{\mathrm{J}}{i}}$ belongs to $\mathbf{Cov}_{\mathbf{Gl}}(\mathbf{Top})$. Consider $i$ and $j$ from $\mathrm{I}$:
    \begin{itemize}
        \item $\subsm{\iota}{i}\circ \subsm{\ell}{i,j}$ is a one-to-one (open) continuous map since both $\subsm{\iota}{i}$ and $\subsm{\ell}{i,j}$ are one-to-one continuous maps, as per the definition of a covering.
        \item We have $U=\subsm{\cup}{i\!\in\! \mathrm{I}}\subsm{\iota}{i}(\subsm{U}{i})$ due to the fact that $\{\subsm{\subsm{\iota}{i}:\subsm{U}{i}\rightarrow U\}}{i\!\in \! \mathrm{I}}\in \mathbf{Cov}_{\mathbf{Gl}}(\mathbf{(o)Top})$. Now, for any $i\in \mathrm{I}$, $\subsm{U}{i}=\subsm{\cup}{j\!\in\! \subsm{\mathrm{J}}{i}}\subsm{\ell}{i,j}(\subsm{V}{i,j})$ as a consequence of $\subsm{\{\subsm{\ell}{i,j}:\subsm{V}{i,j}\rightarrow \subsm{U}{i}\}}{i\!\in \! \mathrm{I}}\in \mathbf{Cov}_{\mathbf{Gl}}(\mathbf{(o)Top})$. Combining these results, we can deduce that:
        $$U=\subsm{\cup}{i\!\in\! \mathrm{I}}\subsm{\iota}{i}(\subsm{U}{i})=\subsm{\cup}{i\!\in\! \mathrm{I}}\subsm{\iota}{i}(\subsm{\cup}{j\!\in\! \subsm{\mathrm{J}}{i}}\subsm{\ell}{i,j}(\subsm{V}{i,j}))=\subsm{\cup}{i\!\in\! \mathrm{I}, j\!\in\! \subsm{\mathrm{J}}{i}}\subsm{\iota}{i}\circ \subsm{\ell}{i,j}(\subsm{V}{i,j}).$$
    \end{itemize}
    \item Consider $\subsm{\{\subsm{\iota}{i}:\subsm{U}{i}\rightarrow U\}}{i\!\in \! \mathrm{I}}\in \mathbf{Cov}_{\mathbf{Gl}}(\mathbf{(o)Top})$, and let $\varphi:V\rightarrow U$ be a morphism in $\Cn{\mathbf{Top}}$. Additionally, assume that $\subsm{U}{i}\subsm{\times}{U} V$ exists for all $i\in \mathrm{I}$. Now, for any $i\in \mathrm{I}$, we have:
    \begin{itemize}
        \item $\iuv{\subsm{U}{i}\subsm{\times}{U} V}{V}$ is one-to-one. To establish this, consider two pairs $(u_1, v_1)$ and $(u_2, v_2) \in \subsm{U}{i}\subsm{\times}{U} V$ such that $\iuv{\subsm{U}{i}\subsm{\times}{U} V}{V} (u_1, v_1)=\iuv{\subsm{U}{i}\subsm{\times}{U} V}{V} (u_2, v_2)$, which is equivalent to $v_1=v_2$. Now, since $(u_1, v_1)$ and $(u_2, v_2) \in \subsm{U}{i}\subsm{\times}{U} V$, we have $\iota_i ( u_1) = \varphi( v_1)$ and $\iota_i ( u_2) = \varphi( v_2)$. Given that $v_1=v_2$, we conclude that $\iota_i ( u_1) = \varphi( v_1) = \varphi( v_2) = \iota_i ( u_2)$. Since $\iota_i$ is one-to-one, it follows that $u_1=u_2$, thus proving the result. When we are in $\mathbf{oTop}$,  $\iuv{\subsm{U}{i}\subsm{\times}{U} V}{V}$ is open by definition of the topology on a pullback.
        \item We now prove that $V=\subsm{\cup}{i\!\in\! \mathrm{I}}\iuv{\subsm{U}{i}\subsm{\times}{U} V}{V}(\subsm{U}{i}\subsm{\times}{U} V)$. The inclusion $\subsm{\cup}{i\!\in\! \mathrm{I}}\iuv{\subsm{U}{i}\subsm{\times}{U} V}{V}(\subsm{U}{i}\subsm{\times}{U} V)\subseteq V$ is trivial by definition. On the other hand, if $v\in V$, we define $\chi: U\subsm{\times}{U} V \rightarrow V$ as the canonical isomorphism that sends $(u,v)$ to $v$. Since $U= \subsm{\cup}{i\!\in\! \mathrm{I}}\subsm{\iota}{i}(\subsm{U}{i})$ due to the definition of a covering, we can write:
        $$U\subsm{\times}{U} V=(\subsm{\cup}{i\!\in\! \mathrm{I}}\subsm{\iota}{i}(\subsm{U}{i}))\subsm{\times}{U} V=\subsm{\cup}{i\!\in\! \mathrm{I}}(\subsm{\iota}{i}(\subsm{U}{i})\subsm{\times}{U} V).$$
        Therefore, there exists an $i\in \mathrm{I}$ such that $v\in \chi(\subsm{\iota}{i}(\subsm{U}{i})\subsm{\times}{U} V)$. In other words, there exists an $x\in \subsm{U}{i}$ such that $(\subsm{\iota}{i}(x),v) \in \subsm{\iota}{i}(\subsm{U}{i})\subsm{\times}{U} V$, and $v=\chi (\subsm{\iota}{i}(x),v)$. Finally, it can be observed that $v=\iuv{\subsm{U}{i}\subsm{\times}{U} V}{V}(x,v)$, as $(\subsm{\iota}{i}(x),v) \in \subsm{\iota}{i}(\subsm{U}{i})\subsm{\times}{U} V$.
    \end{itemize}
\end{enumerate}
\end{proof}

\begin{remark} \label{admia}
Let $X$ be a topological space. Consider the category $\subsm{\mathbf{Sp}}{X}$ ($\subsm{\mathbf{Ob}}{X}$), which is a subcategory of $\mathbf{(o)Top}$ with objects as (open) subspaces of $X$ and morphisms as inclusions. $Q_{\mathbf{G}}$ is not always a glued-up object over $\mathbf{G}$; However, the gluing data functor $\mathbf{G}$ that sends $[i,j]$ to $\Goi{\mbf{G}}{i}\cap \Goi{\mbf{G}}{j}$ admit $Q_{\mathbf{G}}$ as glued-up object. Moreover, the glued-up object over such functors is a covering in the usual sense, i.e., $\subsm{\{U_i\}}{i\!\in\! \mathrm{I}}$ where $U_i \subop X$ for all $i \in \mathrm{I}$ such that $U = \subsm{\cup}{i\! \in\! \mathrm{I}} U_i$.

In both cases, these coverings still form a Grothendieck topology in their respective categories. Furthermore, all these coverings are also coverings in $\mathbf{Cov}(\mathbf{Top})$.

Similarly, if we restrict the category of topological spaces to the subcategory $\mathbf{eTop}$ where objects are topological spaces and maps are topological embeddings, $Q_{\mathbf{G}}$ is not always a glued-up object over $\mathbf{G}$.

For all these categories, we would have that $\subsm{Q}{\mathbf{G}}$ will be a cone over $\mathbf{G}$, but the problem would be to prove it is a terminal cone. In $\mathbf{eTop}$, a map from $\subsm{Q}{\mathbf{G}}$ always exists as open continuous maps, as we have proven in Lemma \ref{otop}, but the map $\mu$ constructed in the proof of the Lemma might not necessarily be one-to-one. However, if we assume that there is a topological space $U$ and maps $\phi_i : \Goi{\mbf{G}}{i}\rightarrow U$ topological embeddings such that $\Goij{\mbf{G}}{i}{j}:= \Goi{\mbf{G}}{i}\times_{U} \Goi{\mbf{G}}{j}$, $ \Gnij{\mbf{G}}{\etaij{i}{j}}^{\op}:=\iuv{\Goij{\mbf{G}}{i}{j}}{\Goi{\mbf{G}}{i}}$, $ \Gnij{\mbf{G}}{\tauij{i}{j}}^{\op}$ is the canonical isomorphism from $\Goi{\mbf{G}}{i}\times_{U} \Goi{\mbf{G}}{j}$ to $\Goi{\mbf{G}}{j}\times_{U} \Goi{\mbf{G}}{i}$, and $\Gnij{\mbf{G}}{\etaijk{n}{i}{j}{k}}^{\op}:=\iuv{ \Goi{\mbf{G}}{i}\times_{U} \Goi{\mbf{G}}{j}\times_{U} \Goi{\mbf{G}}{k}}{\Goi{\mbf{G}}{i}\times_{U} \Goi{\mbf{G}}{n}}$ for all $i,j, k\in \mathrm{I}$ and $n\in \{ j,k\}$, we obtain $Q_{\mathbf{G}}$ is a glued-up object over such a gluing data functor $\mathbf{G}$.

Indeed, under these assumptions and using the notation as in the proof of Theorem \ref{gluingtop}, $(3) \Rightarrow (1)$, we can prove as done before, that for all $i,j \in \mathrm{I}$,
$${\dindpi{\iota}{Q}{i}{}} (  \iuv{\Goij{\mbf{G}}{i}{j}}{\Goi{\mbf{G}}{i}}(\Goij{\mbf{G}}{i}{j}))={\dindpi{\iota}{Q}{j}{}} (  \iuv{\Goij{\mbf{G}}{j}{i}}{\Goi{\mbf{G}}{j}}(\Goij{\mbf{G}}{j}{i}))={\dindpi{\iota}{Q}{i}{}} (\Goi{\mbf{G}}{i})\cap {\dindiv{\iota}{Q'}{j}{}} (\Goi{\mbf{G}}{j}).$$

We can deduce from these equalities that $\mu$ is a one-to-one map.

Let $q, q' \in \subsm{Q}{\mbf{G}}$ such that $\mu(q) = \mu(q')$. By property $(3)(d)$, we know that there exist $i, j \in \mathrm{I}$, $x \in \Goi{\mbf{G}}{i}$, and $y \in \Goi{\mbf{G}}{j}$ such that ${\dindiv{\iota}{Q}{\mbf{G}}{i}}(x) = q$ and ${\dindiv{\iota}{Q}{\mbf{G}}{j}}(y) = q'$. We choose such elements. It follows that $\mu({\dindiv{\iota}{Q}{\mbf{G}}{i}}(x)) = \mu ({\dindiv{\iota}{Q}{\mbf{G}}{j}}(y))$. By the definition of $(Q', \subsm{\iota}{Q'})$, we have ${\dindpi{\iota}{Q}{i}{}}(x) = {\dindpi{\iota}{Q}{j}{}}(y)$. We can deduce that ${\dindpi{\iota}{Q}{i}{}}(x) \in {\dindpi{\iota}{Q}{i}{}}(\Goi{\mbf{G}}{i})\cap {\dindpi{\iota}{Q}{j}{}}(\Goi{\mbf{G}}{j})={\dindpi{\iota}{Q}{i}{}} (  \iuv{\Goij{\mbf{G}}{i}{j}}{\Goi{\mbf{G}}{i}}(\Goij{\mbf{G}}{i}{j}))$.

Thus there is a $(v,w) \in \Goij{\mbf{G}}{i}{j}$ such that ${\dindpi{\iota}{Q}{i}{}}(  \iuv{\Goij{\mbf{G}}{i}{j}}{\Goi{\mbf{G}}{i}}(v,w))= {\dindpi{\iota}{Q}{i}{}}( v)= {\dindpi{\iota}{Q}{i}{}}(x) $. From the injectivity of ${\dindpi{\iota}{Q}{i}{}}$ we can deduce that $x=v$. Since 
$${\dindpi{\iota}{Q}{i}{}}(  \iuv{\Goij{\mbf{G}}{i}{j}}{\Goi{\mbf{G}}{i}}(v,w))={\dindpi{\iota}{Q}{j}{}} (\iuv{\Goij{\mbf{G}}{j}{i}}{\Goi{\mbf{G}}{j}}(v,w))={\dindpi{\iota}{Q}{j}{}} (w) = {\dindpi{\iota}{Q}{j}{}}(y),$$ from the injectivity of ${\dindpi{\iota}{Q}{j}{}}$, we deduce that $w=y$. Thus, we deduce that $(x, y) \in \Goij{\mbf{G}}{i}{j}$.

Therefore,
$${\dindiv{\iota}{Q}{\mbf{G}}{i}}(x) = {\dindiv{\iota}{Q}{\mbf{G}}{i}}(  \iuv{\Goij{\mbf{G}}{i}{j}}{\Goi{\mbf{G}}{i}}(x, y))={\dindiv{\iota}{Q}{\mbf{G}}{j}}(  \iuv{\Goij{\mbf{G}}{j}{i}}{\Goi{\mbf{G}}{j}}(x, y))={\dindiv{\iota}{Q}{\mbf{G}}{j}}(y).$$

This proves that $Q_{\mathbf{G}}$ is a glued-up object over $\mathbf{G}$.

\end{remark}

\bibliographystyle{abnt-num}

\begin{thebibliography}{1}
	
	\bibitem{adamek2009abstract}
	J.~Ad{\'a}mek, H.~Herrlich and G.~Strecker
	\newblock Abstract and Concrete Categories; The Joy of Cats.
	\newblock {\em Reprints in Theory and Applications of Categories}, 2009.
	
   \bibitem{Conrad}
	B.~Conrad.
	\newblock{Sheafification}. Retrieved from
	\newblock{\em 
		https://virtualmath1.stanford.edu/~conrad/  216APage/ handouts/sheafification.pdf}, 2019. Accessed January, 2022.
	
	\bibitem{Joji}
	G.~Janelidze.
	\newblock{Category theory: A first course.} Retrieved from
	\newblock{\em https://www.docdroid.net/d2VQ3wZ/} \newblock{\em ct2020-updated-20220519-pdf}, 2020. Accessed December, 2020.
	
	\bibitem{munkres}
	J.~Munkres.
	\newblock{ Topology, volume 2}.
	\newblock{\em Prentice Hall Upper Saddle River}, 2000.	
	
	\bibitem{maclane}
	S.~MacLane.
	\newblock{Categories for the working mathematician. 4th corrected printing}.
	\newblock{\em Graduate texts in mathematics}, 1988.
	
	\bibitem{curry}
	J.~Curry.
	\newblock{Sheaves, Cosheaves and Applications}.
	\newblock{\em University of Pennsylvania},  December 2014.
	
		\bibitem{grothendieck1971}
	A.~Grothendieck and J.~Dieudonn{\'e}.
	\newblock{El{\'e}ments de g{\'e}om{\'e}trie alg{\'e}brique, volume 166}.
	\newblock{\em Springer Berlin}, 1971.
	
		\bibitem{poincare2010}
	H.~Poincar{\'e}.
	\newblock{Papers on topology: analysis situs and its five supplements, volume 37}.
	\newblock{\em American Mathematical Society}, 2010.
	
		\bibitem{wedhorn2016}
	T.~Wedhorn.
	\newblock{Manifolds, sheaves, and cohomology}.
	\newblock{\em Springer}, 2016.
\end{thebibliography}

\end{document}